\documentclass[reqno]{amsart}

\usepackage[margin=1.4in]{geometry}
\usepackage{amssymb, amsmath, amsthm, amsfonts, amscd}
\usepackage{mathrsfs,mathtools}
\usepackage[shortlabels]{enumitem}
\usepackage{tikz, tikz-cd}
\usepackage[all]{xy}
\usepackage[new]{old-arrows}
\usepackage{pst-node}
\usepackage{varwidth}
\usepackage{graphicx}
\usepackage{subfig}
\usepackage{tabularx, booktabs}
\usepackage{multirow, multicol}
\usepackage{xcolor}
\usepackage{mathrsfs}
\usepackage{yfonts, euscript}
\usepackage{csquotes, dirtytalk}
\usepackage[pagewise]{lineno} 
\usepackage[colorlinks=true, allcolors=blue]{hyperref}

\usetikzlibrary{arrows, arrows.meta}

\newtheorem{thm}{{\bf Theorem}}[section]
\newtheorem{lemma}[thm]{{\bf Lemma}}

\newtheorem{cor}[thm]{{\bf Corollary}}
\newtheorem*{defn}{{\bf Definition}}

\newtheorem*{rmk}{{\bf Remark}}

\newcommand{\RNum}[1]{\uppercase\expandafter{\romannumeral #1\relax}}
\newcommand{\Aut}{\,\mathrm{Aut}\,}

\newcommand{\Hom}{\,\mathrm{Hom}\,}
\newcommand{\ad}{\,\mathrm{ad}\,}

\newcommand{\GL}{\,\mathrm{GL}\,}
\newcommand{\SL}{\,\mathrm{SL}\,}

\newcommand{\PGL}{\,\mathrm{PGL}\,}

\newcommand{\Rad}{\,\mathrm{Rad}\,}

\newcommand{\diag}{{\rm diag}}

\begin{document}


\title[Automorphisms of Twisted Chevalley Groups of type ${}^2 A_\ell$]{Automorphisms of Twisted Chevalley Groups \\ of type ${}^2 A_\ell \ (\ell \geq 5)$ over Local Rings}


\author{Elena I. Bunina}
\address{Department of Mathematics,
    	Bar–Ilan University, Ramat Gan, Israel}
\email{\href{mailto:helenbunina@gmail.com}{helenbunina@gmail.com}}
\thanks{}

\author{Deep H. Makadiya}
\address{Department of Mathematics,
    	Bar–Ilan University, Ramat Gan, Israel}
\email{\href{mailto:deepmakadia25.dm@gmail.com}{deepmakadia25.dm@gmail.com}}
\thanks{}




\begin{abstract}
    This paper is the first in a series devoted to the classification of automorphisms and isomorphisms of twisted Chevalley groups over commutative rings. In the present paper, we prove that every automorphism of a twisted Chevalley group of type ${}^2A_\ell$ $(\ell \geq 5)$ over a local ring in which $2$ is invertible is standard.
\end{abstract}


\maketitle 


\section{Introduction}

Automorphisms and isomorphisms of Chevalley groups over commutative rings have been extensively studied over the past decades. 
For untwisted Chevalley groups, starting from the classical work of Steinberg~\cite{RS1} and Humphreys~\cite{JH0} over fields, a comprehensive theory describing automorphisms has been developed and extended to various classes of rings. 
In particular, the structure of normal subgroups and the classification of automorphisms of Chevalley groups over various classes of commutative rings are now well understood, thanks to the contributions by Abe~\cite{EA0, EA2, EA3, EA4, EA&KS}, Suzuki~\cite{EA&KS}, Taddei~\cite{GT}, Vaserstein~\cite{LV}, Klyachko~\cite{AK}, Vavilov~\cite{NV1}, and many others. 
In recent years, Bunina and collaborators have obtained a series of results describing automorphisms of Chevalley groups over general commutative rings; see, for example,~\cite{EB24:final, EB12:main, EB07:first, EB08:b2andg2, EB09:alwithhalfele, EB10:alwithhalf, EB10:alwithouthalf, EB10:blwithhalf, EB10:f4withhalf, EB&PV14:g2withouthlfnor, EB&PV14:g2withouthalf, EB&MV24:g2with1/3}.

In contrast, the case of twisted Chevalley groups over rings is much less developed, even though the corresponding groups over fields are classical objects and their automorphisms are again standard in the sense of Steinberg. 
A systematic investigation in this direction has only recently been initiated by Garge and Makadiya~\cite{SG&DM1, SG&DM2, SG&DM3}. 
In these works, the authors described the normal subgroup structure of elementary twisted Chevalley groups, established congruence subgroup properties, and studied normalizers and factorization properties over arbitrary commutative rings.

The present paper initiates a series devoted to the classification of automorphisms and isomorphisms of twisted Chevalley groups over commutative rings. In this first part, we focus on adjoint twisted Chevalley groups of type ${}^2 A_\ell$ with $\ell \geq 5$, defined over local rings $R$ satisfying $1/2 \in R$. Our main result shows that, under these assumptions, all automorphisms are standard in an appropriate sense. 

Let $G_{\pi,\sigma}(\Phi,R)$ be a twisted Chevalley group obtained as the fixed point subgroup of the untwisted Chevalley group $G_{\pi}(\Phi,R)$ under an automorphism $\sigma=\rho\circ\theta$, where $\rho$ is a graph automorphism and $\theta$ is a ring automorphism of order $2$. 
Denote by $E'_{\pi,\sigma}(\Phi,R)$ the corresponding elementary twisted Chevalley group. 

More precisely, let $G=G_{\pi,\sigma}(A_\ell,R)$ (respectively, $E'_{\pi,\sigma}(A_\ell,R)$) be a twisted Chevalley group (respectively, an elementary twisted Chevalley group) of type ${}^2A_\ell$, where $\ell\ge 5$ and $R$ is a local ring satisfying $1/2\in R$. 
We show that every automorphism of $E'_{\pi,\sigma}(A_\ell,R)$ is a composition of an inner automorphism and a ring automorphism, whereas every automorphism of $G_{\pi,\sigma}(A_\ell,R)$ is a composition of an inner, a ring, and a central automorphism.

The proof proceeds in several steps. We first establish the result for elementary adjoint groups (see Theorem~\ref{MT_local1}). 
The key observation is that the description of normal subgroups obtained in~\cite{SG&DM1} makes it possible to analyze automorphisms modulo the Jacobson radical. After passing to the corresponding quotient group over the residue field, we reduce the problem to the classical field case and normalize the automorphism by conjugation.

The subsequent argument is of a matrix-theoretic nature. By performing a sequence of suitable basis changes, we arrange that the automorphism fixes distinguished generators of the group, namely the elements $h_{[\alpha]}(-1)$, $w_{[\alpha]}(1)$ and $x_{[\alpha]}(1)$. 
The final step consists in proving that an automorphism fixing all these elements must be induced by an automorphism of the underlying ring.

Having established the result for adjoint elementary groups, we extend it to groups corresponding to arbitrary weight lattices (see Theorem~\ref{MT_local2}). Finally, using the fact that the elementary twisted Chevalley group is characteristic in the full twisted Chevalley group, we derive the description of automorphisms of $G_{\pi,\sigma}(A_\ell,R)$ (see Theorem~\ref{MT_local3}).

The case ${}^2A_2$ is not considered here, since the normal-subgroup description required for our proof is not currently available. The remaining cases ${}^2A_\ell$ with $\ell=3,4$ are also excluded from the present work, since the normalization arguments used in this paper do not apply in these ranks. Consequently, these cases require a separate analysis and will be treated in subsequent papers.

Although the main results of this paper concern groups of type ${}^2 A_\ell \ (\ell \geq 5)$, we develop the necessary definitions and general framework in a broader setting. This allows the present paper to serve as a unified reference for subsequent parts of the series and avoids repetition. In forthcoming works, we shall consider the remaining twisted types ${}^2A_\ell$ $(\ell = 3,4)$, ${}^2D_\ell$ $(\ell \geq 4)$, ${}^2E_6$, and ${}^3D_4$.

The paper is organized as follows. 
In Section~\ref{sec:preliminaries}, we review the basic definitions and general properties of Chevalley groups and twisted Chevalley groups. 
In Section~\ref{sec:G(R)_and_MT}, we give explicit matrix realizations of $G_{\mathrm{sc},\sigma}(A_{\ell}, R)$ and $G_{\ad,\sigma}(A_{\ell}, R)$; more precisely, these groups can be realized as the fixed-point subgroups of $\SL_{\ell+1}(R)$ and $\PGL_{\ell+1}(R)$, respectively, under the automorphism $\sigma$.
In Section~\ref{sec:outline_of_main_thm}, we outline the proof of Theorem~\ref{MT_local1} and reduce it to several key steps. 
Sections~\ref{sec:image_of_h(-1)}--\ref{sec:image_of_x(t)} are devoted to the proofs of these steps. 
Finally, in Section~\ref{sec:proof_of_MT2_and_MT3}, we prove Theorem~\ref{MT_local2} and Theorem~\ref{MT_local3}.


\section{Chevalley and Twisted Chevalley Groups} \label{sec:preliminaries}


In this section, we provide the formal definitions of Chevalley groups and twisted Chevalley groups. 
For comprehensive treatments of these topics, we refer the reader to \cite{EA1}, \cite{EB24:final}, \cite{RC}, \cite{SG&DM1}, \cite{SG&DM2}, \cite{SG&DM3}, \cite{JH}, \cite{EP&NV}, \cite{RS}, and \cite{NV1}. 
We primarily follow the notation established in \cite{SG&DM1}.


\subsection{Chevalley groups}\label{subsec:CG}

Throughout the paper, let $\mathcal{L}$ be a finite-dimensional semisimple Lie algebra over $\mathbb{C}$. 
We fix a Cartan subalgebra $\mathcal{H} \subseteq \mathcal{L}$ and denote by $\Phi \subseteq \mathcal{H}^*$ the associated root system. 
The Lie algebra then decomposes as
\[
    \mathcal{L}
    =
    \mathcal{H}
    \oplus
    \bigoplus_{\alpha \in \Phi} \mathcal{L}_{\alpha},
\text{  where }
    \mathcal{L}_\alpha
    =
    \{\, X \in \mathcal{L}
    \mid
    [H, X] = \alpha(H) X
    \text{ for all } H \in \mathcal{H} \,\}.
\]
It is well known that each root space is one-dimensional, i.e.\ $\dim \mathcal{L}_\alpha = 1$ for all $\alpha \in \Phi$.  
Consequently, $\dim \mathcal{L} = \dim \mathcal{H} + |\Phi|$, and the integer $\ell := \dim \mathcal{H}$ is called the \emph{rank} of~$\mathcal{L}$.

Fix a set of simple roots $\Delta \subseteq \Phi$, and let $\Phi^+$ denote the corresponding set of positive roots.  
For each $\alpha \in \Phi$, there exists a unique element
\[
H_\alpha \in [\mathcal{L}_\alpha, \mathcal{L}_{-\alpha}]
\subseteq \mathcal{H}
\]
such that $\alpha(H_\alpha) = 2$.  
For each $\alpha \in \Phi^+$, choose a nonzero element
$X_\alpha \in \mathcal{L}_\alpha$; then there exists a unique element
$X_{-\alpha} \in \mathcal{L}_{-\alpha}$ satisfying
$H_\alpha = [X_\alpha, X_{-\alpha}]$.
It follows immediately that $H_{-\alpha} = -H_\alpha$ for all $\alpha \in \Phi$.

The set
\[
\{\, H_i (= H_{\alpha_i}), X_\alpha
\mid
\alpha_i \in \Delta,\ \alpha \in \Phi \,\}
\]
forms a basis of $\mathcal{L}$.  
Moreover, this basis may be chosen so that the Lie bracket of any two basis elements is an integral linear combination of basis elements; such a basis is called a \textbf{Chevalley basis} (see, for example, \cite[p.\,147]{JH} or \cite[p.\,7]{RS}).  
In what follows, we fix such a Chevalley basis.


\subsubsection{} 

Consider an $n$-dimensional faithful representation
$\pi : \mathcal{L} \longrightarrow \mathfrak{gl}(V)$
of the Lie algebra $\mathcal{L}$. 
For any $\lambda \in \mathcal{H}^*$, define
\[
    V_\lambda
    =
    \{\, v \in V
    \mid
    \pi(H)v = \lambda(H)v
    \text{ for all } H \in \mathcal{H} \,\}.
\]
If $V_\lambda \neq 0$, then $V_\lambda$ is called a \textbf{weight space}, and the corresponding $\lambda$ is called a \textbf{weight} of the representation $\pi$.  
Any nonzero vector $v \in V_\lambda$ is called a \textbf{weight vector} of weight $\lambda$.  
Let $\Omega \subseteq \mathcal{H}^*$ denote the set of all weights of $\pi$.  
Then $V$ decomposes into its weight spaces:
\[
    V = \bigoplus_{\lambda \in \Omega} V_\lambda.
\]

Now let $\mathcal{U} = \mathcal{U}(\mathcal{L})$ denote the universal enveloping algebra of $\mathcal{L}$, and let $\mathcal{U}_{\mathbb{Z}}$ be the corresponding \textbf{Kostant $\mathbb{Z}$-form}, generated by the elements $X_\alpha^{\,m}/m!$ for $m \in \mathbb{Z}_{\ge 0}$ and $\alpha \in \Phi$.  
The representation $\pi$ extends naturally to an action of $\mathcal{U}$, and hence also of $\mathcal{U}_{\mathbb{Z}}$, on $V$.
The module $V$ contains a lattice $M$ that is invariant under the action of $\mathcal{U}_{\mathbb{Z}}$; such a lattice is called an \textbf{admissible lattice}.  


\subsubsection{}

Let $R$ be a commutative ring with unity. 
Given an admissible lattice $M \subseteq V$, define
\[
V(R) = M \otimes_{\mathbb{Z}} R.
\]
Then $V(R)$ is a free $R$-module of rank $n$. For each $\alpha \in \Phi$ and $t \in R$, consider the automorphism of $V(R)$ defined by
\[
x_\alpha(t) := \exp\!\bigl(t\,\pi(X_\alpha)\bigr),
\text{ where }
    \exp\!\bigl(t\,\pi(X_\alpha)\bigr)
        =
        \sum_{m=0}^{\infty}
        \frac{t^m\,\pi(X_\alpha)^m}{m!}.
\]
Since the Kostant $\mathbb{Z}$-form $\mathcal{U}_{\mathbb{Z}}$ acts on $M$ and each $X_\alpha$ is nilpotent, the above exponentials are well defined on $V(R)$ (see \cite[Chapter~3]{RS} for more details).

The subgroup of $\operatorname{Aut}(V(R))$ generated by all $x_\alpha(t)$, with $t \in R$ and $\alpha \in \Phi$, is called an \textbf{elementary Chevalley group} and is denoted by $E_\pi(\Phi, R)$.  
It is known that $E_\pi(\Phi, R)$ is independent of the choice of admissible lattice $M$.

For a representation $\pi$, let $\Lambda_\pi$ denote the weight lattice of $\pi$, i.e.\ the lattice generated by all weights of $\pi$.  
If $\pi$ and $\pi'$ are representations of $\mathcal{L}$ such that
$\Lambda_\pi = \Lambda_{\pi'}$, then
\[
E_\pi(\Phi, R) \cong E_{\pi'}(\Phi, R).
\]

Let $\Lambda_r$ be the root lattice and $\Lambda_{sc}$ the lattice generated by the fundamental weights.  
If $\pi$ satisfies $\Lambda_\pi = \Lambda_r$ (resp.\ $\Lambda_\pi = \Lambda_{sc}$), then the associated elementary Chevalley group $E_\pi(\Phi, R)$ is called the \textbf{adjoint elementary Chevalley group} (resp.\ the \textbf{universal} or \textbf{simply connected elementary Chevalley group}), and is denoted by $E_{\mathrm{ad}}(\Phi, R)$ (resp.\ $E_{sc}(\Phi, R)$).

Let $U = U_\pi(\Phi, R)$ (resp.\ $U^{-} = U^{-}_\pi(\Phi, R)$) denote the subgroup of $E_\pi(\Phi, R)$ generated by all $x_\alpha(t)$ with $\alpha \in \Phi^{+}$ (resp.\ $\alpha \in \Phi^{-}$) and $t \in R$.  
Let $H = H_\pi(\Phi, R)$ be the subgroup generated by all elements
\[
    h_\alpha(t)
    =
    w_\alpha(t) w_\alpha(1)^{-1},
    \qquad\text{where }
    w_\alpha(t)
    =
    x_\alpha(t)\,
    x_{-\alpha}(-t^{-1})\,
    x_\alpha(t),
    \quad t \in R^{\times}.
\]
If $B = B_\pi(\Phi, R)$ is the subgroup generated by $U$ and $H$, then $U \cap H = 1$, $U$ is normal in $B$, and $B = UH$.
Let $N = N_\pi(\Phi, R)$ be the subgroup generated by all $w_\alpha(t)$, and let $W = W(\Phi)$ be the Weyl group. Then $H$ is normal in $N$, and
$W \cong N/H$
via the correspondence
$s_\alpha \longmapsto H w_\alpha(1)$,
for $\alpha \in \Phi$.


\subsubsection{}

Let $k$ be an algebraically closed field. Then every connected semisimple linear algebraic group over $k$ is, up to isomorphism, an elementary Chevalley group $E_\pi(\Phi, k)$ (see \cite[Chapter~5]{RS}). All these groups can be viewed as subgroups of $GL_n(k)$ defined by a common set of zeros of polynomials in the matrix entries $x_{ij}$ with integer coefficients. Note that the multiplication map and the inverse map are also defined by polynomials with integer coefficients. Therefore, these polynomials can be regarded as polynomials over an arbitrary commutative ring with unity. 


\subsubsection{}

Let $E_\pi(\Phi, \mathbb{C})$ be an elementary Chevalley group viewed as a subgroup of $GL_n(\mathbb{C})$ defined as the zero locus of polynomials
$p_1(x_{ij}), \dots, p_m(x_{ij})$.

These polynomials can be chosen to have integer coefficients. Let $R$ be a commutative ring with unity and consider the group
\[
G(R)
=
\left\{
(a_{ij}) \in GL_n(R)
\ \middle|\
\widetilde{p}_1(a_{ij}) = 0,\dots,
\widetilde{p}_m(a_{ij}) = 0
\right\},
\]
where
\[
\widetilde{p}_1(x_{ij}), \dots, \widetilde{p}_m(x_{ij})
\]
are the polynomials with the same integer coefficients as
$p_1(x_{ij}), \dots, p_m(x_{ij})$, regarded as polynomials over $R$. This group is called the \textbf{Chevalley group} $G_\pi(\Phi, R)$ of type $\Phi$ over the ring~$R$.

If $\pi$ is a representation such that $\Lambda_{\pi} = \Lambda_r$, then
\[
G_{\pi}(\Phi, R) = G_{\mathrm{ad}}(\Phi, R)
\]
is called an \textbf{adjoint Chevalley group}. If $\pi$ is a representation such that $\Lambda_\pi = \Lambda_{sc}$, then
\[
G_\pi(\Phi, R) = G_{sc}(\Phi, R)
\]
is called a \textbf{universal} (or \textbf{simply connected}) \textbf{Chevalley group}.

Note that
$E_\pi(\Phi, R) \subseteq G_\pi(\Phi, R)$.

If $k$ is an algebraically closed field, then
\[
E_\pi(\Phi, k) = G_\pi(\Phi, k).
\]
In general, however, equality need not hold, even over a field. 


\subsubsection{}

The subgroup of diagonal matrices, with respect to the standard basis of weight vectors, in the Chevalley group $G_\pi(\Phi, R)$ is called the \textbf{standard maximal torus} of $G_\pi(\Phi, R)$ and is denoted by $T_\pi(\Phi, R)$.  
This group is isomorphic to
$\mathrm{Hom}(\Lambda_\pi, R^\times)$,
where $R^{\times}$ denotes the group of units of $R$. The isomorphism is defined as follows.

Let
\[
\chi \in \mathrm{Hom}(\Lambda_\pi, R^{\times})
\]
be a character of the weight lattice $\Lambda_\pi$.  
For a weight space $V_\mu$ of weight $\mu$ in $\pi$, let
\[
V_\mu(R)
=
(V_\mu \cap M) \otimes_\mathbb{Z} R
\subseteq V(R).
\]
Define an automorphism $h(\chi)$ of $V(R)$ by
\[
    h(\chi)\cdot v = \chi(\mu)\,v,
    \qquad v \in V_\mu(R).
\]  
Since
\[
V(R)
=
\bigoplus_{\mu \in \Omega_\pi} V_\mu(R),
\]
where $\Omega_\pi$ is the set of weights of $\pi$, this defines $h(\chi)$ on all of $V(R)$.  
Hence,
\[
    T_\pi(\Phi, R)
    =
    \{\, h(\chi)
    \mid
    \chi \in \mathrm{Hom}(\Lambda_\pi, R^{\times}) \,\}.
\]

The element $h(\chi)$ acts on the root subgroup
\[
\mathfrak{X}_\alpha
=
\{\,x_\alpha(t)\mid t\in R\,\}
\]
by conjugation as follows:
\[
    h(\chi)\,x_\alpha(\zeta)\,h(\chi)^{-1}
    =
    x_\alpha\bigl(\chi(\alpha)\,\zeta\bigr),
    \qquad
    \alpha \in \Phi,\quad \zeta \in R.
\]

Observe that
\[
H_\pi(\Phi, R) \subseteq T_\pi(\Phi, R).
\]
An element $h(\chi) \in T_\pi(\Phi, R)$ belongs to $H_\pi(\Phi, R)$ if and only if the character
\[
\chi \in \mathrm{Hom}(\Lambda_\pi, R^{\times})
\]
can be extended to a character
\[
    \chi' \in \mathrm{Hom}(\Lambda_{sc}, R^{\times})
    \qquad\text{such that}\qquad
    \chi'|_{\Lambda_\pi} = \chi.
\]
In particular, for $t \in R^{\times}$ and $\alpha \in \Phi$, we have
\[
    h_\alpha(t) = h(\chi_{\alpha,t}),
    \qquad
    \chi_{\alpha,t}(\lambda)
    =
    t^{\langle \lambda,\alpha^\vee\rangle},
    \quad \lambda \in \Lambda_\pi.
\]
Therefore,
\[
H_\pi(\Phi, R)
=
E_\pi(\Phi, R)\cap T_\pi(\Phi, R).
\]
Moreover, in the simply connected case one has
\[
H_{sc}(\Phi, R)=T_{sc}(\Phi, R).
\]


\subsubsection{}

We now record several standard structural properties of Chevalley groups.

A subgroup $H$ of a group $G$ is called \emph{characteristic} if it is mapped into itself under every automorphism of $G$.  
In particular, every characteristic subgroup is normal.  
Assume that $\operatorname{rank}(\Phi)\geq 2$. If $\Phi$ is of type $B_2$ or $G_2$, assume in addition that $2\in R^{\times}$. Then $E_\pi(\Phi, R)$ is a characteristic subgroup of $G_\pi(\Phi, R)$ (see \cite[Theorem~5]{LV}).

A group $G$ is said to be \emph{perfect} if $[G,G]=G$, where $[G,G]$ denotes the commutator subgroup of $G$.  
Assume that $\operatorname{rank}(\Phi)\geq 2$. If $\Phi$ is of type $B_2$ or $G_2$, assume in addition that $2\in R^{\times}$. Then the elementary Chevalley group $E_\pi(\Phi, R)$ is perfect (see \cite[Theorem~5]{LV}).  

Consider the subgroup
\[
G_{\pi}^{0}(\Phi, R)
=
E_\pi(\Phi, R)\,T_\pi(\Phi, R)
\subseteq G_\pi(\Phi, R).
\]
If $R$ is a semilocal ring, then
\[
G_\pi(\Phi, R)=G_\pi^{0}(\Phi, R)
\]
(see \cite[Corollary~2.4]{EA2}).
In particular, in the simply connected case, if $R$ is a semilocal ring, then
\[
G_{sc}(\Phi, R)=E_{sc}(\Phi, R).
\]

For an ideal $I$ of $R$, the natural projection map
$R \longrightarrow R/I$
induces a group homomorphism
\[
    \lambda_I:
    G_\pi(\Phi, R)
    \longrightarrow
    G_\pi(\Phi, R/I).
\]  
Define
\[
    G_\pi(\Phi, I)
    \coloneq
    \ker(\lambda_I)
\]
and
\[
    G_\pi(\Phi, R, I)
    \coloneq
    \lambda_I^{-1}
    \bigl(
    Z(G_\pi(\Phi, R/I))
    \bigr),
\]
where $Z(G_\pi(\Phi, R/I))$ denotes the center of $G_\pi(\Phi, R/I)$.  
The subgroup $G_\pi(\Phi, I)$ of $G_\pi(\Phi, R)$ is referred to as the \emph{principal congruence subgroup} of level $I$, while $G_\pi(\Phi, R, I)$ is called the \emph{full congruence subgroup} of level $I$.

Let $E_\pi(\Phi, I)$ denote the subgroup of $E_\pi(\Phi, R)$ generated by $x_\alpha(t)$ for all $\alpha \in \Phi$ and $t \in I$.  
Additionally, define $E_\pi(\Phi, R, I)$ as the normal subgroup of $E_\pi(\Phi, R)$ generated by $E_\pi(\Phi, I)$. The subgroup $E_\pi(\Phi, R, I)$ is referred to as the \emph{relative elementary subgroup} of level $I$. 

\begin{thm}[{L. N. Vaserstein \cite{LV}}, E. Abe \cite{EA3}]
\label{thm:LV&EA}
Let $\Phi$ be an irreducible root system of rank at least $2$, and let $R$ be a commutative ring with unity. Assume that $1/2\in R$ if
\[
\Phi\sim B_\ell\ (\ell\geq 2),\quad
C_\ell\ (\ell\geq 3),\quad
F_4,
\]
and that $1/3\in R$ and $R$ has no factor ring with two elements if $\Phi\sim G_2$. If $H$ is a subgroup of $G_\pi(\Phi, R)$ normalized by $E_\pi(\Phi, R)$, then there exists a unique ideal $J$ of $R$ such that
\[
E_\pi(\Phi, R, J)
\subseteq
H
\subseteq
G_\pi(\Phi, R, J).
\]
\end{thm}

Finally, we recall the result on automorphisms of Chevalley groups.
For untwisted Chevalley groups over commutative rings, the automorphism problem has by now been completely resolved in full generality.
We adopt the standard terminology: an automorphism of a Chevalley group is called \emph{standard} if it can be written as a composition of ring, inner, central, and graph automorphisms.

\begin{thm}[E. I. Bunina \cite{EB24:final}]
\label{thm:Bunina-untwisted}
Let $R$ be a commutative ring and let $\Phi$ be an irreducible root system of rank greater than $1$.  
Assume that $1/2\in R$ for the types
$A_2$,
$B_\ell$ ($\ell\geq 2$),
$C_\ell$ ($\ell\geq 3$),
$F_4$,
and that $1/2,1/3\in R$ for the type $G_2$.  
Then every automorphism of $G_{\pi}(\Phi, R)$ and of $E_{\pi}(\Phi, R)$ is standard.
\end{thm}


\subsection{Twisted root system}\label{Subsec:TRS}

Let $V$ be a finite-dimensional real Euclidean vector space, and let $\Phi$ be a crystallographic root system. Let $\Delta$ and $\Phi^+$ be the sets of simple and positive roots, respectively, with respect to some fixed ordering on $V$. Let $\rho$ be a nontrivial angle-preserving permutation of $\Delta$ (such a $\rho$ exists only when $\Phi$ is of type $A_\ell \ (\ell \geq 2)$, $D_\ell \ (\ell \geq 4)$, $E_6$, $B_2$, $F_4$, or $G_2$). The order of $\rho$ is either $2$ or $3$, with the latter possible only when $\Phi$ is of type $D_4$. We define an isometry $\hat{\rho} \in GL(V)$ as follows:
\begin{enumerate}
    \item If $\Phi$ has one root length, define $\hat{\rho}(\alpha)=\rho(\alpha)$ for each $\alpha \in \Delta$.
    \item If $\Phi$ has two root lengths, define $\hat{\rho}(\alpha)=\rho(\alpha)/\sqrt{p}$ for each short root $\alpha \in \Delta$ and $\hat{\rho}(\alpha)=\sqrt{p}\,\rho(\alpha)$ for each long root $\alpha \in \Delta$, where $p=\lVert\alpha\rVert^2/\lVert\beta\rVert^2$, with $\alpha$ long and $\beta$ short. 
\end{enumerate} 

Clearly, the order of $\hat{\rho}$ is the same as that of $\rho$, and $\hat{\rho}$ preserves signs. Note that $\hat{\rho} w_{\alpha} \hat{\rho}^{-1} = w_{\rho(\alpha)}$, hence $\hat{\rho}$ normalizes $W$. Define $V_\rho = \{ v \in V \mid \hat{\rho}(v)=v \}$ and $W_\rho = \{ w \in W \mid \hat{\rho}w\hat{\rho}^{-1} =w \}$. 
It is well known that the natural action of $W_\rho$ on $V_\rho$ is faithful.

For any $\alpha \in V$, define
\[
    \hat{\alpha}
    = \frac{1}{o(\rho)} \sum_{i=0}^{o(\rho)-1} \hat{\rho}^{\,i}(\alpha),
\]
the average of the elements in the $\hat{\rho}$-orbit of $\alpha$. 
Then $\hat{\alpha} \in V_\rho$ and $(\beta, \hat{\alpha}) = (\beta, \alpha)$ for every $\beta \in V_\rho$.
Consequently, the orthogonal projection of $\alpha$ onto $V_\rho$ is precisely $\hat{\alpha}$.

Let $(\alpha) \subseteq \Phi$ denote the $\rho$-orbit of $\alpha$, and let $W_{(\alpha)}$ be the subgroup generated by all reflections $w_\beta \ (\beta \in (\alpha))$.
Let $P_\alpha$ be the positive system of the root subsystem generated by $(\alpha)$, and let $w_{(\alpha)}$ be the unique element of $W_{(\alpha)}$ such that $w_{(\alpha)}(P_\alpha)=-P_\alpha$; equivalently, $w_{(\alpha)}$ is the longest element of $W_{(\alpha)}$.
Then $w_{(\alpha)}|_{V_\rho} = w_{\hat{\alpha}}|_{V_\rho}$ and $w_{(\alpha)}|_{V_\rho} \in W_{\rho}$. 
In fact, the set $\{ w_{\hat{\alpha}}|_{V_\rho} \mid \alpha \in \Delta \}$ forms a generating set of $W_\rho$. 
Therefore, the group $W_\rho|_{V_\rho}$ is a reflection group.

Define
\[
\widetilde{\Phi}_\rho := \{ \hat{\alpha} \mid \alpha \in \Phi \},
\qquad
\widetilde{\Delta}_\rho := \{ \hat{\alpha} \mid \alpha \in \Delta \}.
\]
Then $\widetilde{\Phi}_\rho$ is the (possibly non-reduced) root system corresponding
to the Weyl group $W_\rho|_{V_\rho}$, and $\widetilde{\Delta}_\rho$ is the associated
simple system. 
By selecting, in each direction, the shortest nonzero vector in
$\widetilde{\Phi}_\rho$, we obtain a reduced root system, which we denote by
$\Phi_\rho$.

Define two equivalence relations $\equiv_1$ and $\equiv_2$ on $\Phi$ by declaring
\[
    \alpha \equiv_1 \beta \;(\text{resp., } \alpha \equiv_2 \beta)
    \quad \Longleftrightarrow \quad
    \hat{\alpha} = \hat{\beta}
    \;(\text{resp., } \hat{\alpha} \text{ is a positive scalar multiple of } \hat{\beta}).
\]
Let $(\alpha)$ (resp., $[\alpha]$) denote the equivalence class of $\alpha$ with
respect to $\equiv_1$ (resp., $\equiv_2$), and let $\Phi_1$ (resp., $\Phi_2$) be
the set of all equivalence classes of $\equiv_1$ (resp., $\equiv_2$).
Then $\widetilde{\Phi}_\rho$ is in natural bijection with $\Phi_1$, where a root
$\hat{\alpha} \in \widetilde{\Phi}_\rho$ is identified with the class
$(\alpha) \in \Phi_1$. Similarly, the reduced root system $\Phi_\rho$ is in
natural bijection with $\Phi_2$.

Using these identifications, we shall henceforth write
\[
    \widetilde{\Phi}_\rho := \{ (\alpha) \mid \alpha \in \Phi \},
    \qquad
    \Phi_\rho := \{ [\alpha] \mid \alpha \in \Phi \},
\]
and refer to $\widetilde{\Phi}_\rho$ as the \emph{non-reduced twisted root system}
and to $\Phi_\rho$ as the \emph{(reduced) twisted root system} associated with
$\Phi$. In what follows, we shall mostly work with the (reduced) twisted root
system $\Phi_\rho$.

It is immediate that
$-[\alpha] = [-\alpha]$
and 
$-(\alpha) = (-\alpha)$.

\begin{lemma}[{\cite[page 103]{RS}}]
    \normalfont
    If $\Phi$ is irreducible, then every class $[\alpha]\in\Phi_\rho$, viewed as a subset of $\Phi$, is the set of positive roots of a root subsystem of one of the following types:
    \begin{enumerate}[(a)]
        \item $A_1^n$, where $n=1,2$, or $3$;
        \item $A_2$ (this occurs only if $\Phi$ is of type $A_{2n}$);
        \item $C_2$ (this occurs if $\Phi$ is of type $C_2$ or $F_4$);
        \item $G_2$ (this occurs only if $\Phi$ is of type $G_2$).
    \end{enumerate}
\end{lemma}

If a class $[\alpha]$ in $\Phi_\rho$ is the set of positive roots of a root subsystem of type $X$ (where $X$ is one of the root systems listed above), then we write $[\alpha]\sim X$. If $\Phi$ is of type $X$ and $\rho$ has order $n$, we say that the pair $(\Phi,\rho)$ has twisted type ${}^nX$. 

In the following table we describe twisted root systems $\Phi_\rho$ and the non-reduced systems $\widetilde{\Phi}_\rho$ corresponding to some root systems $\Phi$:

\begin{center}
    \begin{tabular}{| >{\centering\arraybackslash}m{2.5cm} | >{\centering\arraybackslash}m{1.5cm} | >{\centering\arraybackslash}m{1.5cm} | >{\centering\arraybackslash}m{1.8cm} | >{\centering\arraybackslash}m{1.8cm} |}
        \hline
        \multirow{2}{*}{\textbf{Type}} & \multirow{2}{*}{\textbf{$\widetilde{\Phi}_\rho$}} & \multirow{2}{*}{\textbf{$\Phi_\rho$}} & \multicolumn{2}{c|}{\textbf{Type of Roots in $\Phi_\rho$}} \\
        \cline{4-5}
        & & & \textbf{Long} & \textbf{Short}  \\
        \hline 
        ${}^2 A_{2\ell-1} \ (\ell \geq 2)$ & $C_\ell$ & $C_\ell$ & $A_1$ & $A_1^2$ \\
        \hline 
        ${}^2 A_{2\ell} \ (\ell \geq 2)$ & $BC_\ell$ & $B_\ell$ & $A_1^2$ & $A_2$ \\
        \hline 
        ${}^2 D_{\ell} \ (\ell \geq 4)$ & $B_{\ell-1}$ & $B_{\ell-1}$ & $A_1$ & $A_1^2$ \\
        \hline
        ${}^3 D_{4}$ & $G_2$ & $G_2$ & $A_1$ & $A_1^3$ \\
        \hline
        ${}^2 E_{6}$ & $F_4$ & $F_4$ & $A_1$ & $A_1^2$ \\
        \hline
    \end{tabular}
\end{center}


\subsubsection{}

Finally, let us discuss the action of $\rho$ on the weight lattice $\Lambda_{sc}$. Assume that $\Phi$ has one root length. Since $\rho$ permutes the simple roots, and hence all roots, it acts naturally on the root lattice $\Lambda_r$. The fundamental dominant weights $\lambda_1,\dots,\lambda_\ell$ form a $\mathbb{Z}$-basis of $\Lambda_{sc}$. We define $\rho(\lambda_i)=\lambda_j$ whenever $\rho(\alpha_i)=\alpha_j$ and extend this action $\mathbb{Z}$-linearly to $\Lambda_{sc}$. Then $\rho(\Lambda_r)=\Lambda_r$, so $\rho$ induces an automorphism of the fundamental group $\Lambda_{sc}/\Lambda_r$.

Let $\Lambda$ be a lattice satisfying $\Lambda_r\subseteq\Lambda\subseteq\Lambda_{sc}$. If $\Lambda_{sc}/\Lambda_r$ is cyclic, then every subgroup of this quotient is invariant under $\rho$; hence $\rho(\Lambda)=\Lambda$. The only noncyclic case is $\Phi=D_{2n}$, for which
\[
    \Lambda_{sc}/\Lambda_r\cong(\mathbb{Z}/2\mathbb{Z})^2.
\]
The three subgroups of order $2$ in this quotient correspond to the three intermediate lattices
\[
    \Lambda_r\subsetneq\Lambda\subsetneq\Lambda_{sc}
\]
of index $2$. If $\rho$ has order $2$, its action on $\Lambda_{sc}/\Lambda_r$ fixes one nonzero element and interchanges the other two. Consequently, exactly two of the three intermediate lattices are not $\rho$-stable. If $\Phi=D_4$ and $\rho$ is the triality automorphism of order $3$, then $\rho$ cyclically permutes the three nonzero elements of $\Lambda_{sc}/\Lambda_r$; consequently, none of the three intermediate lattices is $\rho$-stable.

Thus, the graph automorphism of $G_\pi(\Phi,R)$ and $E_\pi(\Phi,R)$ is defined precisely when
$\rho(\Lambda_\pi)=\Lambda_\pi$.
See \cite[page~91]{RS}.


\subsection{Twisted Chevalley groups}

Let $\Phi$ be a root system of type $A_\ell$ ($\ell \geq 2$), $D_\ell$ ($\ell \geq 4$), or $E_6$.
Consider the Chevalley group $G_\pi(\Phi, R)$ over a commutative ring $R$, and let $E_\pi(\Phi, R)$ denote its elementary subgroup. Let $\sigma$ be an automorphism of $G_\pi(\Phi, R)$ given by the product of a graph automorphism $\rho$ and a ring automorphism $\theta$ satisfying $o(\rho) = o(\theta)$. We use the same symbol $\rho$ to denote the induced permutation on the set of roots. 
Since $\rho \circ \theta = \theta \circ \rho$, we have $o(\theta) = o(\rho) = o(\sigma)$. 
Further, since $E_\pi(\Phi, R)$ is a characteristic subgroup of $G_\pi(\Phi, R)$, $\sigma$ is also an automorphism of $E_\pi(\Phi, R)$. 

Define $G_{\pi,\sigma}(\Phi,R)=\{g\in G_\pi(\Phi,R)\mid \sigma(g)=g\}$. Clearly, $G_{\pi, \sigma} (\Phi, R)$ is a subgroup of $G_\pi(\Phi, R)$. We call $G_{\pi, \sigma} (\Phi, R)$ the \textbf{twisted Chevalley group} over the ring $R$. The notion of the adjoint twisted Chevalley group and the universal (or simply connected) twisted Chevalley group is clear.

Write $E_{\pi, \sigma} (\Phi, R) = E_\pi(\Phi, R) \cap G_{\pi, \sigma} (\Phi, R)$. Consider the subgroups $U = U_\pi(\Phi, R)$, $U^- = U^{-}_\pi(\Phi, R)$, $H = H_\pi(\Phi, R)$, $B = B_\pi(\Phi, R)$ and $N = N_\pi(\Phi, R)$ of $E_\pi(\Phi, R)$. Then $\sigma$ preserves $U, U^-, H, B$ and $N$. Hence we can make sense of $U_\sigma = U_{\pi, \sigma} (\Phi, R)$, $U^-_\sigma = U^{-}_{\pi, \sigma} (\Phi, R)$, $H_\sigma = H_{\pi, \sigma} (\Phi, R)$, $B_\sigma = B_{\pi, \sigma} (\Phi, R)$ and $N_\sigma = N_{\pi, \sigma} (\Phi, R)$ (if $A \subseteq G_\pi(\Phi,R)$ with $\sigma(A)\subseteq A$ then we define $A_\sigma = A \cap G_{\pi, \sigma}(\Phi, R)$). Note that $\sigma$ preserves $N/H \cong W$ (as it preserves $N$ and $H$). The induced action on $W$ is consistent with the permutation $\rho$ of the roots. Finally, let us define $E'_{\pi, \sigma} (\Phi, R) = \langle U_\sigma, U_\sigma^- \rangle$, a subgroup of $E_{\pi, \sigma} (\Phi, R)$ generated by $U_\sigma$ and $U_\sigma^-$. We call $E'_{\pi, \sigma} (\Phi, R)$ the \textbf{elementary twisted Chevalley group} over the ring $R$. Write $H'_\sigma = H'_{\pi, \sigma} (\Phi, R) = H \cap E'_{\pi, \sigma}(\Phi, R)$, $N'_\sigma = N'_{\pi,\sigma}(\Phi,R) = N \cap E'_{\pi,\sigma}(\Phi,R)$ and $B'_\sigma = B'_{\pi,\sigma}(\Phi,R) = B \cap E'_{\pi,\sigma}(\Phi,R)$. Then $B'_\sigma = U_\sigma H'_\sigma$. 

If $G_\pi(\Phi, R)$ is of type $X$ and $\sigma$ is of order $n$, we say $G_{\pi, \sigma}(\Phi, R)$ is of type ${}^nX$. We write $G_{\pi}(\Phi, R) \sim X$ and $G_{\pi, \sigma}(\Phi, R) \sim {}^nX$. We use a similar notation for $E_\pi(\Phi, R)$, $E_{\pi, \sigma}(\Phi, R)$ and $E'_{\pi, \sigma}(\Phi, R)$. 


\subsubsection{}
Let $T_\pi (\Phi, R)$ be the standard maximal torus of $G_\pi(\Phi, R)$. Define $T_{\pi,\sigma}(\Phi,R)=T_\pi(\Phi,R)\cap G_{\pi,\sigma}(\Phi,R)$ and call it the \textbf{standard maximal torus} of $G_{\pi,\sigma}(\Phi,R)$. 
For a character $\chi \in \mathrm{Hom}(\Lambda_\pi,R^{\times})$, we define its conjugation character $\overline{\chi}_\sigma \in \mathrm{Hom}(\Lambda_\pi,R^{\times})$ with respect to $\sigma$ by $\overline{\chi}_\sigma (\lambda) = \theta (\chi (\rho^{-1}(\lambda)))$ for all $\lambda \in \Lambda_\pi$. 
If $h(\chi) \in T_\pi(\Phi, R)$, then $\sigma (h(\chi)) = h(\overline{\chi}_\sigma)$.
A character $\chi \in \mathrm{Hom}(\Lambda_\pi,R^{\times})$ is called \textbf{self-conjugate (with respect to $\sigma$)} if $\chi = \overline{\chi}_\sigma$, i.e., $\chi (\rho(\lambda)) = \theta(\chi(\lambda))$ for every $\lambda \in \Lambda_\pi$. 
Let $\mathrm{Hom}_1 (\Lambda_\pi,R^{\times}) = \{ \chi \in \mathrm{Hom}(\Lambda_\pi,R^{\times}) \mid \chi = \overline{\chi}_\sigma \}$ denote the subgroup of self-conjugate characters.
Then we have $T_{\pi,\sigma} (\Phi, R) = \{ h(\chi) \mid \chi \in \mathrm{Hom}_1(\Lambda_\pi,R^{\times}) \}$. 
Moreover, an element $h(\chi) \in T_{\pi,\sigma}(\Phi,R)$ lies in $H_{\pi,\sigma}(\Phi,R)$ if and only if the self-conjugate character \(\chi\) extends to a self-conjugate character $\chi' \in \mathrm{Hom}(\Lambda_{sc},R^{\times})$ of the simply connected weight lattice $\Lambda_{sc}$.


\subsubsection{}

Assume that $2\in R^{\times}$ and, in addition, that $3\in R^{\times}$ when $(\Phi,\rho)$ has twisted type ${}^3D_4$.
The elementary subgroup $E'_{\pi,\sigma} (\Phi, R)$ is a characteristic subgroup of the twisted Chevalley group $G_{\pi, \sigma} (\Phi, R)$ (see \cite[Corollary~A.2]{SG&DM1}). 
Hence every automorphism of $G_{\pi,\sigma}(\Phi, R)$ preserves $E'_{\pi,\sigma}(\Phi, R)$.
  
Moreover, $E'_{\pi, \sigma} (\Phi, R)$ is perfect, that is, $[E'_{\pi, \sigma} (\Phi, R), E'_{\pi, \sigma} (\Phi, R)] = E'_{\pi, \sigma} (\Phi, R)$ (see \cite[Corollary~6.6]{SG&DM1}).  

Let us also define 
\[
    G_{\pi,\sigma}^0(\Phi, R) = G_\pi^0(\Phi, R) \cap G_{\pi,\sigma}(\Phi, R) 
    \quad \text{and} \quad
    G'_{\pi,\sigma}(\Phi, R) = T_{\pi,\sigma}(\Phi, R)\, E'_{\pi,\sigma}(\Phi, R).
\]
If $R$ is a semilocal ring, then $G_{\pi,\sigma} (\Phi, R) = G^{0}_{\pi,\sigma} (\Phi, R) = G'_{\pi,\sigma} (\Phi, R)$ (see \cite[Proposition~5.8]{SG&DM1}).
In particular, in the simply connected case, if $R$ is a semilocal ring, then $G_{\mathrm{sc},\sigma}(\Phi,R)=E'_{\mathrm{sc},\sigma}(\Phi,R)$.


\subsection{The subgroup \texorpdfstring{$E'_{\pi, \sigma} (\Phi, R)$}{E'(R)}}\label{subsec:subgroup_E'}

We denote $\bar{\alpha} = \rho(\alpha)$, $\bar{\bar{\alpha}} = \rho^{2}(\alpha)$, $\bar{t} = \theta(t)$, and $\bar{\bar{t}} = \theta^{2}(t)$.  
Define
\[
    R_\theta = \{\, t \in R \mid t = \bar{t} \,\}
    \quad\text{and}\quad
    R^{-}_\theta = \{\, t \in R \mid t = -\bar{t} \,\}.
\]


\subsubsection{}

Observe that the nontrivial angle-preserving permutation $\rho$ of $\Delta$ induces an automorphism of $\mathcal{L}$, also denoted by $\rho$, satisfying
\[
    \rho(H_\alpha) = H_{\bar{\alpha}}, \qquad
    \rho(X_\alpha) = X_{\bar{\alpha}}, \qquad
    \rho(X_{-\alpha}) = X_{-\bar{\alpha}},
\]
for all $\alpha \in \Delta$.  
Consequently, we have $\rho(X_\alpha) = \epsilon_\alpha X_{\bar{\alpha}}$, where $\epsilon_\alpha \in \{\pm 1\}$ for all $\alpha \in \Phi$.

\begin{lemma}[{\cite[Proposition~3.1]{EA1}}]\label{epsilonalpha}
    \normalfont
    A Chevalley basis of $\mathcal{L}$ can be chosen such that:
    \begin{enumerate}[(a)]
        \item $\epsilon_\alpha = \epsilon_{\bar{\alpha}}$, 
        \item $\epsilon_\alpha = -1$ if $[\alpha] \sim A_2$ and $\alpha = \bar{\alpha}$,
        \item $\epsilon_\alpha = 1$ otherwise.
    \end{enumerate}
\end{lemma}

From now on, we always work with such a Chevalley basis.


\subsubsection{}

We now define certain special elements of $E'_{\pi, \sigma}(\Phi, R)$ as follows:

\begin{enumerate}
    \item If $[\alpha] \sim A_1$ (that is, $[\alpha]=\{ \alpha \}$), define $x_{[\alpha]}(t):=x_\alpha(t)$ for $t\in R_\theta$. In this case, $x_{[\alpha]}(t)x_{[\alpha]}(u)=x_{[\alpha]}(t+u)$ for every $t,u \in R_\theta$.

    \item If $[\alpha] \sim A_1^2$ (that is, $[\alpha]=\{ \alpha, \bar{\alpha} \}$), define $x_{[\alpha]}(t):=x_\alpha(t)x_{\bar{\alpha}}(\bar t)$ for $t\in R$. In this case, $x_{[\alpha]}(t)x_{[\alpha]}(u)=x_{[\alpha]}(t+u)$ for every $t,u \in R$. 
    
    \item If ${[\alpha]} \sim A_1^3$ (that is, $[\alpha]=\{ \alpha, \bar{\alpha}, \bar{\bar{\alpha}} \}$), define $x_{[\alpha]}(t):=x_\alpha(t)x_{\bar{\alpha}}(\bar t)x_{\bar{\bar{\alpha}}}(\bar{\bar t})$ for $t\in R$. 
    In this case, $x_{[\alpha]}(t)x_{[\alpha]}(u)=x_{[\alpha]}(t+u)$ for every $t,u \in R$. 
    
    \item If ${[\alpha]} \sim A_2$ with $\alpha \neq \bar{\alpha}$ (that is, ${[\alpha]}=\{ \alpha, \bar{\alpha}, \alpha + \bar{\alpha} \}$), define
    \[
        x_{[\alpha]}(t_1,t_2)
        :=
        x_\alpha(t_1)x_{\bar{\alpha}}(\bar t_1)
        x_{\alpha+\bar{\alpha}}(N_{\bar{\alpha},\alpha}t_2),
    \]
    where $t_1,t_2\in R$ satisfy $t_1\bar t_1=t_2+\bar t_2$. 
    In this case, $x_{[\alpha]}(t_1, t_2) \, x_{[\alpha]}(u_1,u_2)=x_{[\alpha]}(t_1 + u_1, t_2+u_2 + \bar{t}_1 u_1)$ for every $t_1, t_2, u_1, u_2 \in R$ such that $t_1 \bar{t}_1 = t_2 + \bar{t}_2$ and $u_1 \bar{u}_1 = u_2 + \bar{u}_2$.
\end{enumerate}

Define $\mathcal{A}(R):=\{(t_1,t_2)\mid t_1,t_2\in R \text{ and } t_1\bar t_1=t_2+\bar t_2\}$. For $[\alpha]\sim A_2$, the element $x_{[\alpha]}(t_1,t_2)$ is defined only when $(t_1,t_2)\in\mathcal{A}(R)$. The multiplication formula above suggests the following operation on $\mathcal{A}(R)$: for $(t_1,t_2),(u_1,u_2)\in\mathcal{A}(R)$, define
\[
    (t_1,t_2)\oplus(u_1,u_2)
    =
    (t_1+u_1,t_2+u_2+\bar t_1u_1).
\]
With this operation, $\mathcal{A}(R)$ is a group with identity $(0,0)$, and the inverse of $(t_1,t_2)$ is $(-t_1,\bar t_2)$. Thus,
\[
    x_{[\alpha]}(t,u)^{-1}=x_{[\alpha]}(-t,\bar u).
\]
Furthermore, the multiplicative monoid of $R$ acts on $\mathcal{A}(R)$ by
\[
    r\cdot(t_1,t_2)=(rt_1,r\bar r\,t_2),
\]
for $r\in R$ and $(t_1,t_2)\in\mathcal{A}(R)$. See \cite{EA1} for further details.

For $[\alpha] \in \Phi_\rho$, we write 
\begin{align*}
    R_{[\alpha]} = \begin{cases}
        R_\theta & \text{if } [\alpha] \sim A_1, \\
        R & \text{if } [\alpha] \sim A_1^2 \text{ or } A_1^3, \\
        \mathcal{A}(R) & \text{if } [\alpha] \sim A_2.
    \end{cases}
\end{align*}
If $[\alpha] \sim A_2$, then $t \in R_{[\alpha]}$ means that $t=(t_1,t_2)\in\mathcal{A}(R)$. We use the notation $r\cdot t$ as follows: it means $rt$ when $[\alpha]\sim A_1$, $A_1^2$, or $A_1^3$ (with $r\in R_\theta$ in the case $[\alpha]\sim A_1$), and it means $(rt_1,r\bar r\,t_2)$ when $[\alpha]\sim A_2$.

\begin{lemma}\label{lemma:generators_of_E(R)}
    The group $E'_{\pi, \sigma} (\Phi, R)$ is generated by $x_{[\alpha]}(t)$ for all $[\alpha] \in \Phi_\rho$ and $t \in R_{[\alpha]}$.
\end{lemma}


\subsubsection{}

Define $R^{\times}:=\{r\in R\mid \text{there exists }s\in R\text{ such that }rs=1\}$, $R_\theta^{\times}:=R_\theta\cap R^{\times}$, and $\mathcal{A}(R)^{\times}:=\{(t_1,t_2)\in\mathcal{A}(R)\mid t_2\in R^{\times}\}$.

For a given $[\alpha] \in \Phi_\rho$, we write 
\[
    R_{[\alpha]}^{\times} = \begin{cases}
        R^{\times}_\theta & \text{if } [\alpha] \sim A_1, \\
        R^{\times} & \text{if } [\alpha] \sim A_1^2 \text{ or } A_1^3, \\
        \mathcal{A}(R)^{\times} & \text{if } [\alpha] \sim A_2.
    \end{cases} 
\]

With these notations in place, we now define certain special elements of $N_{\pi, \sigma} (\Phi, R)$ and $H_{\pi, \sigma} (\Phi, R)$:
\begin{enumerate}
    \item[\textbf{(W1)}]
    For $[\alpha] \sim A_1$ and $t \in R_\theta^{{\times}}$, define
    \[
        w_{[\alpha]}(t) := x_{[\alpha]}(t) \, x_{-[\alpha]}(-t^{-1}) \, x_{[\alpha]}(t) = w_\alpha (t).
    \]

    \item[\textbf{(W2)}]
    For $[\alpha] \sim A_1^2$ and $t \in R^{{\times}}$, define
    \[
        w_{[\alpha]}(t) := x_{[\alpha]}(t) \, x_{-[\alpha]}(-t^{-1}) \, x_{[\alpha]}(t) = w_\alpha(t) \, w_{\bar{\alpha}}(\bar t).
    \]

    \item[\textbf{(W3)}]
    For $[\alpha]\sim A_1^3$ and $t \in R^{{\times}}$, define
    \[
        w_{[\alpha]}(t) := x_{[\alpha]}(t) \, x_{-[\alpha]}(-t^{-1}) \, x_{[\alpha]}(t) = w_\alpha(t) \, w_{\bar{\alpha}}(\bar{t}) \, w_{\bar{\bar{\alpha}}}(\bar{\bar{t}}).
    \]

    \item[\textbf{(W4)}]
    For $[\alpha]\sim A_2$ with $\alpha \neq \bar{\alpha}$ and for $t = (t_1,t_2) \in \mathcal{A}(R)^{{\times}}$, define
    \begin{multline*}
        w_{[\alpha]}(t) = w_{[\alpha]}(t_1,t_2) := \;  x_{[\alpha]}(t_1,t_2) \, x_{-[\alpha]}\bigl(-\bar t_2^{-1}\cdot(t_1,t_2)\bigr) \, x_{[\alpha]}\bigl(t_2\bar t_2^{-1}\cdot(t_1,t_2)\bigr) =\\
        = \;  x_{[\alpha]}(t_1,t_2) \, x_{-[\alpha]}\bigl(-\bar t_2^{-1}t_1,\bar t_2^{-1}\bigr) \, x_{[\alpha]}\bigl(t_2\bar t_2^{-1}t_1,t_2\bigr) 
        = \;  w_{\alpha}(\bar{t}_2) w_{\bar{\alpha}}(1) w_{\alpha}(t_2).
    \end{multline*}

    \item[\textbf{(W4$'$)}]
    For $[\alpha]\sim A_2$ with $\alpha \neq \bar{\alpha}$ and for $t \in R^{{\times}}$, define
    \[
        w'_{[\alpha]}(t) := w_\alpha(\bar t) \, w_{\bar{\alpha}}(1) \, w_\alpha(t).
    \] 

    \medskip

    \item[\textbf{(H1)}]
    For $[\alpha] \sim A_1$ and $t \in R_\theta^{{\times}}$, define
    \[
        h_{[\alpha]}(t) := w_{[\alpha]}(t) \, w_{[\alpha]}(-1) = h_\alpha(t).
    \]

    \item[\textbf{(H2)}]
    For $[\alpha]\sim A_1^2$ and $t\in R^{{\times}}$, define
    \[
        h_{[\alpha]}(t) := w_{[\alpha]}(t)\, w_{[\alpha]}(-1) = h_\alpha(t)\, h_{\bar{\alpha}}(\bar t).
    \]

    \item[\textbf{(H3)}]
    For $[\alpha]\sim A_1^3$ and $t\in R^{{\times}}$, define
    \[
        h_{[\alpha]}(t) := w_{[\alpha]}(t) \, w_{[\alpha]}(-1) = h_\alpha(t)\, h_{\bar{\alpha}}(\bar t)\, h_{\bar{\bar{\alpha}}}(\bar{\bar t}).
    \]

    \item[\textbf{(H4)}]
    For $[\alpha]\sim A_2$ with $\alpha \neq \bar{\alpha}$ and $(t_1, t_2),(u_1, u_2)\in\mathcal A(R)^{{\times}}$, define
    \[
        h_{[\alpha]}\bigl((t_1, t_2),(u_1, u_2)\bigr) := w_{[\alpha]}(t_1, t_2) \, w_{[\alpha]}(u_1, u_2) = h_{\alpha}(\bar{t}_2 \, u_2^{-1}) h_{\bar{\alpha}}(t_2 \, \bar{u}_2^{-1}).
    \]

    \item[\textbf{(H4$'$)}]
    For $[\alpha]\sim A_2$ with $\alpha\neq\bar{\alpha}$ and for
    $t \in R^{{\times}}$, define
    \[
        h'_{[\alpha]}(t) := h_\alpha(t) \, h_{\bar{\alpha}}(\bar t).
    \]
\end{enumerate}

\begin{rmk}
    \normalfont
    \begin{enumerate}[(a)]
        \item Let $w_{[\alpha]}(t)$, for $t \in R_{[\alpha]}^{\times}$, be as defined in $(\mathrm{W1}), (\mathrm{W2}), (\mathrm{W3})$, or $(\mathrm{W4})$. 
        Then $w_{[\alpha]}(t) \in N'_{\pi,\sigma}(\Phi,R) \subseteq N_{\pi,\sigma}(\Phi,R)$. 
        However, the element $w'_{[\alpha]}(t)$ defined in $(\mathrm{W4}')$ does not, in general, belong to $N'_{\pi, \sigma}(\Phi, R)$.
        
        \item Similarly, let $h_{[\alpha]} = h_{[\alpha]}(t)$, for $t \in R_{[\alpha]}^{\times}$, be as defined in $(\mathrm{H1}), (\mathrm{H2})$, or $(\mathrm{H3})$, and let $h_{[\alpha]} = h_{[\alpha]}(t, u)$, for $t, u \in \mathcal{A}(R)^{\times}$, be as defined in $(\mathrm{H4})$. 
        Then $h_{[\alpha]} \in H'_{\pi,\sigma}(\Phi,R) \subseteq H_{\pi,\sigma}(\Phi,R)$. 
        However, the element $h'_{[\alpha]}(t)$ defined in $(\mathrm{H4}')$ does not, in general, belong to $H'_{\pi, \sigma}(\Phi, R)$.
    \end{enumerate}
\end{rmk}

We do not discuss the properties of the above elements in detail here, and instead refer the reader to \cite{SG&DM1} and \cite{EA1} for a comprehensive treatment. 
In the subsequent sections, we shall make use of various commutation relations for twisted Chevalley groups, commonly known as the Chevalley commutator formulas, as well as the Steinberg relations describing the (conjugation) action of the elements $w_{[\alpha]}(t)$ and $h_{[\alpha]}(t)$ on the root unipotents $x_{[\alpha]}(t)$. 
For the sake of conciseness, we do not reproduce these identities here; a complete list of the relevant formulas can be found in \cite{SG&DM1}.


\subsubsection{}

We now address the following two fundamental lemmas. Although these results are generally regarded as folklore among experts, we have been unable to locate formal proofs in the existing literature. For the sake of completeness, we provide detailed proofs here.

\begin{lemma}\label{lemma:Lamda_I_is_sur}
    Let $R$ be a local ring, and let $I \subseteq \operatorname{Rad}(R)$ be a $\theta$-invariant ideal. Assume that $1/2\in R$ and, in addition, that $1/3\in R$ when $(\Phi,\rho)$ has twisted type ${}^3D_4$. Then the natural homomorphism
    \[
        \lambda_I \colon G_{\pi,\sigma}(\Phi,R) \to G_{\pi,\sigma}(\Phi,R/I)
    \]
    is surjective.
\end{lemma}

\begin{proof}
    Define $\widetilde{R} \coloneqq R/I$, and let $q \colon R \to \widetilde{R}$ be the canonical quotient homomorphism. Since the ideal $I$ is $\theta$-invariant, the automorphism $\theta$ induces a well-defined automorphism $\widetilde{\theta}$ on $\widetilde{R}$ satisfying
    \[
        q \circ \theta = \widetilde{\theta} \circ q.
    \]
    Because $R$ is a local ring, $\widetilde{R}$ is also local. Moreover, since $I \subseteq \operatorname{Rad}(R)$, the induced homomorphism on the groups of units,
    \[
        q^{\times} \colon R^{\times} \to \widetilde{R}^{\times},
    \]
    is surjective. Put $n \coloneqq o(\theta) = o(\rho)$.
    
    Since both $R$ and $\widetilde{R}$ are local rings, the twisted Chevalley groups admit the standard decompositions
    \[
        G_{\pi,\sigma}(\Phi,R) = T_{\pi,\sigma}(\Phi,R) E'_{\pi,\sigma}(\Phi,R)
        \qquad \text{and} \qquad
        G_{\pi,\sigma}(\Phi,\widetilde{R}) = T_{\pi,\sigma}(\Phi,\widetilde{R}) E'_{\pi,\sigma}(\Phi,\widetilde{R}).
    \]
    Therefore, it suffices to prove that both the torus and the elementary factors lift via $\lambda_I$. 

    We first consider the elementary factor. By Lemma~\ref{lemma:generators_of_E(R)}, it is sufficient to lift the elementary generators $x_{[\alpha]}(\widetilde{t})$ for $[\alpha] \in \Phi_\rho$ and $\widetilde{t} \in \widetilde{R}_{[\alpha]}$. That is, we must find an element $t \in R_{[\alpha]}$ such that $q(t) = \widetilde{t}$, which guarantees that $\lambda_I\bigl(x_{[\alpha]}(t)\bigr) = x_{[\alpha]}(\widetilde{t})$. 

    If $[\alpha] \sim A_1$, then $\widetilde{R}_{[\alpha]} = \widetilde{R}_{\widetilde{\theta}}$ and $R_{[\alpha]} = R_{\theta}$. We first choose an arbitrary lift $t' \in R$ such that $q(t') = \widetilde{t}$, and define
    \[
        t \coloneqq \frac{1}{n} \sum_{j=0}^{n-1} \theta^{j}(t').
    \]
    Then $t \in R_\theta$, and because $q \circ \theta = \widetilde{\theta} \circ q$ and $\widetilde{t} \in \widetilde{R}_{\widetilde{\theta}}$, we obtain $q(t) = \widetilde{t}$ as desired.

    If $[\alpha] \sim A_1^2$ or $[\alpha] \sim A_1^3$, then $\widetilde{R}_{[\alpha]} = \widetilde{R}$ and $R_{[\alpha]} = R$. In this case, any lift $t \in R$ of $\widetilde{t}$ satisfies the condition.

    If $[\alpha] \sim A_2$, then $\widetilde{R}_{[\alpha]} = \mathcal{A}(\widetilde{R})$ and $R_{[\alpha]} = \mathcal{A}(R)$. For $\widetilde{t} = (\widetilde{t}_1, \widetilde{t}_2) \in \mathcal{A}(\widetilde{R})$, we have the relation $\widetilde{t}_1 \widetilde{\theta}(\widetilde{t}_1) = \widetilde{t}_2 + \widetilde{\theta}(\widetilde{t}_2)$. Let $t_1', t_2' \in R$ be arbitrary lifts of $\widetilde{t}_1$ and $\widetilde{t}_2$, respectively. Set $t_1 \coloneqq t_1'$ and
    \[
        t_2 \coloneqq t'_2 + \frac{1}{2} \bigl( t'_1 \theta(t'_1) - t'_2 - \theta(t'_2) \bigr).
    \]
    Then $t = (t_1, t_2) \in \mathcal{A}(R)$ and $q(t) = \widetilde{t}$, as required.
    
    Consequently, it follows that $\lambda_I \bigl(E'_{\pi,\sigma}(\Phi,R)\bigr) = E'_{\pi,\sigma}(\Phi,\widetilde{R})$. 

    \smallskip

    We now consider the torus factor. By the definitions of $T_{\pi, \sigma}(\Phi, R)$ and $T_{\pi, \sigma}(\Phi, \widetilde{R})$, it suffices to show that for every self-conjugate character $\widetilde{\chi} \in \Hom_1(\Lambda_{\pi}, \widetilde{R}^{\times})$, we can construct a self-conjugate character $\chi \in \Hom_1(\Lambda_{\pi}, R^{\times})$ such that $q^{\times} \circ \chi = \widetilde{\chi}$. This ensures that $\lambda_I\bigl(h(\chi)\bigr) = h(\widetilde{\chi})$. 

    First, suppose $n=2$. By \cite[Theorem~74.3, p.~508]{CW&IR}, there exists a $\mathbb{Z}$-basis 
    \[
        a_1,\ldots,a_p,\quad b_1,\ldots,b_q,\quad c_1,d_1,\ldots,c_r,d_r
    \]
    for $\Lambda_{\pi}$ such that 
    \[
        \rho(a_i)=a_i, \qquad \rho(b_j)=-b_j, \qquad \rho(c_k)=d_k,\quad \rho(d_k)=c_k.
    \]

    For every $i$, we choose an arbitrary lift $u_i' \in R^{\times}$ such that $q(u_i') = \widetilde{\chi}(a_i)$, and define
    \[
        u_i \coloneqq \frac{1}{2}\bigl(u_i' + \theta(u_i')\bigr).
    \]
    Because $\widetilde{\chi}$ is self-conjugate, we have $\widetilde{\theta}\bigl(\widetilde{\chi}(a_i)\bigr) = \widetilde{\chi}(a_i)$, which implies that $q(u_i) = \widetilde{\chi}(a_i)$. Moreover, $u_i$ is invertible since its image modulo $I \subseteq \operatorname{Rad}(R)$ is invertible. We define $\chi(a_i) \coloneqq u_i$.

    For every $j$, set $\widetilde{v}_j \coloneqq \widetilde{\chi}(b_j)$. Since $\rho(b_j) = -b_j$ and $\widetilde{\chi}$ is self-conjugate, we have
    \[
        \widetilde{\theta}(\widetilde{v}_j) = \widetilde{\chi}(-b_j) = \widetilde{v}_j^{-1}.
    \]
    We aim to find an element $v_j \in R^{\times}$ such that $q(v_j) = \widetilde{v}_j$ and $v_j \theta(v_j) = 1$. Because $\widetilde{R}$ is a local ring and $2 \in \widetilde{R}^{\times}$, at least one of the elements $1 + \widetilde{v}_j$ or $1 - \widetilde{v}_j$ must be invertible. 

    If $1 + \widetilde{v}_j$ is invertible, we choose a lift $v_j' \in R^{\times}$ such that $q(v_j') = 1 + \widetilde{v}_j$, and define $v_j \coloneqq v_j' \theta(v_j')^{-1}$. This directly yields $v_j \theta(v_j) = 1$. Additionally, since $\widetilde{\theta}(\widetilde{v}_j) = \widetilde{v}_j^{-1}$, we obtain
    \[
        q(v_j) = \frac{1 + \widetilde{v}_j}{1 + \widetilde{v}_j^{-1}} = \widetilde{v}_j.
    \]
    Alternatively, if $1 - \widetilde{v}_j$ is invertible, we choose a lift $v_j' \in R^{\times}$ such that $q(v_j') = 1 - \widetilde{v}_j$, and define $v_j \coloneqq -v_j' \theta(v_j')^{-1}$. Again, $v_j \theta(v_j) = 1$, and 
    \[
        q(v_j) = -\frac{1 - \widetilde{v}_j}{1 - \widetilde{v}_j^{-1}} = \widetilde{v}_j,
    \]
    as desired. In either case, we set $\chi(b_j) \coloneqq v_j$.

    Finally, for every $k$, we choose an arbitrary lift $z_k \in R^{\times}$ such that $q(z_k) = \widetilde{\chi}(c_k)$, and define
    \[
        \chi(c_k) \coloneqq z_k \qquad \text{and} \qquad \chi(d_k) \coloneqq \theta(z_k).
    \]
    Because $\widetilde{\chi}$ is self-conjugate, this assignment is consistent:
    \[
        q\bigl(\chi(d_k)\bigr) = q\bigl(\theta(z_k)\bigr) = \widetilde{\theta}\bigl(\widetilde{\chi}(c_k)\bigr) = \widetilde{\chi}(d_k).
    \]
    Extending these values multiplicatively yields a character $\chi \colon \Lambda_\pi \to R^{\times}$ satisfying $q^{\times} \circ \chi = \widetilde{\chi}$ and $\chi\bigl(\rho(\lambda)\bigr) = \theta\bigl(\chi(\lambda)\bigr)$ for all $\lambda \in \Lambda_\pi$. Thus, $\chi$ is self-conjugate.

    It remains to consider the case $n=3$. In this setting, $\Phi$ is of type $D_4$, and the automorphism $\rho$ cyclically permutes the three nonzero elements of $\Lambda_{sc}/\Lambda_r \cong (\mathbb{Z}/2\mathbb{Z})^2$. Consequently, the only $\rho$-invariant intermediate lattices are $\Lambda_\pi = \Lambda_r$ and $\Lambda_\pi = \Lambda_{sc}$. In the former case, the simple roots form a $\mathbb{Z}$-basis permuted by $\rho$; in the latter case, the fundamental weights form such a basis.

    For each $\rho$-orbit of length three in this basis, we select a representative $\lambda$ and choose a lift $u_\lambda \in R^{\times}$ such that $q(u_\lambda) = \widetilde{\chi}(\lambda)$. We then define
    \[
        \chi\bigl(\rho^m(\lambda)\bigr) \coloneqq \theta^m(u_\lambda) \qquad \text{for } m = 0, 1, 2.
    \]
    For each basis element $\lambda$ fixed by $\rho$, we choose a lift $u'_\lambda \in R^{\times}$ such that $q(u'_\lambda) = \widetilde{\chi}(\lambda)$, and define
    \[
        u_\lambda \coloneqq \frac{1}{3}\bigl(u'_\lambda + \theta(u'_\lambda) + \theta^2(u'_\lambda)\bigr).
    \]
    This guarantees that $\theta(u_\lambda) = u_\lambda$ and $q(u_\lambda) = \widetilde{\chi}(\lambda)$. Because the image of $u_\lambda$ modulo $I \subseteq \operatorname{Rad}(R)$ is invertible, $u_\lambda$ is invertible in $R$. We set $\chi(\lambda) \coloneqq u_\lambda$.

    Extending these assignments multiplicatively once more produces a self-conjugate character $\chi \in \operatorname{Hom}_1(\Lambda_\pi, R^{\times})$ that lifts $\widetilde{\chi}$. 
    
    Consequently, we have shown that
    \[
        \lambda_I\bigl(T_{\pi,\sigma}(\Phi,R)\bigr) = T_{\pi,\sigma}(\Phi,\widetilde{R}).
    \]

    Since both the elementary and the torus factors lift along $\lambda_I$, the structural decompositions of the twisted Chevalley groups over $R$ and $\widetilde{R}$ ensure that the homomorphism
    \[
        \lambda_I \colon G_{\pi,\sigma}(\Phi,R) \longrightarrow G_{\pi,\sigma}(\Phi,\widetilde{R})
    \]
    is surjective.
\end{proof}

\begin{lemma}\label{lemma:E_ad=E_pi/Z}
    Assume that $(\Phi,\rho)$ has twisted type ${}^{2}A_\ell$ $(\ell\geq3)$, ${}^{2}D_\ell$ $(\ell\geq4)$, ${}^{2}E_6$, or ${}^{3}D_4$.
    Assume that $1/2\in R$ and, in the case ${}^{3}D_4$, also that $1/3\in R$.
    Then 
    \[
        E'_{\pi,\sigma}(\Phi,R) \big/ Z\!\left(E'_{\pi,\sigma}(\Phi,R)\right)  \;\cong\; E'_{\ad,\sigma}(\Phi,R).
    \]
\end{lemma}

\begin{proof}
    Consider the natural map 
    \[
        \delta_\pi:E_\pi(\Phi,R)\longrightarrow E_{\ad}(\Phi,R)
    \]
    given by
    \[
        \delta_\pi(x_{\alpha,\pi}(r)) = x_{\alpha,\ad}(r)
        \qquad
        (\alpha\in\Phi,\ r\in R).
    \]
    It is well known that $\delta_\pi$ is surjective and that its kernel is the center of $E_\pi(\Phi,R)$. 
    
    Since the map $\delta_\pi$ commutes with the automorphism
    $\sigma=\rho\circ\theta$, we have
    \[
        \delta_\pi(x_{[\alpha],\pi}(t))=x_{[\alpha],\ad}(t)
    \]
    for all $[\alpha] \in \Phi_\rho$ and all $t \in R_{[\alpha]}$.
    Therefore $\delta_\pi$ restricts to a surjective homomorphism
    \[
        \delta'_\pi: E'_{\pi,\sigma}(\Phi,R) \longrightarrow E'_{\ad,\sigma}(\Phi,R).
    \]
    
    We claim that
    \[
         \ker\delta'_\pi
         =
         Z\bigl(E'_{\pi,\sigma}(\Phi,R)\bigr).
    \]
    First, observe that $\ker \delta'_\pi = E'_{\pi,\sigma}(\Phi,R)\cap\ker\delta_\pi$. 
    Since $\ker\delta_\pi$ is central in $E_\pi(\Phi,R)$, it follows immediately that
    \[
        \ker\delta'_\pi\subseteq Z\bigl(E'_{\pi,\sigma}(\Phi,R)\bigr).
    \]
    Conversely, let $z\in Z(E'_{\pi,\sigma}(\Phi,R))$. Since $\delta'_\pi$ is
    surjective, the element $\delta'_\pi(z)$ commutes with every element of
    $E'_{\ad,\sigma}(\Phi,R)$. Hence
    \[
        \delta'_\pi(z)\in Z(E'_{\ad,\sigma}(\Phi,R)).
    \]
    By Suzuki's theorem~\cite{KS3}, the adjoint elementary twisted Chevalley group has trivial center under the present hypotheses. Thus $\delta'_\pi(z)=1$, and hence 
    $z \in \ker\delta'_\pi$.
    
    This proves the reverse inclusion, and hence the claim.
    
    Finally, the first isomorphism theorem gives
    \[
         E'_{\pi,\sigma}(\Phi,R)\big/Z\bigl(E'_{\pi,\sigma}(\Phi,R)\bigr)
         \cong
         E'_{\ad,\sigma}(\Phi,R),
    \]
    as required.
\end{proof}


\subsection{Normal subgroups of twisted Chevalley groups}
\label{subsec:normal_subgroup_of_G}

We now recall the main theorem from~\cite{SG&DM1}. Let $I$ be a
$\theta$-invariant ideal of $R$, i.e., $\theta(I)=I$. The canonical
projection $R\to R/I$ induces a group homomorphism
\[
    \lambda_I\colon
    G_{\pi,\sigma}(\Phi,R)
    \longrightarrow
    G_{\pi,\sigma}(\Phi,R/I).
\]
Define
\[
    G_{\pi,\sigma}(\Phi,I)
    \coloneqq
    \ker(\lambda_I)
\]
and
\[
    G_{\pi,\sigma}(\Phi,R,I)
    \coloneqq
    \lambda_I^{-1}
    \left(
        Z\bigl(G_{\pi,\sigma}(\Phi,R/I)\bigr)
    \right),
\]
where $Z(G_{\pi,\sigma}(\Phi,R/I))$ denotes the center of
$G_{\pi,\sigma}(\Phi,R/I)$. The subgroup
$G_{\pi,\sigma}(\Phi,I)$ is called the \emph{principal congruence
subgroup} of level~$I$, while $G_{\pi,\sigma}(\Phi,R,I)$ is called
the \emph{full congruence subgroup} of level~$I$.

For $[\alpha]\in\Phi_\rho$, define
\[
    I_{[\alpha]}
    =
    \begin{cases}
        I_\theta
        & \text{if }[\alpha]\sim A_1,\\
        I
        & \text{if }[\alpha]\sim A_1^2\text{ or }A_1^3,\\
        \mathcal{A}(I)
        & \text{if }[\alpha]\sim A_2,
    \end{cases}
\]
where
\[
    I_\theta
    =
    I\cap R_\theta
    =
    \{\,t\in I\mid \theta(t)=t\,\}
\]
and
\[
    \mathcal{A}(I)
    =
    \{\,(t_1,t_2)\in\mathcal{A}(R)
       \mid t_1,t_2\in I\,\}.
\]

Let $E'_{\pi,\sigma}(\Phi,I)$ denote the subgroup of
$E'_{\pi,\sigma}(\Phi,R)$ generated by the elements
$x_{[\alpha]}(t)$ for all $[\alpha]\in\Phi_\rho$ and
$t\in I_{[\alpha]}$. Additionally, define
$E'_{\pi,\sigma}(\Phi,R,I)$ as the normal subgroup of
$E'_{\pi,\sigma}(\Phi,R)$ generated by
$E'_{\pi,\sigma}(\Phi,I)$. This subgroup is called the
\emph{relative elementary subgroup} of level~$I$.

\begin{thm}[S. M. Garge and D. H. Makadiya~\cite{SG&DM1}]
\label{thm:normalsubgroups}
    \normalfont
    Assume that $(\Phi,\rho)$ has one of the following twisted
    types:
    ${}^2A_\ell$ ($\ell\geq 3$),
${}^2D_\ell$ ($\ell\geq 4$),
        ${}^2E_6$,
        ${}^3D_4$.
    
    Assume that $2\in R^\times$ and, in addition, that
    $3\in R^\times$ in the case of twisted type ${}^3D_4$.
    Then a subgroup $H$ of $G_{\pi,\sigma}(\Phi,R)$ is normalized
    by $E'_{\pi,\sigma}(\Phi,R)$ if and only if there exists a
    unique $\theta$-invariant ideal $I$ of $R$ such that
    \[
        E'_{\pi,\sigma}(\Phi,R,I)
        \subseteq
        H
        \subseteq
        G_{\pi,\sigma}(\Phi,R,I).
    \]
\end{thm}

As an immediate consequence of the above theorem, we obtain a
classification of the normal subgroups of elementary twisted
Chevalley groups.

\begin{cor}\label{cor:normalsubgroups}
    \normalfont
    Let $R$ and $(\Phi,\rho)$ be as in
    Theorem~\ref{thm:normalsubgroups}. A subgroup $H$ of
    $E'_{\pi,\sigma}(\Phi,R)$ is normal in
    $E'_{\pi,\sigma}(\Phi,R)$ if and only if there exists a unique
    $\theta$-invariant ideal $I$ of $R$ such that
    \[
        E'_{\pi,\sigma}(\Phi,R,I)
        \subseteq
        H
        \subseteq
        G_{\pi,\sigma}(\Phi,R,I)
        \cap
        E'_{\pi,\sigma}(\Phi,R).
    \]
\end{cor}


\subsection{Automorphisms of twisted Chevalley groups}
\label{subsec:automorphism}

We now introduce some specific types of automorphisms of a twisted
Chevalley group $G_{\pi,\sigma}(\Phi,R)$. Since
$E'_{\pi,\sigma}(\Phi,R)$ is a characteristic subgroup of
$G_{\pi,\sigma}(\Phi,R)$, each of these automorphisms restricts
naturally to $E'_{\pi,\sigma}(\Phi,R)$.

\medskip

\noindent
\textbf{Inner automorphisms.}
Let $S$ be a commutative ring containing $R$ as a subring, and assume
that the automorphism $\theta$ extends to an automorphism of $S$,
also denoted by $\theta$. Let
$g\in G_{\pi,\sigma}(\Phi,S)$
be such that
\[
    gG_{\pi,\sigma}(\Phi,R)g^{-1}
    =
    G_{\pi,\sigma}(\Phi,R).
\]
Then the map
$x\longmapsto gxg^{-1}$
defines an automorphism of $G_{\pi,\sigma}(\Phi,R)$, denoted by
$i_g$, and is called an \emph{inner automorphism} induced by $g$.
If $g\in G_{\pi,\sigma}(\Phi,R)$, then $i_g$ is called
\emph{strictly inner}.

\medskip

\noindent
\textbf{Ring automorphisms.}
Let $\mu\colon R\to R$ be an automorphism of the ring $R$ that
commutes with $\theta$, i.e.,
\[
    \mu\circ\theta=\theta\circ\mu.
\]
The entrywise map
$(a_{ij})
    \longmapsto
    \bigl(\mu(a_{ij})\bigr)$
defines an automorphism of $G_{\pi,\sigma}(\Phi,R)$, which, by an
abuse of notation, is also denoted by $\mu$ and is called a
\emph{ring automorphism}. Under this automorphism,
\[
    x_{[\alpha]}(t)
    \longmapsto
    x_{[\alpha]}\bigl(\mu(t)\bigr)
\]
for all $[\alpha]\in\Phi_\rho$ and $t\in R_{[\alpha]}$. Here, if
$[\alpha]\sim A_2$ and $t=(t_1,t_2)\in\mathcal{A}(R)$, then
\[
    \mu(t)
    \coloneqq
    \bigl(\mu(t_1),\mu(t_2)\bigr).
\]

\medskip

\noindent
\textbf{Central automorphisms.}
Let
\[
    G=G_{\pi,\sigma}(\Phi,R)
\]
and let
$\tau\colon G\longrightarrow Z(G)$
be a group homomorphism. Then the map
\[
    x\longmapsto \tau(x)x
\]
defines an endomorphism of $G$, which we also denote by $\tau$.
If this endomorphism is an automorphism, it is called a
\emph{central automorphism}.

Since the adjoint twisted Chevalley groups considered here are
centerless, all central automorphisms of
$G_{\mathrm{ad},\sigma}(\Phi,R)$ are trivial. Moreover, since
$E'_{\pi,\sigma}(\Phi,R)$ is perfect, every homomorphism from
$E'_{\pi,\sigma}(\Phi,R)$ to an abelian group, in particular to its
center, is trivial. Consequently, every central automorphism of
$E'_{\pi,\sigma}(\Phi,R)$ is the identity automorphism.

\begin{defn}
    An automorphism of $G_{\pi,\sigma}(\Phi,R)$ or
    $E'_{\pi,\sigma}(\Phi,R)$ is said to be \emph{standard} if it
    can be expressed as a composition of an inner, a ring, and a
    central automorphism.
\end{defn}

We now recall Steinberg's theorem on automorphisms of elementary
twisted Chevalley groups over a field.

\begin{thm}[R. Steinberg~\cite{RS}, Theorem~36]
\label{thm:auto of E over field}
    \normalfont
    Let $k$ be a field, let $\sigma$ be nontrivial, and let
    $G=E'_{\pi,\sigma}(\Phi,k)$
    be an elementary twisted Chevalley group of one of the following
    types:
    ${}^2A_\ell$ ($\ell\geq 3$),
        ${}^2D_\ell$ ($\ell\geq 4$),
        ${}^2E_6$,
        ${}^3D_4$.
        
    Then every automorphism of $G$ is a product of an inner, a
    diagonal, and a field automorphism.
\end{thm}

The diagonal automorphisms occurring in Steinberg's formulation are
induced by conjugation by suitable torus elements over field
extensions and are therefore inner automorphisms in the generalized
sense adopted above. Thus every automorphism in
Theorem~\ref{thm:auto of E over field} is standard in our
terminology.

The main objective of this series of papers is to generalize the
above result to commutative rings. In the present paper, we focus
primarily on the case ${}^2A_\ell$ $(\ell\geq 5)$ over local rings.


\section{The Group \texorpdfstring{$G_{\pi, \sigma}(A_\ell, R)$}{G(R)} and the Main Theorems}\label{sec:G(R)_and_MT}


In this section, assuming $2\in R^{\times}$, we provide a detailed treatment of the twisted Chevalley groups associated with the root system $A_\ell$ for $\ell \geqslant 3$. 
We begin by describing the structure of the underlying Lie algebra of type $A_\ell$, followed by an examination of the simply connected and adjoint twisted groups. 
Finally, we state our main theorems for automorphisms of these groups. 


\subsection{The Lie algebra of type \texorpdfstring{$A_\ell \ (\ell \geqslant 3)$}{A-l}}

Let $\Phi$ be the root system of type $A_\ell$, which can be realized in the Euclidean space $\mathbb{R}^{\ell+1}$ as the set of vectors $\{\varepsilon_i - \varepsilon_j \mid 1 \leqslant i \neq j \leqslant \ell+1\}$. We fix a system of simple roots $\Delta = \{\alpha_i = \varepsilon_i - \varepsilon_{i+1} \mid 1 \leqslant i \leqslant \ell \}$. Let $\rho$ be the non-trivial, angle-preserving permutation of $\Delta$, defined by $\rho(\alpha_i) = \alpha_{\ell+1-i}$ for all $i = 1, \dots, \ell$. This permutation extends naturally to an isometry of the root system $\Phi$, which we likewise denote by $\rho$.

The corresponding Lie algebra $\mathcal{L} = \mathfrak{sl}_{\ell+1}(\mathbb{C})$ consists of all complex $(\ell+1) \times (\ell+1)$ matrices with vanishing trace. Our primary objective is to select a Chevalley basis $\{X_\alpha, H_{\alpha_i} \mid \alpha \in \Phi, \alpha_i \in \Delta \}$ for $\mathcal{L}$ that satisfies the requirements of Lemma~\ref{epsilonalpha}. 

We begin by specifying the basis elements corresponding to the simple roots. For each $\alpha_i = \varepsilon_i - \varepsilon_{i+1} \in \Delta$, we set $X_{\alpha_i} = E_{i, i+1}$ and $X_{-\alpha_i} = E_{i+1, i}$, which implies that $H_{\alpha_i} = E_{i,i} - E_{i+1, i+1}$. Note that the permutation $\rho$ of $\Phi$ induces an automorphism of $\mathcal{L}$ satisfying:
\[
    \rho(H_{\alpha_i}) = H_{\rho(\alpha_i)}, \qquad \rho(X_{\alpha_i}) = X_{\rho(\alpha_i)}, \qquad \rho(X_{-\alpha_i}) = X_{-\rho(\alpha_i)}
\]
for all $\alpha_i \in \Delta$. It is a straightforward verification that this automorphism is given explicitly by the map
$A \mapsto -Q A^t Q^{-1}$,
where $Q$ is the antidiagonal matrix 
\[
    Q = \begin{pmatrix}
        0 & \dots & 0 & 0 & 1 \\
        0 & \dots & 0 & -1 & 0 \\
        0 & \dots & 1 & 0 & 0 \\
        \vdots & \reflectbox{$\ddots$} & \vdots & \vdots & \vdots \\
        (-1)^\ell & \dots & 0 & 0 & 0
    \end{pmatrix}_{(\ell + 1) \times (\ell + 1)}.
\]
Having established the basis elements for the simple roots, we now extend this definition to all roots in $\Phi$. For any $\alpha = \varepsilon_i - \varepsilon_j \in \Phi$, we define:
\[
    X_{\alpha} = \begin{cases}
        - E_{i, j}, & \text{if } |i-j| \text{ is even and } i + j > \ell+2; \\
        E_{i,j}, & \text{otherwise}.
    \end{cases}
\]
One may verify that this choice of basis elements satisfies the consistency conditions required by Lemma~\ref{epsilonalpha}.


\subsection{The simply connected group \texorpdfstring{$G_{\text{sc}, \sigma} (A_\ell, R)$}{G-sc(R)}}

Let $G_{\text{sc}}(A_\ell, R)$ be the simply connected Chevalley group of type $A_\ell$ for $\ell \geqslant 3$. We consider the involution $\sigma = \rho \circ \theta$ of $G_{\text{sc}}(A_\ell, R)$, where $\rho$ is the graph automorphism induced by the non-trivial symmetry of the Dynkin diagram of type $A_\ell$, and $\theta$ is a ring automorphism of order $2$. The subgroup of fixed points under $\sigma$, denoted by $G_{\text{sc}, \sigma} (A_\ell, R)$, is the simply connected twisted Chevalley group of type ${}^2 A_\ell$.

It is well known that the simply connected Chevalley group $G_{\text{sc}}(A_\ell, R)$ is isomorphic to the special linear group $\operatorname{SL}_{\ell+1}(R)$. Under this isomorphism, the graph automorphism $\rho$ of $\operatorname{SL}_{\ell+1}(R)$ is given explicitly by the mapping
\[
    A \mapsto Q (A^{-t}) Q^{-1},
\]
where $Q$ is the antidiagonal matrix defined in the previous section. Let $\bar{A}$ denote the matrix obtained by applying the ring automorphism $\theta$ to each entry of $A$. Then the action of the twisted involution $\sigma$ on an element $A \in \operatorname{SL}_{\ell+1}(R)$ is given by
\[
    \sigma(A) = Q (\bar{A}^{-t}) Q^{-1}.
\]

We define the special unitary group $\operatorname{SU}_{\ell+1}(R)$ to be the fixed-point subgroup of $\operatorname{SL}_{\ell+1}(R)$ under $\sigma$:
\[
    \operatorname{SU}_{\ell+1}(R) = \{ A \in \operatorname{SL}_{\ell+1}(R) \mid A Q \bar{A}^t = Q \}.
\]
Consequently, the simply connected twisted Chevalley group $G_{\text{sc}, \sigma} (A_\ell, R)$ is isomorphic to $\operatorname{SU}_{\ell+1}(R)$.

Under this realization, we record the explicit matrix forms of the elementary generators, Weyl elements, and torus elements. 
For $[\alpha] \in \Phi_\rho$ and $t \in R_{[\alpha]}$, the elementary generators $x_{[\alpha]}(t)$ are defined as:
\[
    x_{[\alpha]}(t) = \begin{cases}
        I + t X_\alpha & \text{if } [\alpha] \sim A_1, \\
        (I+t X_\alpha) (I + \bar{t} X_{\bar{\alpha}}) & \text{if } [\alpha] \sim A_1^2, \\
        (I + t_1 X_\alpha) (I + \bar{t_1} X_{\bar{\alpha}}) (I + N_{\bar{\alpha}, \alpha} \, t_2 X_{\alpha + \bar{\alpha}}) & \text{if } [\alpha] \sim A_2.
    \end{cases}
\]
Similarly, for $t \in R_{[\alpha]}^{{\times}}$, the Weyl elements $w_{[\alpha]}(t)$ have the following form (cf. Section~\ref{subsec:subgroup_E'}):
\[
    w_{[\alpha]}(t) = \begin{cases}
        (I - X_{\alpha} X_{-\alpha} - X_{-\alpha} X_{\alpha}) + (t X_{\alpha} - t^{-1} X_{-\alpha}) & \text{if } [\alpha] \sim A_1, \\[1.5ex]
        
        \begin{aligned}
            & \left(I - \sum_{\beta \in \{ \alpha,\bar{\alpha} \}} (X_{\beta} \, X_{-\beta} + X_{-\beta} \, X_{\beta}) \right) \\
            & \qquad + (t \, X_{\alpha} + \bar{t} \, X_{\bar{\alpha}} - t^{-1}X_{-\alpha}-\bar t^{-1}X_{-\bar\alpha}),
        \end{aligned} & \text{if } [\alpha]\sim A_1^2,\\[1.5ex]
        
        \begin{aligned}
            & \bigl(I - X_{\alpha + \bar{\alpha}} \, X_{-\alpha-\bar\alpha}
             - X_{-\alpha-\bar\alpha} \, X_{\alpha+\bar\alpha} \bigr) + N_{\alpha,\bar\alpha} \bigl(\bar t_2X_{\alpha+\bar\alpha}
            +t_2^{-1}X_{-\alpha-\bar\alpha}\bigr) \\
            & \qquad - \frac{(\bar t_2^{-1}t_2+1)}{2}
            \bigl(
            X_{\alpha}X_{-\alpha}
            +X_{-\alpha}X_{\alpha}
            +X_{\bar\alpha}X_{-\bar\alpha}
            +X_{-\bar\alpha}X_{\bar\alpha} \\
            &\hspace{5.5em}
            -X_{\alpha+\bar\alpha}X_{-\alpha-\bar\alpha}
            -X_{-\alpha-\bar\alpha}X_{\alpha+\bar\alpha}
            \bigr),
        \end{aligned} & \text{if } [\alpha]\sim A_2.
    \end{cases}
\]
For $[\alpha]\sim A_1$ or $A_1^2$, the torus elements are given by the following formula, where $t_\alpha=t$ and $t_{\bar\alpha}=\bar t$:
\[
    h_{[\alpha]}(t) = \begin{cases}
        I + (t-1) X_{\alpha} X_{-\alpha} + (t^{-1}-1) X_{-\alpha}X_{\alpha} & \text{if } [\alpha] \sim A_1, \\
        I + \sum_{\beta \in \{\alpha, \bar{\alpha}\}} \left( (t_\beta-1) X_{\beta} X_{-\beta} + (t_\beta^{-1}-1) X_{-\beta}X_{\beta} \right) & \text{if } [\alpha] \sim A_1^2.
    \end{cases}
\]

In this matrix realization, $x_{[\alpha]}(t)$ is unipotent, $w_{[\alpha]}(t)$ is a monomial matrix, and $h_{[\alpha]}(t)$ is a diagonal matrix in $\operatorname{SU}_{\ell+1}(R)$, all of which lie in the elementary subgroup $E'_{\text{sc}, \sigma} (\Phi, R)$. 
Specifically, if $\ell = 2m-1$ is odd, this elementary subgroup contains the special elements $x_{[\alpha]}(1)$, $w_{[\alpha]}(1)$, and $h_{[\alpha]}(-1)$ for all roots $[\alpha] \in \Phi_\rho$. Similarly, if $\ell = 2m$ is even, it contains these same elements for all roots $[\alpha] \in \Phi_\rho$ of type $A_1^2$, alongside $x_{[\alpha]}(1,1/2)$ and $w_{[\alpha]}(1,1/2)$ for all roots $[\alpha] \in \Phi_\rho$ of type $A_2$. 

Before turning to the adjoint case, we establish the following lemma. 
For brevity, we write $u_{[\alpha]}$ for $x_{[\alpha]}(1)$ when $[\alpha] \sim A_1$ or $A_1^2$, and $x_{[\alpha]}(1,1/2)$ when $[\alpha] \sim A_2$.

\begin{lemma}\label{lemma:M(R)_as_R_algebra}
    The elements $u_{[\alpha]}$ for $[\alpha] \in \Phi_\rho$ generate $M_{\ell+1}(R)$ as a unital $R$-algebra.
\end{lemma}

\begin{proof}
    Let $M$ be the unital $R$-subalgebra of $M_{\ell+1}(R)$ generated by $\{u_{[\alpha]} \mid [\alpha] \in \Phi_\rho\}$. 
    It suffices to show that $E_{i,i} \in M$ for all $i = 1, \dots, \ell+1$, and that $E_{i,i+1}, E_{i+1,i} \in M$ for all $i = 1, \dots, \ell$.

    \medskip
    
    First, consider the case $\ell = 2m-1$ with $m \geq 2$. In this case, all roots are of type $A_1$ or $A_1^2$. For each $i = 1, \dots, m$, set $\beta_i = \varepsilon_i - \varepsilon_{2m+1-i}$. Then the classes $[\beta_i]$ comprise all positive roots of type $A_1$. Observe that
    \begin{align*}
        X_{\pm\beta_i} &= x_{\pm[\beta_i]}(1) - I, \qquad E_{i,\,2m+1-i}=X_{\beta_i}, \\
        E_{i,i} - E_{2m+1-i,\,2m+1-i} &= H_{\beta_i} = [X_{\beta_i}, X_{-\beta_i}], \\
        E_{i,i} + E_{2m+1-i,\,2m+1-i} &= \tfrac{1}{2}\bigl(I - h_{[\beta_i]}(-1)\bigr) = \tfrac{1}{2}\bigl(I - \bigl(x_{[\beta_i]}(1)^{-1} x_{-[\beta_i]}(1) x_{[\beta_i]}(1)^{-1}\bigr)^2\bigr).
    \end{align*}
    Hence $E_{i,i}, E_{2m+1-i,\,2m+1-i} \in M$ for $i = 1, \dots, m$, and therefore $E_{i,i} \in M$ for all $i = 1, \dots, \ell+1$. In addition, $E_{m,\,m+1}, E_{m+1,\,m} \in M$.
    
    For each $i = 1, \dots, m-1$, we have
    \begin{align*}
        E_{i,i+1} + E_{2m-i,\,2m+1-i} &= X_{\alpha_i} + X_{\alpha_{2m-i}} = x_{[\alpha_i]}(1) - I, \\
        E_{i+1,i} + E_{2m+1-i,\,2m-i} &= X_{-\alpha_i} + X_{-\alpha_{2m-i}} = x_{-[\alpha_i]}(1) - I.
    \end{align*}
    Multiplying these elements by the diagonal matrix units already obtained, we conclude that $E_{i,i+1}, E_{i+1,i} \in M$ for all $i = 1, \dots, \ell$.

    \medskip
    
    Now consider the case $\ell = 2m$ with $m \geq 2$. Here all roots are of type $A_2$ or $A_1^2$. For each $i = 1, \dots, m$, set $\beta_i = \varepsilon_i - \varepsilon_{2m+2-i}$. Then $\rho(\beta_i) = \beta_i$ and $[\beta_i] \sim A_2$, more precisely
    \[
        [\beta_i] = \{ \varepsilon_i - \varepsilon_{m+1}, \ \varepsilon_{m+1} - \varepsilon_{2m+2-i}, \ \varepsilon_i - \varepsilon_{2m+2-i} \}.
    \]
    We have
    \begin{align*}
        X_{\pm\beta_i} &= (-1)^{m-i}\bigl(x_{\pm[\beta_i]}(1,1/2) - I\bigr)^2, \qquad E_{i,\,2m+2-i}=X_{\beta_i}, \\
        E_{i,i} - E_{2m+2-i,\,2m+2-i} &= H_{\beta_i} = [X_{\beta_i}, X_{-\beta_i}], \\
        E_{i,i} + E_{2m+2-i,\,2m+2-i} &= \frac{1}{2}\left(I - \bigl((I - X_{\beta_i})(I + X_{-\beta_i})(I - X_{\beta_i})\bigr)^2\right).
    \end{align*}
    It follows that $E_{i,i}, E_{2m+2-i,\,2m+2-i} \in M$ for $i = 1, \dots, m$, and hence $E_{i,i} \in M$ for all $i = 1, \dots, m, m+2, \dots, 2m+1$.
    
    For $i = 1, \dots, m-1$, we similarly have
    \begin{align*}
        E_{i,i+1} + E_{2m+1-i,\,2m+2-i}& = X_{\alpha_i} + X_{\alpha_{2m+1-i}} = x_{[\alpha_i]}(1) - I, \\
        E_{i+1,i} + E_{2m+2-i,\,2m+1-i}& = X_{-\alpha_i} + X_{-\alpha_{2m+1-i}} = x_{-[\alpha_i]}(1) - I,
    \end{align*}
    which implies $E_{i,i+1}, E_{i+1,i} \in M$ for all $i = 1, \dots, m-1, m+2, \dots, 2m$.
    
    It remains to treat the middle indices. Consider the root class
    \[
        [\beta_m] = \{\alpha_m, \alpha_{m+1}, \beta_m\}.
    \]
    Then
    \begin{align*}
        E_{m,\,m+2} &= X_{\beta_m} = \bigl(x_{[\beta_m]}(1,1/2) - I\bigr)^2, \\
        E_{m+2,\,m} &= X_{-\beta_m} = \bigl(x_{-[\beta_m]}(1,1/2) - I\bigr)^2, \\
        E_{m,m+1} + E_{m+1,m+2} &= X_{\alpha_m} + X_{\alpha_{m+1}} = \bigl(x_{[\beta_m]}(1,1/2) - I\bigr) - \tfrac{1}{2}X_{\beta_m}, \\
        E_{m+1,m} + E_{m+2,m+1} &= X_{-\alpha_m} + X_{-\alpha_{m+1}} = \bigl(x_{-[\beta_m]}(1,1/2) - I\bigr) - \tfrac{1}{2}X_{-\beta_m}.
    \end{align*}
    Combining these relations with the previously obtained elements, we deduce that $E_{i,i+1}, E_{i+1,i} \in M$ for $i = m, m+1$, and consequently $E_{m+1,m+1} \in M$.
\end{proof}


\subsection{The adjoint group \texorpdfstring{$G_{\text{ad}, \sigma} (A_\ell, R)$}{G-ad(R)}}

In its classical matrix realization, the adjoint Chevalley group $G_{\text{ad}}(A_\ell, R)$ is isomorphic to the projective general linear group $\operatorname{PGL}_{\ell+1}(R)$, which is defined as the group of all $R$-algebra automorphisms of the matrix algebra $M_{\ell+1}(R)$. This yields the exact sequence
\[
    1 \to R^{\times} \to \operatorname{GL}_{\ell+1}(R) \to \operatorname{PGL}_{\ell+1}(R) \to \operatorname{Pic}(R),
\]
where $\operatorname{Pic}(R)$ denotes the Picard group of $R$ (cf.\ \cite{HM}). When $R$ is a local ring, $\operatorname{Pic}(R)$ is trivial, and the sequence simplifies to the isomorphism
\[
    \operatorname{PGL}_{\ell+1}(R) \cong \operatorname{GL}_{\ell+1}(R) / \{\lambda I_{\ell+1} \mid \lambda \in R^{\times}\}.
\]
Consequently, over a local ring, each element of the adjoint group can be represented as an equivalence class $[A]$ of a matrix $A \in \operatorname{GL}_{\ell+1}(R)$. For the remainder of this subsection, we assume $R$ is a local ring.

Consider the involution $\sigma$ on $\operatorname{PGL}_{\ell+1}(R)$ defined by the mapping
\[
    [A] \mapsto [Q \bar{A}^{-t} Q^{-1}],
\]
where $Q$ is the antidiagonal matrix established in the previous section, and $\bar{A} = \theta(A)$ denotes the entrywise application of the ring involution $\theta$. The adjoint twisted Chevalley group $G_{\text{ad}, \sigma}(A_\ell, R)$ is isomorphic to the subgroup of $\sigma$-fixed points in $\operatorname{PGL}_{\ell+1}(R)$:
\[
    G_{\text{ad}, \sigma}(A_\ell, R) \cong \{[A] \in \operatorname{PGL}_{\ell+1}(R) \mid [A Q \bar{A}^t] = [Q]\}.
\]
This fixed-point condition dictates that for any representative $A \in \operatorname{GL}_{\ell+1}(R)$ of such an element, there exists a scalar $\lambda \in R^{\times}$ satisfying
\[
    A Q \bar{A}^t = \lambda Q.
\]

By Lemma~\ref{lemma:E_ad=E_pi/Z}, we have $E'_{\text{ad}, \sigma}(A_\ell, R) \cong E'_{\text{sc}, \sigma}(A_\ell, R) / Z(E'_{\text{sc}, \sigma}(A_\ell, R))$. Therefore, any class $[A] \in E'_{\text{ad}, \sigma}(A_\ell, R)$ admits a representative $A \in E'_{\text{sc}, \sigma}(A_\ell, R) \subset \operatorname{GL}_{\ell+1}(R)$ with determinant $1$. 

In particular, every elementary generator $x_{[\alpha]}(t) \in E'_{\text{ad}, \sigma}(A_\ell, R) \subset \operatorname{PGL}_{\ell+1}(R)$ lifts to a corresponding element $x_{[\alpha]}(t) \in E'_{\text{sc}, \sigma}(A_\ell, R) \subset \operatorname{GL}_{\ell+1}(R)$, as defined in the preceding section. Consequently, we slightly abuse notation within the adjoint group by identifying the generator with its representative, writing $x_{[\alpha]}(t) = [x_{[\alpha]}(t)]$. The distinction between the group element and its lift in $\operatorname{GL}_{\ell+1}(R)$ will remain clear from context. We adopt an analogous convention for the Weyl elements $w_{[\alpha]}(t)$ and the torus elements $h_{[\alpha]}(t)$.


\subsection{The main theorems}

We conclude this section by stating the main results of the paper.

\begin{thm}\label{MT_local1}
    Let $R$ be a local ring with $2\in R^{\times}$, and let $G = E'_{\text{ad}, \sigma}(\Phi,R)$ be an adjoint elementary twisted Chevalley group of type ${}^2A_{\ell} \ (\ell \geqslant 5)$. Then every automorphism of $G$ is standard; that is, it can be written as a composition of a strictly inner automorphism and a ring automorphism of~$G$.  
\end{thm}

\begin{thm}\label{MT_local2}
    Let $R$ be a local ring with $2\in R^{\times}$, and let $G = E'_{\pi, \sigma} (\Phi, R)$ be an elementary twisted Chevalley group of type ${}^2A_{\ell} \ (\ell \geqslant 5)$. Then every automorphism of $G$ is standard; that is, it can be expressed as a composition of an inner automorphism and a ring automorphism of~$G$.
\end{thm}

\begin{thm}\label{MT_local3}
    Let $R$ be a local ring with $2\in R^{\times}$, and let $G = G_{\pi, \sigma} (\Phi, R)$ be a twisted Chevalley group of type ${}^2A_{\ell} \ (\ell \geqslant 5)$. Then every automorphism of $G$ is standard; that is, it can be written as a composition of an inner, a central, and a ring automorphism of~$G$.
\end{thm}

The proof of Theorem~\ref{MT_local1} is carried out in Sections~\ref{sec:outline_of_main_thm}--\ref{sec:image_of_x(t)}. 
In Section~\ref{sec:outline_of_main_thm}, we outline the argument and reduce the problem to several key steps. 
These steps are then proved in Sections~\ref{sec:image_of_h(-1)}--\ref{sec:image_of_x(t)}. 
Finally, in Section~\ref{sec:proof_of_MT2_and_MT3}, we prove Theorems~\ref{MT_local2} and~\ref{MT_local3}.


\section{Outline of the Proof of Theorem~\ref{MT_local1}}\label{sec:outline_of_main_thm}

Let $R$ be a local ring with unity, and let $I \subseteq R$ be a $\theta$-invariant ideal. We introduce the following notation (cf.~Section~\ref{subsec:normal_subgroup_of_G}):
\[
    N_I \coloneqq E'_{\text{ad}, \sigma} (\Phi, R) \cap G_{\text{ad}, \sigma}(\Phi, I), \quad
    C_I \coloneqq E'_{\text{ad}, \sigma} (\Phi, R) \cap G_{\text{ad}, \sigma}(\Phi, R, I), \quad
    E_I \coloneqq E'_{\text{ad}, \sigma} (\Phi, R, I).
\]
Since we are working in the adjoint case, we have $G_{\text{ad}, \sigma}(\Phi, I) = G_{\text{ad}, \sigma}(\Phi, R, I)$ for all such ideals $I$, and hence $N_I = C_I$. By Corollary~\ref{cor:normalsubgroups}, under the hypotheses of the main theorem, the following classification holds:
\begin{quote}
    A subgroup $H$ of $E'_{\text{ad}, \sigma}(\Phi, R)$ is normal in $E'_{\text{ad}, \sigma}(\Phi, R)$ if and only if there exists a $\theta$-invariant ideal $I$ of $R$ such that
    \begin{equation}\label{eq_m1}
        E_I \subseteq H \subseteq C_I = N_I.
    \end{equation}
\end{quote}

Let $J = \Rad(R)$ be the Jacobson radical of $R$ (which is also its unique maximal ideal), and let $k \coloneqq R/J$ denote the residue field. Since $J$ is $\theta$-invariant, the automorphism $\theta$ induces a well-defined automorphism of the field~$k$. Consequently, the groups $G_{\text{ad}, \sigma}(\Phi, k)$ and $E'_{\text{ad}, \sigma}(\Phi, k)$ are well-defined.


\begin{lemma}
    \normalfont    
    Let $N_J = E'_{\text{ad}, \sigma} (\Phi, R) \cap G_{\text{ad}, \sigma} (\Phi, J)$. Then 
    \begin{enumerate}[(a)]
        \item $N_J$ is the greatest proper normal subgroup of $E'_{\text{ad}, \sigma} (\Phi, R)$.
        \item $E'_{\text{ad}, \sigma} (\Phi, R)/N_J \cong E'_{\text{ad}, \sigma} (\Phi, k)$.
    \end{enumerate}
\end{lemma}

\begin{proof}
    Observe that $N_J$ is a proper normal subgroup of $E'_{\text{ad}, \sigma}(\Phi, R)$. Suppose $H$ is a proper normal subgroup of $E'_{\text{ad}, \sigma}(\Phi, R)$. By \eqref{eq_m1}, there exists a $\theta$-invariant proper ideal $I$ of $R$ such that 
    \[
        E_I \subseteq H \subseteq N_I.
    \]
    Since $J$ is the unique maximal ideal of $R$, we have $I \subseteq J$, which implies $H \subseteq N_I \subseteq N_J$. 
    Therefore, $N_J$ is the greatest proper normal subgroup of $E'_{\text{ad}, \sigma}(\Phi, R)$, which proves part~(a).
    
    Let $\lambda_J: E'_{\text{ad}, \sigma}(\Phi, R) \to E'_{\text{ad}, \sigma}(\Phi, k)$ be the natural surjection. Since $N_J=\ker(\lambda_J)$, we have
    \[
        E'_{\text{ad}, \sigma}(\Phi, R)/N_J \cong E'_{\text{ad}, \sigma}(\Phi, k).
    \]
    This proves part~(b).
\end{proof}


Let $\varphi \in \operatorname{Aut}(E'_{\ad, \sigma} (\Phi, R))$ be an arbitrary automorphism. 
By part (a) of the preceding lemma, the congruence subgroup $N_J$ is invariant under $\varphi$, and thus $\varphi$ induces an automorphism on the quotient:
\[
    \bar{\varphi}: E'_{\text{ad}, \sigma}(\Phi, R)/N_J \cong E'_{\text{ad}, \sigma} (\Phi, k) \longrightarrow E'_{\text{ad}, \sigma} (\Phi, k),
\]
where $k = R/J$ is the residue field. Since $k$ is a field, it follows from the classical theorem of Steinberg (see Theorem~\ref{thm:auto of E over field}) that the automorphism $\bar{\varphi}$ of $E'_{\text{ad}, \sigma}(\Phi, k)$ can be decomposed as a composition $\bar{\varphi} = i_g \circ f$, where $i_g$ is an inner automorphism induced by some $g \in G_{\text{ad}, \sigma}(\Phi, k)$, and $f$ is induced by a field automorphism of $k$ that commutes with the automorphism induced by $\theta$.

Since $g \in G_{\text{ad}, \sigma}(\Phi, k)$ and the natural map $\lambda_J: G_{\text{ad}, \sigma}(\Phi, R) \to G_{\text{ad}, \sigma}(\Phi, k)$ is surjective (see Lemma~\ref{lemma:Lamda_I_is_sur}), there exists $g_1 \in G_{\text{ad}, \sigma}(\Phi, R)$ such that $\lambda_J(g_1) = g$. 
Now consider the map $\varphi_1 := i_{g_1}^{-1} \circ \varphi$. 
This defines an automorphism of $E'_{\text{ad}, \sigma} (\Phi, R)$ such that the induced automorphism $\overline{\varphi}_1$ is induced by a field automorphism of $k$ that commutes with the automorphism induced by $\theta$.


\subsection{Reduction to key steps}

To prove Theorem~\ref{MT_local1}, we reduce the argument to the verification of several key steps, which will be established in the subsequent sections. Since the proofs for the odd- and even-rank cases differ substantially, we treat them separately. Before stating these steps, we introduce the notation
\[
    \GL_{\ell+1}(R,J) = \{ A \in \GL_{\ell+1}(R) \mid A \equiv I \pmod{J} \}.
\]

\smallskip

\noindent \textbf{The case $\mathbf{A_{2m-1} \ (m \geq 3)}$.}
We first consider the case of odd rank. In this situation, we proceed through the following sequence of steps.

\medskip

\noindent \textbf{Step (O-1).} Find an element $g_2 \in \GL_{\ell+1}(R, J)$ such that the map $\varphi_2 := i_{[g_2]^{-1}} \circ \varphi_1$ satisfies
\[
    \varphi_2 \left(h_{[\alpha]}(- 1)\right) = h_{[\alpha]}(- 1) \qquad \text{for all } [\alpha] \in \Phi_\rho.
\]

\noindent \textbf{Step (O-2).} Find an element $g_3 \in \GL_{\ell+1}(R, J)$ such that the map $\varphi_3 := i_{[g_3]^{-1}} \circ \varphi_2$ satisfies
\[
    \varphi_3\left(w_{[\alpha]}(1)\right) = w_{[\alpha]}(1) \qquad \text{for all } [\alpha] \in \Phi_\rho.
\]

\noindent \textbf{Step (O-3).} Find an element $g_4 \in \GL_{\ell+1}(R, J)$ such that the map $\varphi_4 := i_{[g_4]^{-1}} \circ \varphi_3$ satisfies
\[
    \varphi_4\left(x_{[\alpha]}(1)\right) = x_{[\alpha]}(1) \qquad \text{for all } [\alpha] \in \Phi_\rho.
\]

\smallskip

\noindent \textbf{The case $\mathbf{A_{2m} \ (m \geq 3)}$.} 
For the even rank case, we first define some special elements. 
For each $i = 1, \dots, m$, define the root $[\beta_i] = [\alpha_i + \cdots + \alpha_{2m+1-i}]$.
These roots $[\beta_i]$ $(i=1,\dots,m)$ are precisely all the positive roots in $\Phi_\rho$ of type $A_2$.

For each $[\beta_i] \in \Phi_\rho$, $a \in R^{\times}$, and $b \in R$, we define $v_i(a,b)=v_{[\beta_i]}(a,b)$ as the element of $\GL_{2m+1}(R)$ which acts on the subspace $\langle e_{i}, e_{m+1}, e_{2m+2-i}\rangle$ by
\[
    \begin{pmatrix}
        b & 0 & a \\
        0 & -1 & 0 \\
        a^{-1}(1-b^2) & 0 & -b
    \end{pmatrix},
\]
and acts as the identity on all other basis vectors. 
Then for all $i= 1, \dots, m$, we have
\[
    v_i(a,b)^2 = I 
    \qquad \text{and} \qquad
    w_{[\beta_i]}(1,1/2) = v_i(1/2,0).
\]

\medskip

\noindent \textbf{Step (E-1).} Find an element $g_2 \in \GL_{\ell+1}(R, J)$ such that the map $\varphi_2 := i_{[g_2]^{-1}} \circ \varphi_1$ satisfies
\[
    \varphi_2 \left(h_{[\alpha]}(- 1)\right) = h_{[\alpha]}(- 1) \qquad \text{for all } [\alpha] \in \Phi_\rho \text{ such that } [\alpha] \sim A_1^2.
\]

\noindent \textbf{Step (E-2).} Find an element $g_3 \in \GL_{\ell+1}(R, J)$ such that the map $\varphi_3 := i_{[g_3]^{-1}} \circ \varphi_2$ satisfies
\begin{gather*}
    \varphi_3\left(w_{[\alpha]}(1)\right) = w_{[\alpha]}(1) \qquad \text{for all } [\alpha] \in \Phi_\rho \text{ such that } [\alpha] \sim A_1^2, \\
    \varphi_3\left(w_{[\alpha]}(1,1/2)\right) = v_{[\alpha]}(a,b) \qquad \text{for all } [\alpha] \in \Phi_\rho \text{ such that } [\alpha] \sim A_2,
\end{gather*}
where $a \in R^{\times}$, $a \equiv 1/2 \pmod{J}$, and $b \in J$.

\smallskip

\noindent \textbf{Step (E-3).} Find an element $g_4 \in \GL_{\ell+1}(R, J)$ such that the map $\varphi_4 := i_{[g_4]^{-1}} \circ \varphi_3$ satisfies
\begin{gather*}
    \varphi_4\left(x_{[\alpha]}(1)\right) = x_{[\alpha]}(1) \qquad \text{for all } [\alpha] \in \Phi_\rho \text{ such that } [\alpha] \sim A_1^2, \\
    \varphi_4\left(w_{[\alpha]}(1,1/2)\right) = w_{[\alpha]}(1,1/2) \qquad \text{for all } [\alpha] \in \Phi_\rho \text{ such that } [\alpha] \sim A_2, \\
    \varphi_4\left(x_{[\alpha]}(1, 1/2)\right) = x_{[\alpha]}(1, 1/2) \qquad \text{for all } [\alpha] \in \Phi_\rho \text{ such that } [\alpha] \sim A_2.
\end{gather*}

\medskip

For the moment, we assume the validity of each of the above steps and proceed with the proof of Theorem~\ref{MT_local1}. 
Set $C_1 := g_2 g_3 g_4 \in \GL_{\ell + 1}(R, J)$. 
Then 
\[
    \varphi_1 = i_{[C_1]} \circ \varphi_4.
\]
We claim that $[C_1] \in G_{\ad, \sigma} (\Phi, R)$.

For convenience, set
\[
    u_{[\alpha]}=
    \begin{cases}
        x_{[\alpha]}(1), & [\alpha]\not\sim A_2,\\[1mm]
        x_{[\alpha]}(1,1/2), & [\alpha]\sim A_2.
    \end{cases}
\]
By Steps (O-3) and (E-3),
\[
    \varphi_4(u_{[\alpha]})=u_{[\alpha]}
    \qquad
    ([\alpha]\in\Phi_\rho).
\]
Hence
\[
    \varphi_1(u_{[\alpha]})
    =
    (i_{[C_1]}\circ\varphi_4)(u_{[\alpha]})
    =
    [C_1u_{[\alpha]}C_1^{-1}].
\]

Since $\varphi_1$ is an automorphism of
$E'_{\ad,\sigma}(\Phi,R)$, it follows that
\[
    [C_1u_{[\alpha]}C_1^{-1}]
    \in
    E'_{\ad,\sigma}(\Phi,R)
\]
for every root $[\alpha]\in\Phi_\rho$. Therefore this element is fixed by the involution $\sigma$, and hence there exists a unit
$\lambda_{[\alpha]}\in R^{\times}$ such that
\[
    Q(C_1u_{[\alpha]}C_1^{-1})Q^{-1}
    =
    \lambda_{[\alpha]}
    \,\overline{(C_1u_{[\alpha]}C_1^{-1})}^{\, -t}.
\]

On the other hand, since $u_{[\alpha]}\in E'_{\mathrm{sc},\sigma}(\Phi,R)$, we have
\[
    Qu_{[\alpha]}Q^{-1}
    =
    \overline{u_{[\alpha]}}^{-t}.
\]
Substituting this identity into the previous relation yields
\[
    D u_{[\alpha]}
    =
    \lambda_{[\alpha]}u_{[\alpha]}D,
\]
where
\[
    D:= Q^{-1} \overline{C_1}^{\,t} Q C_1.
\]

Since $C_1\equiv I\pmod J$, we also have $D \equiv I \pmod{J}$. Consequently,
\[
    \lambda_{[\alpha]} \equiv 1 \pmod{J}.
\]

We now prove that
\[
    \lambda_{[\alpha]} = 1 \qquad \text{for all }[\alpha]\in\Phi_\rho.
\]

Let
\[
    \Gamma := \langle u_{[\alpha]} \mid [\alpha]\in\Phi_\rho \rangle.
\]
The relation above defines a group homomorphism
\[
    \chi:\Gamma \longrightarrow R^\times
\]
by
\[
    D \, g \, D^{-1} = \chi(g) \, g, \qquad g \in \Gamma.
\]
In particular,
\[
    \chi(u_{[\alpha]}) = \lambda_{[\alpha]}.
\]
Therefore $\chi$ is trivial on the commutator subgroup $[\Gamma,\Gamma]$. 

A direct rank-two application of the Chevalley commutator formulas
(see~\cite{SG&DM1}) shows that $u_{[\alpha]}\in[\Gamma,\Gamma]$ for short roots of $C_m$ and long roots of $B_m$, while $u_{[\alpha]}^2\in[\Gamma,\Gamma]$ for the remaining roots. It follows that
\[
    \lambda_{[\alpha]} = 1
    \qquad\text{or}\qquad
    \lambda_{[\alpha]}^2 = 1.
\]

Suppose that $\lambda_{[\alpha]}^2=1$. Since $\lambda_{[\alpha]}\equiv 1\pmod J$, we may write
\[
    \lambda_{[\alpha]} = 1 + a, \qquad a \in J.
\]
Then
\[
    1 = (1+a)^2 = 1+2a+a^2,
\]
and therefore
$a(2+a)=0$.

Because $2\in R^\times$ and $a\in J$, the element $2+a$ is invertible. Hence $a=0$, and so
$\lambda_{[\alpha]}=1$.

Thus
\[
    \lambda_{[\alpha]}=1 \qquad \text{for all } [\alpha]\in\Phi_\rho.
\]
Consequently,
\[
    D u_{[\alpha]}=u_{[\alpha]}D \qquad \text{for all } [\alpha]\in\Phi_\rho.
\]
Since the matrices \(u_{[\alpha]}\) generate \(M_{\ell+1}(R)\) as a unital \(R\)-algebra (cf. Lemma~\ref{lemma:M(R)_as_R_algebra}), the matrix \(D\) commutes with the whole algebra \(M_{\ell+1}(R)\). 
Hence
$D = \lambda I$
for some \(\lambda\in R^{\times}\). 
That is,
\[
    Q^{-1}\overline{C_1}^{\,t}QC_1=\lambda I.
\]
Equivalently,
\[
    QC_1Q^{-1}=\lambda\,\overline{C_1}^{\, -t}.
\]
Thus the projective class \([C_1]\) is fixed by the involution \(\sigma\). Since
\(C_1\in\GL_{\ell+1}(R)\), we obtain
\[
    [C_1]\in G_{\mathrm{ad},\sigma}(\Phi,R)\subset \PGL_{\ell+1}(R).
\]
This proves the claim. 

As a consequence, the map
$\varphi_4=i_{[C_1]}^{-1}\circ\varphi_1$
is an automorphism of the group $E'_{\ad,\sigma}(\Phi,R)$.

\medskip

\noindent \textbf{Step (C).} Show that there exists an automorphism $\mu$ of the ring $R$ such that $\mu \circ \theta = \theta \circ \mu$ and
\[
    \varphi_4 (x_{[\alpha]}(t)) = x_{[\alpha]}(\mu(t)) \qquad \text{for all } [\alpha] \in \Phi_\rho \text{ and } t \in R_{[\alpha]}.
\]

\begin{rmk}
    Recall that, if $[\alpha] \sim A_2$, then $R_{[\alpha]} = \mathcal{A}(R)$. 
    For $t = (t_1, t_2) \in \mathcal{A}(R)$, we write 
    \[
        \mu (t) = (\mu(t_1), \mu(t_2)).
    \]
\end{rmk}

Assuming Step~(C) holds, it follows that the map $\varphi_4$ is induced by an automorphism $\mu$ of the ring $R$; in accordance with standard notation, we also denote $\varphi_4$ by $\mu$.

Therefore, we obtain
\[
    \varphi = i_{g_1} \circ \varphi_1 = i_{g_1} \circ i_{[C_1]} \circ \mu.
\]
Setting $C = g_1 [C_1]$, we obtain
$\varphi = i_C \circ \mu$,
where $C \in G_{\ad, \sigma} (\Phi, R)$ and $\mu$ is a ring automorphism.
This completes the proof of Theorem~\ref{MT_local1}.

\medskip

It remains to establish the validity of the above steps, which will be carried out in the subsequent sections. 
In particular, Sections~\ref{sec:image_of_h(-1)}, \ref{sec:image_of_w(1)}, and \ref{sec:image_of_x(1)} contain the proofs of Steps (O-1)/(E-1), (O-2)/(E-2), and (O-3)/(E-3), respectively, while Section~\ref{sec:image_of_x(t)} is devoted to the common Step~(C).

\subsection{A lifting lemma}

Before concluding this section, we introduce a lemma that will be used in Sections~\ref{sec:image_of_h(-1)}--\ref{sec:image_of_x(t)}. The common first step in all of these sections is to carefully lift our setup from $\PGL_{\ell+1}(R)$ to $\GL_{\ell+1}(R)$, allowing us to perform the necessary computations within $\GL_{\ell+1}(R)$. The following lemma facilitates this process.

First, consider the commutative diagram
\[
    \begin{tikzcd}[row sep=large, column sep=large]
        E'_{\mathrm{sc},\sigma}(\Phi,R) \arrow[r, "\delta_R"] \arrow[d, "\lambda_{\mathrm{sc}}"'] & E'_{\mathrm{ad},\sigma}(\Phi,R) \arrow[d, "\lambda_{\mathrm{ad}}"] \\
        E'_{\mathrm{sc},\sigma}(\Phi,k) \arrow[r, "\delta_k"'] & E'_{\mathrm{ad},\sigma}(\Phi,k)
    \end{tikzcd}
\]
where the horizontal maps $\delta_R$ and $\delta_k$ are the natural central quotient homomorphisms, and the vertical maps $\lambda_{\mathrm{sc}}$ and $\lambda_{\mathrm{ad}}$ are the reduction homomorphisms modulo the Jacobson radical $J$.

\begin{lemma}\label{lemma:compatible-representatives}
    Let $\varphi_1 \in \operatorname{Aut}(E'_{\mathrm{ad}, \sigma}(\Phi, R))$ be as above; that is, the induced automorphism $\overline{\varphi}_1$ of $E'_{\mathrm{ad}, \sigma}(\Phi, k)$ is obtained by applying a field automorphism $f \colon k \to k$ entrywise, where $f$ commutes with the automorphism of $k$ induced by $\theta$. 
    Denote by $f_{\mathrm{ad}}$ and $f_{\mathrm{sc}}$ the automorphisms of $E'_{\mathrm{ad}, \sigma}(\Phi, k)$ and $E'_{\mathrm{sc}, \sigma}(\Phi, k)$, respectively, induced by the field automorphism $f$.
    In particular, $f_{\mathrm{ad}} = \overline{\varphi}_1$.

    For every $x \in E'_{\mathrm{sc}, \sigma}(\Phi, R)$, there exists an element $y \in E'_{\mathrm{sc}, \sigma}(\Phi, R)$ such that 
    \[
        \delta_R(y) = \varphi_1\bigl(\delta_R(x)\bigr)
        \qquad \text{and} \qquad
        \lambda_{\mathrm{sc}}(y) = f_{\mathrm{sc}} \bigl(\lambda_{\mathrm{sc}}(x)\bigr).
    \]
\end{lemma}

\begin{proof}
    Observe that we have the following commutative diagram:
    \[
        \begin{tikzcd}[row sep=large, column sep=large]
            E'_{\mathrm{sc},\sigma}(\Phi,R) \arrow[r, "\delta_R"] \arrow[d, "\lambda_{\mathrm{sc}}"'] & E'_{\mathrm{ad},\sigma}(\Phi,R) \arrow[r, "\varphi_1"] \arrow[d, "\lambda_{\mathrm{ad}}"] & E'_{\mathrm{ad},\sigma}(\Phi,R) \arrow[d, "\lambda_{\mathrm{ad}}"] \\
            E'_{\mathrm{sc},\sigma}(\Phi,k) \arrow[r, "\delta_k"'] & E'_{\mathrm{ad},\sigma}(\Phi,k) \arrow[r, "\overline{\varphi}_1 = f_{\mathrm{ad}}"'] & E'_{\mathrm{ad},\sigma}(\Phi,k)
        \end{tikzcd}
    \]
    By Lemma~\ref{lemma:E_ad=E_pi/Z} and the proof of Lemma~\ref{lemma:Lamda_I_is_sur}, every map in the above diagram is surjective. Moreover, note that 
    \[
        f_{\mathrm{ad}} \circ \delta_k = \delta_k \circ f_{\mathrm{sc}}.
    \]

    Since the group $E'_{\mathrm{sc},\sigma}(\Phi,R)$ is perfect, we can express $x$ as a product of commutators:
    \[
        x = \prod_{s=1}^{N} [a_s, b_s]
    \]
    for some elements $a_s, b_s \in E'_{\mathrm{sc},\sigma}(\Phi,R)$.

    Since the homomorphism $\delta_R$ is surjective, for each index $s$ we may choose elements $A_s, B_s \in E'_{\mathrm{sc},\sigma}(\Phi,R)$ such that
    \[
        \delta_R(A_s) = \varphi_1\bigl(\delta_R(a_s)\bigr)
        \quad \text{and} \quad
        \delta_R(B_s) = \varphi_1\bigl(\delta_R(b_s)\bigr).
    \]
    Define
    \[
        y \coloneqq \prod_{s=1}^{N} [A_s, B_s].
    \]
    It immediately follows that
    \[
        \delta_R(y) = \prod_{s=1}^{N} \bigl[ \varphi_1\bigl(\delta_R(a_s)\bigr), \varphi_1\bigl(\delta_R(b_s)\bigr) \bigr] = \varphi_1\bigl(\delta_R(x)\bigr).
    \]

    Next, for each $s$, the commutativity of the diagram yields
    \begin{align*}
        \delta_k \bigl(\lambda_{\mathrm{sc}}(A_s)\bigr)
        = \lambda_{\mathrm{ad}}\bigl(\delta_R(A_s)\bigr) 
        = \lambda_{\mathrm{ad}}\bigl(\varphi_1\bigl(\delta_R(a_s)\bigr)\bigr) 
        = f_{\mathrm{ad}}\bigl(\delta_k\bigl(\lambda_{\mathrm{sc}}(a_s)\bigr)\bigr) 
        = \delta_k\bigl(f_{\mathrm{sc}}\bigl(\lambda_{\mathrm{sc}}(a_s)\bigr)\bigr).
    \end{align*}
    This implies that
    \[
        \lambda_{\mathrm{sc}}(A_s) = z_s \, f_{\mathrm{sc}}\bigl(\lambda_{\mathrm{sc}}(a_s)\bigr)
    \]
    for some central element $z_s \in \ker(\delta_k) = Z\bigl(E'_{\mathrm{sc},\sigma}(\Phi,k)\bigr)$.
    Similarly,
    \[
        \lambda_{\mathrm{sc}}(B_s) = z'_s \, f_{\mathrm{sc}}\bigl(\lambda_{\mathrm{sc}}(b_s)\bigr)
    \]
    for some central element $z'_s \in Z\bigl(E'_{\mathrm{sc},\sigma}(\Phi,k)\bigr)$.

    Because central factors do not affect commutators, we obtain
    \begin{multline*}
        \lambda_{\mathrm{sc}}(y)
        = \prod_{s=1}^{N} \bigl[ \lambda_{\mathrm{sc}}(A_s), \lambda_{\mathrm{sc}}(B_s) \bigr] 
        = \prod_{s=1}^{N} \bigl[ f_{\mathrm{sc}}\bigl(\lambda_{\mathrm{sc}}(a_s)\bigr), f_{\mathrm{sc}}\bigl(\lambda_{\mathrm{sc}}(b_s)\bigr) \bigr] \\
        = f_{\mathrm{sc}}\Biggl( \lambda_{\mathrm{sc}}\Biggl( \prod_{s=1}^{N} [a_s, b_s] \Biggr) \Biggr) 
        = f_{\mathrm{sc}}\bigl(\lambda_{\mathrm{sc}}(x)\bigr).
    \end{multline*}
    This completes the proof.
\end{proof}

\begin{rmk}
    While Lemma~\ref{lemma:compatible-representatives} will be used directly in Section~\ref{sec:image_of_h(-1)}, for the remaining sections (Sections~\ref{sec:image_of_w(1)} through \ref{sec:image_of_x(t)}), we will rely on the following immediate consequence of this lemma.
    
    Let
    \[
        \lambda_J \colon \GL_{\ell+1}(R) \to \GL_{\ell+1}(k)
    \]
    denote the entrywise reduction modulo $J$. Let $C \in \GL_{\ell+1}(R, J)$ be an element of the principal congruence subgroup modulo $J$, and define 
    \[
        \psi \coloneqq i_{[C]}^{-1} \circ \varphi_1.
    \]
    Observe that $\psi$ may not be an automorphism of $E'_{\mathrm{ad}, \sigma}(\Phi, R)$; rather, it is an isomorphism from $E'_{\mathrm{ad}, \sigma}(\Phi, R)$ onto a subgroup of $\PGL_{\ell+1}(R)$.
    Nevertheless, the induced map $\overline{\psi}$ is an automorphism of $E'_{\mathrm{ad}, \sigma}(\Phi, k)$ satisfying $\overline{\psi} = \overline{\varphi}_1$.

    If $y$ is chosen as in Lemma~\ref{lemma:compatible-representatives}, then the element
    \[
        y':= C^{-1} \, y \, C \in \SL_{\ell+1}(R)
    \]
    serves as a representative for $\psi\bigl(\delta_R(x)\bigr)$. Furthermore, because $\lambda_J(C) = I_{\ell+1}$, we obtain
    \[
        \lambda_J (y') = f_{\mathrm{sc}}\bigl(\lambda_{\mathrm{sc}}(x)\bigr).
    \]
\end{rmk}


\section{The Images of \texorpdfstring{$h_{[\alpha]}(-1)$}{h(-1)}}\label{sec:image_of_h(-1)}


Fix a simple system $\Delta = \{\alpha_1, \dots, \alpha_\ell\}$ of type $A_\ell$ with $\ell \geq 4$, as specified in Section~\ref{sec:G(R)_and_MT}. Let $\Delta_\rho = \{[\alpha_1], \dots, [\alpha_m]\}$ be the corresponding simple system of the associated twisted root system, where $m = \lfloor (\ell+1)/2 \rfloor$. 

First suppose that $\ell=2m-1$, i.e., that $\ell$ is odd. 
Then each simple root $[\alpha_i]$ $(i=1,\dots,m-1)$ is of type $A_1^2$, while $[\alpha_m]$ is of type $A_1$. 
Consequently, the elements $h_{[\alpha_i]}(-1)$ belong to the elementary subgroup $E'_{\mathrm{ad},\sigma}(\Phi,R)$ for all $i=1,\dots,m$.

Now assume that $\ell=2m$, i.e., that $\ell$ is even. 
In this case, each root $[\alpha_i]$ $(i=1,\dots,m-1)$ is of type $A_1^2$, whereas $[\alpha_m]$ is of type $A_2$. 
It follows that $h_{[\alpha_i]}(-1)\in E'_{\mathrm{ad},\sigma}(\Phi,R)$ for $i=1, \dots, m-1$. 
However, the element $h_{[\alpha_m]}(-1)$ need not belong to $E'_{\mathrm{ad},\sigma}(\Phi,R)$.
Therefore, in the even-rank case, we must restrict our attention to the simple roots of type $A_1^2$.

For each simple root $[\alpha_i]\in \Delta_\rho$ of type $A_1$ or $A_1^2$, we consider the element
\[
    H_i:=\varphi_1\bigl(h_{[\alpha_i]}(-1)\bigr)
    \in E'_{\ad,\sigma}(\Phi,R)\subset \PGL_{\ell+1}(R).
\]
We claim that there exists a representative $h_i \in \SL_{\ell+1}(R)$ of $H_i$ such that 
\[
    h_i \equiv h_{[\alpha_i]}(-1) \pmod{J}, \qquad h_i^2 = I, \qquad h_i h_j = h_j h_i.
\]
As noted in Section~\ref{sec:G(R)_and_MT}, the symbol $h_{[\alpha_i]}(-1) \in E'_{\text{sc}, \sigma}(\Phi, R)$ denotes the standard lift of the corresponding element in $E'_{\text{ad}, \sigma}(\Phi, R)$. 
Although we utilize the same notation for both, the distinction should be clear from the context. 

To avoid unnecessary repetition, we treat the odd- and even-rank cases simultaneously. 
In the even-rank case, the element $h_m$ is not taken into account. 
Accordingly, we set
\[
    m_0 :=
    \begin{cases}
        m, & \text{if } \ell=2m-1,\\
        m-1, & \text{if } \ell=2m.
    \end{cases}
\]
Thus the elements $h_i$ considered below are precisely $h_1, \ldots, h_{m_0}$.


\subsection{Choice of representatives for \texorpdfstring{$H_i$}{H-i}}

Our goal in this section is to choose representatives $h_i$ of $H_i$ satisfying the relations stated above.

Applying Lemma~\ref{lemma:compatible-representatives} to the standard lift
\[
    h_{[\alpha_i]}(-1) \in E'_{\mathrm{sc},\sigma}(\Phi,R),
\]
and using the fact that $f(-1) = -1$, we obtain a representative $h_i \in \SL_{\ell+1}(R)$ of $H_i$ satisfying
\[
    h_i \equiv h_{[\alpha_i]}(-1) \pmod{J}.
\]

Since $H_i^2 = I$, we have $h_i^{2} = \lambda_i I$ for some $\lambda_i \in R^{\times}$.
Taking determinants and reducing modulo $J$ yield $\lambda_i^{\ell + 1} = 1$ and $\lambda_i\equiv1\pmod J$.

We now show that the order of $\lambda_i$ is odd. If $\ell$ is even, this is immediate, since $o(\lambda_i)$ divides the odd integer $\ell+1$. If $\ell$ is odd, write $\ell+1 = 2^k s$ with $s$ odd. Then $(\lambda_i^s)^{2^k} = 1$. Since $R$ is local, $2 \in R^{\times}$, and $\lambda_i \equiv 1 \pmod{J}$, the only $2^k$-th root of unity congruent to $1$ modulo $J$ is the identity. Thus $\lambda_i^s = 1$, and $o(\lambda_i)$ is odd in this case as well.

Let $K_i \subset R^{\times}$ be the cyclic subgroup generated by $\lambda_i$. Since $|K_i|$ is odd, the squaring map $x \mapsto x^2$ is an automorphism of $K_i$. Hence there exists a unique $\eta_i \in K_i$ such that $\eta_i^2 = \lambda_i^{-1}$. Since $\eta_i^{\ell+1}=1$ and $\eta_i\equiv1\pmod J$, replacing $h_i$ by $\eta_i h_i$ preserves the determinant and the congruence, so we may assume that
\[
    h_i^2 = I \qquad (1 \leq i \leq m_0).
\]

Next, since the elements $H_i$ and $H_j$ commute in $E'_{\text{ad}, \sigma}(\Phi, R)$, their representatives satisfy
\[
    h_i h_j = \mu_{ij} \, h_j h_i
\]
for some scalar $\mu_{ij} \in R^{\times}$ and all $1 \leq i, j \leq m_0$. Because these matrices are congruent modulo $J$ to commuting diagonal matrices, it follows that 
\[
    \mu_{ij} \equiv 1 \pmod{J}.
\]
Furthermore, the relations $h_i^2 = h_j^2 = I$ imply that $\mu_{ij}^2 = 1$. 
Since $R$ is a local ring with $2 \in R^{\times}$, we must have $\mu_{ij} = 1$. 
Consequently, the representatives commute:
\[
    h_i h_j = h_j h_i \qquad (1 \leq i, j \leq m_0).
\]


\subsection{Normalization of \texorpdfstring{$h_i$}{h-i}}

Fix the standard basis $\{e_1, \ldots, e_{\ell+1}\}$ of $R^{\ell+1}$. 
With respect to this basis, the matrices $h_{[\alpha_i]}(-1) \in E'_{\text{sc},\sigma}(\Phi,R)$, for $1 \leq i \leq m_0$, have the form
\[
    h_{[\alpha_i]}(-1)=\operatorname{diag}[\pm1,\ldots,\pm1].
\]
In other words, for each $1 \leq i \leq m_0$ and $1 \leq j \leq \ell+1$, we have
\[
    h_{[\alpha_i]}(-1)(e_j)=\varepsilon_{ij}e_j,
    \qquad
    \varepsilon_{ij} \in \{1,-1\}.
\]

For all $j \in \{1,\ldots,\ell+1\}$, set $e_j^{(0)} := e_j$. 
For $k=1,\ldots,m_0$, define inductively
\[
    e_j^{(k)}
    :=
    \frac{e_j^{(k-1)}+\varepsilon_{kj}h_k(e_j^{(k-1)})}{2}
    \qquad
    (1\le j\le \ell+1).
\]
Since \(h_k\equiv h_{[\alpha_k]}(-1)\pmod J\), we have
\[
    e_j^{(k)}\equiv e_j^{(k-1)}\pmod J.
\]
Thus, at every step, \(\{e_1^{(k)},\ldots,e_{\ell+1}^{(k)}\}\) is a basis of \(R^{\ell+1}\), and the corresponding change-of-basis matrix lies in \(\GL_{\ell+1}(R,J)\).

We claim that for each \(k=1,\ldots,m_0\), with respect to the basis $\{e_1^{(k)},\ldots,e_{\ell+1}^{(k)}\}$, all elements \(h_i\), \(1\le i\le k\), act on \(e_j^{(k)}\) by multiplication by \(\varepsilon_{ij}\).
Equivalently, their action on this basis coincides with the action of \(h_{[\alpha_i]}(-1)\), \(1\le i\le k\), on the original basis \(\{e_1,\ldots,e_{\ell+1}\}\).

We prove this claim by induction on \(k\). First, the involution \(h_k\) acts on \(e_j^{(k)}\)
by multiplication by \(\varepsilon_{kj}\). Indeed,
\begin{multline*}
    h_k(e_j^{(k)})
    =
    \frac{h_k(e_j^{(k-1)})+\varepsilon_{kj}h_k^2(e_j^{(k-1)})}{2}
    =
    \frac{h_k(e_j^{(k-1)})+\varepsilon_{kj}e_j^{(k-1)}}{2}
    =\\
    \varepsilon_{kj}
    \frac{e_j^{(k-1)}+\varepsilon_{kj}h_k(e_j^{(k-1)})}{2}
    =
    \varepsilon_{kj}e_j^{(k)}.
\end{multline*}

Now let \(i<k\), and assume by induction that
\[
    h_i(e_j^{(k-1)})=\varepsilon_{ij}e_j^{(k-1)}.
\]
Since the representatives \(h_i\) commute with each other, we obtain
\begin{multline*}
    h_i(e_j^{(k)})
    =
    \frac{h_i(e_j^{(k-1)})+\varepsilon_{kj}h_i h_k(e_j^{(k-1)})}{2}
    =
    \frac{\varepsilon_{ij}e_j^{(k-1)}
    +\varepsilon_{kj}h_k h_i(e_j^{(k-1)})}{2}
    =\\
    \varepsilon_{ij}
    \frac{e_j^{(k-1)}+\varepsilon_{kj}h_k(e_j^{(k-1)})}{2}
    =
    \varepsilon_{ij}e_j^{(k)}.
\end{multline*}

This completes the induction.

Thus, for each \(1\le i\le m_0\), the action of \(h_i\) on the basis $\{e_1^{(m_0)},\ldots,e_{\ell+1}^{(m_0)}\}$ coincides with the action of \(h_{[\alpha_i]}(-1)\) on the original basis
\(\{e_1,\ldots,e_{\ell+1}\}\).

\smallskip

Let \(g_2\in \GL_{\ell+1}(R)\) be the matrix whose columns are $e_1^{(m_0)},\ldots,e_{\ell+1}^{(m_0)}$. By construction, we have $g_2 \in \GL_{\ell+1}(R,J)$.
Set
\[
    \varphi_2 := i_{[g_2]^{-1}} \circ \varphi_1.
\]
Then $\varphi_2$ is an isomorphism of $E'_{\ad, \sigma} (\Phi, R)$ onto a subgroup of $\PGL_{\ell+1} (R)$ and satisfies
\[
    \varphi_2\bigl(h_{[\alpha_i]}(-1)\bigr) = h_{[\alpha_i]}(-1)
    \qquad \text{for all } i \in \{ 1, \dots, m_0 \} .
\]

By the standard torus relations, if $\ell = 2m-1$ is odd, then $m_0 = m$, and therefore
\[
    \varphi_2 \bigl(h_{[\alpha]}(-1)\bigr) = h_{[\alpha]}(-1)
    \qquad
    \text{for every root } [\alpha] \in \Phi_\rho.
\]
Likewise, if $\ell = 2m$ is even, then $m_0 = m-1$, and therefore
\[
    \varphi_2 \bigl(h_{[\alpha]}(-1)\bigr) = h_{[\alpha]}(-1)
    \qquad
    \text{for every root } [\alpha]\in\Phi_\rho \text{ of type } A_1^2.
\]

\section{The Images of \texorpdfstring{$w_{[\alpha]}(1)$}{w(1)}}\label{sec:image_of_w(1)}


For each root $[\alpha] \in \Phi_\rho$, define
\[
    W_{[\alpha]} :=
    \begin{cases}
        \varphi_{2}\bigl(w_{[\alpha]}(1)\bigr) & \text{if } [\alpha] \sim A_1 \text{ or } A_1^2, \\
        \varphi_{2}\bigl(w_{[\alpha]}(1,1/2)\bigr) & \text{if } [\alpha] \sim A_2.
    \end{cases}
\]

We claim that $W_{[\alpha]}$ admits a representative $w_{[\alpha]} \in \SL_{\ell+1}(R)$ such that
\begin{gather*}
    w_{[\alpha]} \equiv 
    \begin{cases}
        w_{[\alpha]}(1) \pmod{J} & \text{if } [\alpha] \sim A_1 \text{ or } A_1^2, \\
        w_{[\alpha]}(1, 1/2) \pmod{J} & \text{if } [\alpha] \sim A_2,
    \end{cases}
    \\[4pt]
    (w_{[\alpha]})^2 =
    \begin{cases}
        h_{[\alpha]}(-1) & \text{if } [\alpha] \sim A_1 \text{ or } A_1^2, \\
        I & \text{if } [\alpha] \sim A_2.
    \end{cases}
\end{gather*}


\subsection{Choice of representatives for \texorpdfstring{$W_{[\alpha]}$}{W}}

Since $\varphi_2 = i_{[g_2]}^{-1} \circ \varphi_1$ and $g_2 \in \GL_{\ell+1}(R, J)$, applying the remark following Lemma~\ref{lemma:compatible-representatives} to the standard simply connected representative, and utilizing the identities $f(1) = 1$ and $f(1/2) = 1/2$, yields a representative $w_{[\alpha]} \in \SL_{\ell+1}(R)$ of $W_{[\alpha]}$ such that
\[
    w_{[\alpha]} \equiv
    \begin{cases}
        w_{[\alpha]}(1) \pmod{J}, & [\alpha] \sim A_1 \text{ or } A_1^2, \\[2mm]
        w_{[\alpha]}(1, 1/2) \pmod{J}, & [\alpha] \sim A_2.
    \end{cases}
\]

Since, in $\PGL_{\ell + 1} (R)$, we have
\[
    (W_{[\alpha]})^2 = \begin{cases}
        h_{[\alpha]}(-1) & \text{if } [\alpha] \sim A_1 \text{ or } A_1^2, \\
        I & \text{if } [\alpha] \sim A_2,
    \end{cases}
\]
it follows that
\[
    (w_{[\alpha]})^2 = \begin{cases}
        \lambda_{[\alpha]} \, h_{[\alpha]}(-1) & \text{if } [\alpha] \sim A_1 \text{ or } A_1^2, \\
        \lambda_{[\alpha]} \, I & \text{if } [\alpha] \sim A_2,
    \end{cases}
\]
for some $\lambda_{[\alpha]} \in R^{\times}$. 

Taking determinants yields
\[
    \lambda_{[\alpha]}^{\ell+1} = 1,
\]
and the congruence modulo $J$ implies
\[
    \lambda_{[\alpha]} \equiv1 \pmod{J}.
\]

As in Section~\ref{sec:image_of_h(-1)}, the $2$-primary part of the order of $\lambda_{[\alpha]}$ is trivial, since $2 \in R^{\times}$ and $\lambda_{[\alpha]} \equiv 1 \pmod{J}$. 
Hence the squaring map is an automorphism of the cyclic group generated by $\lambda_{[\alpha]}$. 
Therefore there exists $\eta_{[\alpha]} \in \langle \lambda_{[\alpha]} \rangle$ such that
\[
    \eta_{[\alpha]}^2 = \lambda_{[\alpha]}^{-1}.
\]
Since $\eta_{[\alpha]}^{\ell+1}=1$ and $\eta_{[\alpha]}\equiv1\pmod J$, replacing $w_{[\alpha]}$ by $\eta_{[\alpha]} w_{[\alpha]}$ gives a representative satisfying
\[
    (w_{[\alpha]})^2 =
    \begin{cases}
        h_{[\alpha]}(-1) & \text{if } [\alpha] \sim A_1 \text{ or } A_1^2, \\
        I & \text{if } [\alpha] \sim A_2,
    \end{cases}
\]
as required.


\medskip

We now proceed by considering the odd- and even-rank cases separately.


\subsection{The case \texorpdfstring{${}^2 A_{2m - 1} \ (m \geq 3)$}{2A(2m-1)}}\label{subsec:image_of_w_for_odd_case}

Assume that $(\Phi,\rho)$ has twisted type ${}^2 A_{2m-1}$ with $m \geq 3$, and let $\Delta_\rho = \{ [\alpha_1], \dots, [\alpha_m] \}$ be the simple system of $\Phi_\rho$. As noted in Section~\ref{sec:preliminaries}, $\Phi_\rho$ is of type $C_m$. We realize this root system in the Euclidean space $\mathbb{R}^{m}$ as
\[
    \Phi_\rho = \{ \pm \nu_i \pm \nu_j \mid 1 \leq i < j \leq m \} \cup \{ \pm 2\nu_i \mid 1 \leq i \leq m \},
\]
where $\{\nu_1,\dots,\nu_m\}$ denotes the standard orthonormal basis of $\mathbb{R}^m$.

Under this identification, the simple roots are given by 
\[
    [\alpha_i] = \nu_i - \nu_{i+1} \quad (i=1,\dots,m-1), \qquad [\alpha_m] = 2 \nu_m.
\]
In addition, for each $i=1,\dots,m-1$, we define
\[
    [\widetilde{\alpha_i}] = \nu_i+\nu_{i+1}.
\]
Observe that the roots $[\alpha_i]$ and $[\widetilde{\alpha_i}]$ are orthogonal for every $i=1,\dots,m-1$.

We also introduce the roots $[\beta_i]\in\Phi_\rho$ $(i=1,\dots,m)$ corresponding to the vectors $2 \nu_i$. 
The set $\{[\beta_1],\dots,[\beta_m]\}$ coincides with the set of all positive long roots of $\Phi_\rho$. 
Moreover, each root $[\beta_i]$ is of type $A_1$ and is given explicitly by
\[
    [\beta_i] = [\alpha_i+\cdots+\alpha_{2m-i}].
\]

\subsubsection{\textbf{Normalization of $w_{[\alpha]}$ for $[\alpha]\sim A_1$}}

As we just noted, the roots $\{[\beta_1], \dots, [\beta_m]\}$ are precisely the positive roots of type $A_1$ in $\Phi_\rho$. Our goal in this subsection is to normalize the corresponding elements $w_{[\beta_i]}$. We begin by determining their initial form.

Let $\{e_1, \dots, e_{2m}\}$ be the standard basis of $R^{2m}$, and define
\[
    B_r := \langle e_r, e_{2m+1-r} \rangle \qquad (1 \leq r \leq m).
\]

Fix $i \in \{1, \dots, m\}$. 
Since $w_{[\beta_i]}(1)$ commutes with $h_{[\beta_i]}(-1)$ in $\PGL_{2m}(R)$, the corresponding representatives in $\GL_{2m}(R)$ satisfy
\[
    w_{[\beta_i]} \, h_{[\beta_i]}(-1)
    =
    \lambda_i\, h_{[\beta_i]}(-1)\, w_{[\beta_i]}
\]
for some $\lambda_i\in R^\times$. 
Reducing modulo $J$ immediately yields $\lambda_i \equiv 1 \pmod{J}$.

We claim that $\lambda_i=1$. Indeed, rewriting the above relation as
\[
    w_{[\beta_i]}
    =
    \lambda_i \, h_{[\beta_i]}(-1) \, w_{[\beta_i]} \, h_{[\beta_i]}(-1)^{-1}
\]
and squaring both sides yields
\[
    w_{[\beta_i]}^2
    =
    \lambda_i^2 \,
    h_{[\beta_i]}(-1) \, w_{[\beta_i]}^2 \, h_{[\beta_i]}(-1)^{-1}.
\]
Since our representatives have been chosen so that
\[
    w_{[\beta_i]}^2 = h_{[\beta_i]}(-1),
\]
it follows that
\[
    h_{[\beta_i]}(-1)
    =
    \lambda_i^2 h_{[\beta_i]}(-1),
\]
and therefore
$\lambda_i^2=1$.

Since $R$ is a local ring with $2\in R^\times$ and $\lambda_i\equiv 1\pmod J$, the equality $\lambda_i^2=1$ implies that $\lambda_i=1$. Consequently,
\[
    h_{[\beta_i]}(-1) \, w_{[\beta_i]}
    =
    w_{[\beta_i]} \, h_{[\beta_i]}(-1)
    \qquad (1 \leq i \leq m).
\]

A straightforward computation using this commutation relation shows that $w_{[\beta_i]}$ preserves both the subspace $B_i$ and its complement
\[
    B' := B_1\oplus\cdots\oplus B_{i-1}
    \oplus B_{i+1}\oplus\cdots\oplus B_m.
\]

Moreover,
\[
    w_{[\beta_i]}^2 = h_{[\beta_i]}(-1),
\]
and $h_{[\beta_i]}(-1)$ acts trivially on $B'$. Since $w_{[\beta_i]}\equiv I\pmod J$ on $B'$, the usual argument shows that
\[
    w_{[\beta_i]}|_{B'} = I.
\]
Hence, $w_{[\beta_i]}$ acts non-trivially only on the subspace $B_i$.

Restricting to $B_i$, we obtain
\[
    \bigl(w_{[\beta_i]}|_{B_i}\bigr)^2 = -I_2,
    \qquad
    w_{[\beta_i]}|_{B_i}\equiv J_2 \pmod J,
\text{ where }
    J_2 =
    \begin{pmatrix}
        0 & 1\\
        -1 & 0
    \end{pmatrix}.
\]

It follows that
\[
    w_{[\beta_i]}|_{B_i}
    =
    \begin{pmatrix}
        a_i & b_i \\
        -\dfrac{1+a_i^2}{b_i} & -a_i
    \end{pmatrix},
\]
for some $a_i\in R$ and $b_i\in R^\times$ satisfying
\[
    a_i \in J,
    \qquad
    b_i\equiv 1\pmod J.
\]

Now set
\[
    T_i =
    \begin{pmatrix}
        b_i & 0\\
        -a_i & 1
    \end{pmatrix}.
\]
Then $T_i\in \GL_2(R)$ and $T_i\equiv I_2 \pmod J$. A direct calculation shows that
\[
    T_i^{-1} \, \bigl(w_{[\beta_i]}|_{B_i}\bigr) \, T_i = J_2.
\]

Finally, let
\[
    T = \diag(T_1, \dots, T_m).
\]
Then $T\in \GL_{2m}(R,J)$, and a direct computation using the obtained initial form of $w_{[\beta_i]}$ yields
\[
    T^{-1} \, w_{[\beta_i]} \, T = w_{[\beta_i]}(1)
    \qquad (1 \leq i \leq m).
\]
Therefore, $T$ provides the desired change of basis. 
In other words, after this conjugation, all elements $w_{[\beta_i]}$ are normalized to their standard forms.

\subsubsection{\textbf{Normalization of $w_{[\alpha_i]}$ for $i = 1, \dots, m-1$}}

After incorporating this change of basis into the final matrix $g_3$, we may assume without loss of generality that
\[
    \varphi_2 \bigl(h_{[\alpha_i]}(-1)\bigr) = h_{[\alpha_i]}(-1)
    \quad \text{and} \quad
    \varphi_2 \bigl(w_{[\beta_i]}(1)\bigr) = w_{[\beta_i]}(1)
\]
for all $1 \leq i \leq m$.

Fix $i\in\{1,\dots,m-1\}$. Since $w_{[\alpha_i]}(1)$ commutes with $h_{[\alpha_i]}(-1)$ in $\PGL_{2m}(R)$, the standard scalar argument shows that the chosen representative $w_{[\alpha_i]}$ also commutes with $h_{[\alpha_i]}(-1)$ in $\GL_{2m}(R)$. Consequently, $w_{[\alpha_i]}$ preserves both the subspace $B_i\oplus B_{i+1}$ and its complement
\[
    B' := B_1 \oplus \cdots \oplus B_{i-1} \oplus B_{i+2} \oplus \cdots \oplus B_m.
\]
Moreover, since
\[
    w_{[\alpha_i]}^2 = h_{[\alpha_i]}(-1),
\]
and $h_{[\alpha_i]}(-1)$ acts trivially on $B'$, while $w_{[\alpha_i]}\equiv I \pmod{J}$ on $B'$, it follows that
\[
    w_{[\alpha_i]}|_{B'}=I.
\]
Hence, $w_{[\alpha_i]}$ acts non-trivially only on the subspace
$B_i \oplus B_{i+1}$. 
We claim that, on this subspace, it has the block form
\[
    \begin{pmatrix}
        0 & P_i \\
        -P_i^{-1} & 0
    \end{pmatrix},
\]
where $P_i \in \GL_2(R)$ is of the form
\[
    P_i =
    \begin{pmatrix}
        a_i & b_i \\
        b_i & -a_i
    \end{pmatrix},
    \qquad
    a_i \equiv 1 \pmod{J},
    \quad
    b_i \in J.
\]

To see this, consider the relations
\begin{gather*}
    w_{[\alpha_i]}(1)\, h_{[\beta_i]}(-1)\, w_{[\alpha_i]}(1)^{-1}
    =
    h_{[\beta_{i+1}]}(-1), \\
    w_{[\alpha_i]}(1)\, w_{[\beta_i]}(1)\, w_{[\alpha_i]}(1)^{-1}
    =
    w_{[\beta_{i+1}]}(1)^{-1}.
\end{gather*}
By the usual scalar argument, these relations also hold for the chosen representatives in $\GL_{2m}(R)$. 
Using the relation
\[
    w_{[\alpha_i]} h_{[\beta_i]}(-1)
    =
    h_{[\beta_{i+1}]}(-1) w_{[\alpha_i]},
\]
on $B_i \oplus B_{i+1}$, the matrix of $w_{[\alpha_i]}$ takes the form
\[
    w_{[\alpha_i]} =
    \begin{pmatrix}
        0 & P_i \\
        Q_i & 0
    \end{pmatrix},
    \qquad
    P_i, Q_i \in \GL_2(R).
\]

Since $w_{[\alpha_i]}^2 = h_{[\alpha_i]}(-1)$ and $h_{[\alpha_i]}(-1)$ acts as $-I$ on $B_i \oplus B_{i+1}$, we obtain
\[
    P_i Q_i = Q_i P_i = -I_2,
\]
and hence
\[
    Q_i = -P_i^{-1}.
\]

Next, using the relation
\[
    w_{[\alpha_i]} w_{[\beta_i]}(1)
    =
    w_{[\beta_{i+1}]}(1)^{-1} w_{[\alpha_i]},
\]
we obtain on $B_i \oplus B_{i+1}$ that
\[
    \begin{pmatrix}
        0 & P_i \\
        -P_i^{-1} & 0
    \end{pmatrix}
    \begin{pmatrix}
        J_2 & 0 \\
        0 & I_2
    \end{pmatrix}
    =
    \begin{pmatrix}
        I_2 & 0 \\
        0 & -J_2
    \end{pmatrix}
    \begin{pmatrix}
        0 & P_i \\
        -P_i^{-1} & 0
    \end{pmatrix},
\]
where
\[
    J_2 =
    \begin{pmatrix}
        0 & 1 \\
        -1 & 0
    \end{pmatrix}
\]
and $I_2$ is the identity matrix. 
This implies that
\[
    P_i J_2 = -J_2 P_i.
\]
Given
\[
    P_i \equiv K_2 \pmod{J},
    \qquad
    K_2 =
    \begin{pmatrix}
        1 & 0 \\
        0 & -1
    \end{pmatrix},
\]
the above relation implies that $P_i$ must be of the form
\[
    P_i =
    \begin{pmatrix}
        a_i & b_i \\
        b_i & -a_i
    \end{pmatrix},
    \qquad
    a_i \equiv 1 \pmod{J},
    \quad
    b_i \in J,
\]
which completes the proof of the claim.

\medskip

We now proceed to normalize the representative $w_{[\alpha_i]}$ of $\varphi_2(w_{[\alpha_i]}(1))$ for the simple root $[\alpha_i]$, where $i = 1, \dots,m-1$. 

For each $1 \leq r \leq m$, we inductively define an invertible matrix
$T_r \in \GL_2(R)$ as follows:
\[
    T_1 := I_2,
    \qquad
    T_r := P_{r-1}^{-1} T_{r-1} K_2
    \quad (2 \le r \le m).
\]
Since
\[
    P_{r-1}J_2=-J_2P_{r-1}
    \quad \text{and} \quad
    K_2J_2=-J_2K_2,
\]
both $P_{r-1}$ and $K_2$ anti-commute with $J_2$. Hence $P_{r-1}^{-1}$ also anti-commutes with $J_2$.
If $T_{r-1}$ commutes with $J_2$, then
$$
    T_r J_2
    =
    P_{r-1}^{-1}T_{r-1}K_2J_2        
    =
    -P_{r-1}^{-1}T_{r-1}J_2K_2     
    =
    -P_{r-1}^{-1}J_2T_{r-1}K_2        
    =
    J_2P_{r-1}^{-1}T_{r-1}K_2         
    =
    J_2T_r.
$$
Thus, by induction, every $T_r$ commutes with $J_2$.

Moreover, since $P_{r-1}\equiv K_2\pmod J$, we have
\[
    T_r\equiv I_2\pmod J
    \qquad
    (1\le r\le m).
\]
Set
\[
    T:= \diag(T_1, T_2, \dots, T_m).
\]
Then $T\in \GL_{2m}(R,J)$.

We claim that the matrix $T$ is the required change of basis. More precisely, we claim that
\[
    T^{-1} w_{[\beta_i]}(1) T
    =
    w_{[\beta_i]}(1)
    \qquad
    (i = 1, \dots, m),
\]
and
\[
    T^{-1} w_{[\alpha_i]} T
    =
    w_{[\alpha_i]}(1)
    \qquad
    (i = 1, \dots, m-1).
\]

The first identity holds because each $T_r$ commutes with $J_2$, and hence conjugation by $T^{-1}$ on the left and $T$ on the right preserves the already normalized elements $w_{[\beta_r]}(1)$.

For the second identity, consider the restriction to the subspace $B_i \oplus B_{i+1}$. The corresponding block of $T^{-1} w_{[\alpha_i]} T$ is
\[
    \begin{pmatrix}
        T_i^{-1} & 0 \\
        0 & T_{i+1}^{-1}
    \end{pmatrix}
    \begin{pmatrix}
        0 & P_i \\
        -P_i^{-1} & 0
    \end{pmatrix}
    \begin{pmatrix}
        T_i & 0 \\
        0 & T_{i+1}
    \end{pmatrix}
    =
    \begin{pmatrix}
        0 & T_i^{-1} P_i T_{i+1} \\
        -\bigl(T_i^{-1} P_i T_{i+1}\bigr)^{-1} & 0
    \end{pmatrix}.
\]
By construction,
\[
    T_{i+1} = P_i^{-1} T_i K_2,
\]
and therefore
\[
    T_i^{-1} P_i T_{i+1}
    =
    T_i^{-1} P_i \bigl(P_i^{-1} T_i K_2\bigr)
    =
    K_2.
\]
It follows that, on $B_i \oplus B_{i+1}$,
\[
    T^{-1} w_{[\alpha_i]} T
    =
    \begin{pmatrix}
        0 & K_2 \\
        -K_2 & 0
    \end{pmatrix}.
\]
On each $B_r$ with $r \notin \{i, i+1\}$, the element $w_{[\alpha_i]}$ acts trivially, and hence so does its conjugate by $T$. Therefore,
\[
    T^{-1} w_{[\alpha_i]} T
    =
    w_{[\alpha_i]}(1)
    \qquad
    (1 \leq i \leq m-1),
\]
as required.

\subsubsection{\textbf{Conclusion}}

Let $g_3 \in \GL_{2m}(R)$ be the product, in the order of application, of the change-of-basis matrices used above. By construction, $g_3 \in \GL_{2m}(R, J)$. Define
\[
    \varphi_3 := i_{[g_3]}^{-1} \circ \varphi_2.
\]
Then $\varphi_3$ is an isomorphism of $E'_{\ad, \sigma}(\Phi, R)$ onto a subgroup of $\PGL_{2m}(R)$ such that
\[
    \varphi_3\bigl(w_{[\alpha]}(1)\bigr) = w_{[\alpha]}(1)
    \qquad \text{for all } [\alpha] \in \Phi_\rho.
\]
Consequently, Step~(O-2) outlined in Section~\ref{sec:outline_of_main_thm} follows.


\subsection{The case \texorpdfstring{${}^2 A_{2m} \ (m \geq 3)$}{2A(2m)}}\label{subsec:image_of_w_for_even_case}

Assume that $(\Phi,\rho)$ has twisted type ${}^2 A_{2m}$ with $m \geq 3$, and let $\Delta_\rho = \{ [\alpha_1], \dots, [\alpha_m] \}$ be the simple system of $\Phi_\rho$. 
As noted in Section~\ref{sec:preliminaries}, $\Phi_\rho$ is of type $B_m$. 
We realize this root system in the Euclidean space $\mathbb{R}^{m}$ as
\[
    \Phi_\rho=
    \{\pm \nu_i \pm \nu_j \mid 1 \leq i < j \leq m\}
    \cup
    \{\pm \nu_i \mid 1 \leq i \leq m\},
\]
where $\{\nu_1,\dots,\nu_m\}$ denotes the standard orthonormal basis of $\mathbb{R}^m$.

Under this identification, the simple roots are given by
\[
[\alpha_i]=\nu_i-\nu_{i+1}
\quad (i=1,\dots,m-1),
\qquad
[\alpha_m]=\nu_m.
\]
In addition, for each $i=1,\dots,m-1$, we define
\[
[\widetilde{\alpha_i}]=\nu_i+\nu_{i+1}.
\]
Observe that the roots $[\alpha_i]$ and $[\widetilde{\alpha_i}]$ are orthogonal for every $i=1,\dots,m-1$.

We also introduce the roots $[\beta_i]\in\Phi_\rho$ $(i=1,\dots,m)$ corresponding to the vectors $\nu_i$. 
The set $\{[\beta_1],\dots,[\beta_m]\}$ coincides with the set of all positive short roots of $\Phi_\rho$. 
Moreover, each root $[\beta_i]$ is of type $A_2$ and is given explicitly by
\[
    [\beta_i]
    =
    [\alpha_i+\alpha_{i+1}+\cdots+\alpha_{2m+1-i}].
\]


\subsubsection{\textbf{Normalization of $w_{[\alpha]}$ for roots of type $A_1^2$}}

By Step~(E-1), we have
\[
    \varphi_2\bigl(h_{[\alpha_i]}(-1)\bigr)=h_{[\alpha_i]}(-1)
    \qquad (1\leq i\leq m-1).
\]

Our objective is to show that, after a suitable change of basis, the elements $w_{[\alpha]} = \varphi_2 (w_{[\alpha]}(1))$ assume the standard forms $w_{[\alpha]}(1)$ for all $[\alpha] \in \Phi_\rho$ of type $A_1^2$. 

Since the elements $w_{[\alpha_i]}(1)$ for $i=1,\dots,m-1$, together with $w_{[\widetilde{\alpha}_1]}(1)$, generate all elements $w_{[\alpha]}(1)$ corresponding to roots $[\alpha]\in\Phi_\rho$ of type $A_1^2$, it suffices to normalize $w_{[\alpha_i]}$ for $i=1,\dots,m-1$ and $w_{[\widetilde{\alpha}_1]}$.

We first determine the initial structure of $w_{[\alpha_i]}$ and $w_{[\widetilde{\alpha_i}]}$ for $i = 1, \dots, m-1$. 

Let $\{e_1,\dots,e_{2m+1}\}$ be the standard basis of $R^{2m+1}$, and define
\[
    B_r:=\langle e_r, e_{2m+2-r}\rangle \quad (1\leq r\leq m),
    \qquad
    B_{m+1}:=\langle e_{m+1}\rangle.
\]

Fix $i\in\{1,\dots,m-1\}$. Since $w_{[\alpha_i]}(1)$ commutes with $h_{[\alpha_i]}(-1)$ in $\PGL_{2m+1}(R)$, the standard scalar argument shows that the chosen representative $w_{[\alpha_i]}$ also commutes with $h_{[\alpha_i]}(-1)$ in $\GL_{2m+1}(R)$. Consequently, $w_{[\alpha_i]}$ preserves both the subspace $B_i\oplus B_{i+1}$ and its complement
\[
    B' := B_1 \oplus \cdots \oplus B_{i-1} \oplus B_{i+2} \oplus \cdots \oplus B_{m+1}.
\]
Moreover, since
$w_{[\alpha_i]}^2 = h_{[\alpha_i]}(-1)$,
and $h_{[\alpha_i]}(-1)$ acts trivially on $B'$, while $w_{[\alpha_i]}\equiv I \pmod{J}$ on $B'$, it follows that
\[
    w_{[\alpha_i]}|_{B'}=I.
\]
Hence, $w_{[\alpha_i]}$ acts non-trivially only on the subspace
$B_i \oplus B_{i+1}$.

We next consider the relations
\begin{align*}
    h_{[\alpha_{i-1}]}(-1)\, w_{[\alpha_i]}(1)\, h_{[\alpha_{i-1}]}(-1)^{-1}
    &= w_{[\alpha_i]}(1)^{-1} \qquad (i>1),\\
    h_{[\alpha_{i+1}]}(-1)\, w_{[\alpha_i]}(1)\, h_{[\alpha_{i+1}]}(-1)^{-1}
    &= w_{[\alpha_i]}(1)^{-1} \qquad (i<m-1),\\
    w_{[\alpha_i]}(1)^2 &= h_{[\alpha_i]}(-1).
\end{align*}
Applying $\varphi_2$ and using the fact that it fixes every $h_{[\alpha]}(-1)$ for $[\alpha]$ of type $A_1^2$, we obtain the corresponding relations for $W_{[\alpha_i]}$. By the standard scalar argument, these relations also hold for the chosen representative $w_{[\alpha_i]}$.

A direct computation using the applicable relations above shows that the restriction of $w_{[\alpha_i]}$ to the subspace $B_i\oplus B_{i+1}$ has the form
\[
    \begin{pmatrix}
        0 & P_i \\
        - P_i^{-1} & 0
    \end{pmatrix},
\]
where $P_i\in\GL_2(R)$ satisfies
\[
    P_i \equiv K_2 \pmod{J},
    \qquad
    K_2=
    \begin{pmatrix}
        1 & 0\\
        0 & -1
    \end{pmatrix}.
\]

We now determine the initial form of $w_{[\widetilde{\alpha_i}]}$. Consider the relations
\begin{align*}
    h_{[\alpha_i]}(-1) \, w_{[\widetilde{\alpha_i}]}(1) &= w_{[\widetilde{\alpha_i}]}(1) \, h_{[\alpha_i]}(-1), \\
    h_{[\alpha_{i-1}]}(-1) \,w_{[\widetilde{\alpha_i}]}(1) \,h_{[\alpha_{i-1}]}(-1)^{-1} &= w_{[\widetilde{\alpha_i}]}(1)^{-1} \quad (i>1), \\
    h_{[\alpha_{i+1}]}(-1)\,w_{[\widetilde{\alpha_i}]}(1)\,h_{[\alpha_{i+1}]}(-1)^{-1} &= w_{[\widetilde{\alpha_i}]}(1)^{-1} \quad (i<m-1), \\
    w_{[\widetilde{\alpha_i}]}(1)^2 &= h_{[\widetilde{\alpha_i}]}(-1).
\end{align*}

Applying the same reasoning as above, we conclude that $w_{[\widetilde{\alpha_i}]}$ acts non-trivially only on the subspace $B_i\oplus B_{i+1}$. Moreover, its restriction to this subspace is given by
\[
    \begin{pmatrix}
        0 & Q_i\\
        -Q_i^{-1} & 0
    \end{pmatrix},
\]
where $Q_i\in\GL_2(R)$ satisfies
\[
    Q_i \equiv J_2 \pmod J,
    \qquad
    J_2=
    \begin{pmatrix}
        0 & 1\\
        -1 & 0
    \end{pmatrix}.
\]

\smallskip

We now proceed to normalize the elements $w_{[\alpha_i]}$ $(i=1,\dots,m-1)$ and $w_{[\widetilde{\alpha}_1]}$.

By the standard scalar argument, $w_{[\alpha_1]}$ and $w_{[\widetilde{\alpha_1}]}$ commute exactly; hence
\[
    (Q_1 P_1^{-1})^2 = I_2, \qquad
    Q_1 P_1^{-1} \equiv \begin{pmatrix}
        0 & -1\\
        -1 & 0
    \end{pmatrix} \pmod{J}.
\]
A direct computation shows that
\[
    Q_1 P_1^{-1} = \begin{pmatrix}
        a & b\\
        b^{-1}(1-a^2) & -a
    \end{pmatrix},
\]
for some $a\in R$ and $b\in R^\times$ satisfying
\[
    a \in J, \qquad b \equiv -1 \pmod J.
\]

We define
\[
    T_1 =
    \begin{pmatrix}
        -b & 0\\
        a & 1
    \end{pmatrix}.
\]
Then \(T_1\in \GL_2(R)\) and \(T_1\equiv I_2\pmod J\). A straightforward calculation yields
\[
    T_1^{-1}(Q_1P_1^{-1})T_1
    =
    \begin{pmatrix}
        0 & -1\\
        -1 & 0
    \end{pmatrix}.
\]

For $2 \leq r \leq m$, we define inductively
\[
    T_r := P_{r-1}^{-1} T_{r-1} K_2, \qquad T_{m+1}:=1,
\]
and set
\[
    T = \operatorname{diag}(T_1,T_2,\dots,T_m,T_{m+1}) \in \GL_{2m+1}(R,J).
\]

We claim that $T$ is the desired change of basis. More precisely, $T$ satisfies
\[
    T^{-1} w_{[\alpha_i]} T = w_{[\alpha_i]}(1) \quad (i=1,\dots,m-1),
    \qquad
    T^{-1} w_{[\widetilde{\alpha}_1]} T = w_{[\widetilde{\alpha}_1]}(1).
\]

First, fix $i \in \{ 1, \dots, m-1 \}$ and consider $T^{-1} w_{[\alpha_i]} T$. On the subspace $B_i \oplus B_{i+1}$, this matrix has the block decomposition
\[
    \begin{pmatrix}
        T_i^{-1} & 0 \\
        0 & T_{i+1}^{-1}
    \end{pmatrix}
    \begin{pmatrix}
        0 & P_i \\
        - P_i^{-1} & 0
    \end{pmatrix}
    \begin{pmatrix}
        T_i & 0 \\
        0 & T_{i+1}
    \end{pmatrix}
    =
    \begin{pmatrix}
        0 & T_i^{-1} P_i T_{i+1} \\
        -\bigl(T_i^{-1}P_iT_{i+1}\bigr)^{-1} & 0
    \end{pmatrix}.
\]
Using the definition $T_{i+1}=P_i^{-1}T_iK_2$, we obtain
\[
    T_i^{-1}P_iT_{i+1}
    =
    T_i^{-1}P_i(P_i^{-1}T_iK_2)
    =
    K_2.
\]
Hence, on $B_i\oplus B_{i+1}$ we have
\[
    T^{-1} w_{[\alpha_i]} T
    =
    \begin{pmatrix}
        0 & K_2\\
        - K_2 & 0
    \end{pmatrix}.
\]
On each subspace $B_r$ with $r\notin\{i,i+1\}$, the element $w_{[\alpha_i]}$ acts trivially, and therefore so does its conjugate by $T$. Consequently,
\[
    T^{-1} w_{[\alpha_i]} T = w_{[\alpha_i]}(1) \qquad (1 \le i \le m-1),
\]
as required.

Finally, consider the matrix $T^{-1} w_{[\widetilde{\alpha_1}]} T$. On the subspace $B_1 \oplus B_2$, this matrix has the form
\[
    T_1^{-1} Q_1 T_2 = T_1^{-1} Q_1 P_1^{-1} T_1 K_2 =
    \begin{pmatrix}
        0 & -1\\
        -1 & 0
    \end{pmatrix} K_2 = J_2.
\]
On each subspace $B_r$ with $r\notin\{1,2\}$, the element $w_{[\widetilde{\alpha_1}]}$ acts trivially, and therefore so does its conjugate by $T$. Consequently,
\[
    T^{-1} w_{[\widetilde{\alpha_1}]} T
    =
    w_{[\widetilde{\alpha_1}]}(1),
\]
as desired.


\subsubsection{\textbf{Normalization of $w_{[\alpha]}$ for roots of type $A_2$}}

Recall that the positive roots of type $A_2$ are precisely $\{[\beta_1],\dots,[\beta_m]\}$.
For $a\in R^\times$ and $b\in R$, let
\[
    v_i(a,b):=v_{[\beta_i]}(a,b)\in\GL_{2m+1}(R)
\]
denote the element acting on $B_i\oplus B_{m+1}$ as
\[
    \begin{pmatrix}
        b & a & 0\\
        a^{-1}(1-b^2) & -b & 0\\
        0 & 0 & -1
    \end{pmatrix},
\]
and acting trivially on all remaining blocks (cf. Section~\ref{sec:outline_of_main_thm}).

After incorporating this change of basis into the final matrix $g_3$, we may assume without loss of generality that
\[
    \varphi_2 \bigl(h_{[\alpha]}(-1)\bigr) = h_{[\alpha]}(-1),
    \qquad
    \varphi_2 \bigl(w_{[\alpha]}(1)\bigr) = w_{[\alpha]}(1)
\]
for all $[\alpha] \in \Phi_\rho$ of type $A_1^2$.

Since the roots $[\beta_1],\dots,[\beta_m]$ are conjugate under the action of the Weyl group generated by roots of type $A_1^2$, it is enough to normalize the element
\[
    w_m:=w_{[\beta_m]}=w_{[\alpha_m]}.
\]
Indeed, to establish Step~(E-2) from Section~\ref{sec:outline_of_main_thm}, it suffices to show that, after a suitable basis change, $w_m$ can be brought to the form $v_m(a,b)$ for some $a \in R^\times$ and $b \in R$ satisfying $a \equiv 1/2 \pmod{J}$ and $b \in J$.

\smallskip

For every $1\leq i\leq m-2$, the roots $[\alpha_i]$ and $[\alpha_m]$ are orthogonal, and hence the element $w_{[\alpha_m]}(1,1/2)$ commutes with $h_{[\alpha_i]}(-1)$ in $\operatorname{PGL}_{2m+1}(R)$.

For $i = m-1$, the same commutation relation follows directly from the block forms. 
Indeed, the element $w_{[\alpha_m]}(1,1/2)$ acts trivially on $B_{m-1}$ and preserves $B_m$ and $B_{m+1}$, whereas $h_{[\alpha_{m-1}]}(-1)$ acts as $-I_2$ on both $B_{m-1}$ and $B_m$, and as the identity on $B_{m+1}$.

By the standard scalar argument, these commutation relations hold exactly for the representative $w_m$ in $\operatorname{GL}_{2m+1}(R)$. Since the subspaces $B_1, \ldots, B_{m+1}$ are precisely the common eigenspaces of the elements
\[
    h_{[\alpha_i]}(-1), \qquad (i = 1, \ldots, m-1),
\]
it follows that $w_m$ preserves each subspace $B_i$, for $i = 1, \ldots, m+1$.

Since $w_m^2 = I_{2m+1}$ and
\[
    \left. w_m \right|_{B_i} \equiv I_2 \pmod{J} \qquad (i=1,\ldots,m-1),
\]
we obtain
\[
    \left. w_m \right|_{B_i} = I_2 \qquad (i = 1, \ldots, m-1).
\]

Moreover,
\[
    \left(\left. w_m \right|_{B_{m+1}} \right)^2 = 1
    \text{ and }
    \left. w_m \right|_{B_{m+1}} \equiv -1 \pmod{J}.
\]
Therefore,
\[
    \left. w_m \right|_{B_{m+1}} = -1.
\]

It remains to determine the action of $w_m$ on $B_m$. On this block we have
\[
    \left( \left.w_m \right|_{B_m} \right)^2 = I_2,
    \qquad
    \left. w_m \right|_{B_m} \equiv
    \begin{pmatrix}
        0 & 1/2\\
        2 & 0
    \end{pmatrix}
    \pmod{J}.
\]
Consequently,
\[
    \left. w_m \right|_{B_m} =
    \begin{pmatrix}
        b & a\\
        a^{-1}(1-b^2) & -b
    \end{pmatrix}
\]
for some $a\in R^\times$ and $b\in R$ satisfying
\[
    a \equiv \frac{1}{2} \pmod{J}, \qquad b \in J.
\]
Equivalently, $w_m = v_m(a,b)$, as required.

\begin{rmk}
    We note that a further change of basis can be used to normalize $w_m$ to the desired standard form. However, such a change of basis generally affects the elements $w_{[\widetilde{\alpha_i}]}$ that have already been normalized. Consequently, at this stage we have two possible approaches.
    
    The first is to normalize all elements $w_{[\alpha]}$ corresponding to roots $[\alpha]\in\Phi_\rho$ of type $A_1^2$, while bringing each $w_{[\beta_i]}$ to the form $v_i(a,b)$. The second is to normalize all elements $w_{[\alpha_i]}$ $(i=1,\dots,m-1)$ together with all elements $w_{[\beta_i]}$ $(i=1,\dots,m)$.
    
    In either approach, one can subsequently show that the remaining elements are automatically fixed. In this paper, we adopt the first approach.
\end{rmk}

\subsubsection{\textbf{Conclusion}}
Let $g_3 \in \GL_{2m+1}(R)$ be the product, in the order of application, of the change-of-basis matrices used above. By construction, $g_3 \in \GL_{2m+1}(R, J)$. Define
\[
    \varphi_3 := i_{[g_3]}^{-1} \circ \varphi_2.
\]
Then $\varphi_3$ is an isomorphism of $E'_{\ad, \sigma}(\Phi, R)$ onto a subgroup of $\PGL_{2m+1}(R)$ such that
\begin{gather*}
    \varphi_3 \left( w_{[\alpha]}(1) \right) = w_{[\alpha]}(1) \qquad \text{for all } [\alpha] \in \Phi_\rho \text{ such that } [\alpha] \sim A_1^2, \\
    \varphi_3 \left( w_{[\alpha]}(1,1/2) \right) = v_{[\alpha]}(a,b) \qquad \text{for all } [\alpha] \in \Phi_\rho \text{ such that } [\alpha] \sim A_2,
\end{gather*}
where $a\in R^\times$ and $b\in R$ satisfy $a \equiv 1/2 \pmod{J}$ and $b \in J$.
Consequently, Step~(E-2) outlined in Section~\ref{sec:outline_of_main_thm} follows.

\section{The Images of \texorpdfstring{$x_{[\alpha]}(1)$}{x(1)}} \label{sec:image_of_x(1)}


For each root $[\alpha] \in \Phi_\rho$, define
\[
    X_{[\alpha]} := 
    \begin{cases}
        \varphi_{3}\bigl(x_{[\alpha]}(1)\bigr), & \text{if } [\alpha] \sim A_1 \text{ or } A_1^2, \\[4pt]
        \varphi_{3}\bigl(x_{[\alpha]}(1,1/2)\bigr), & \text{if } [\alpha] \sim A_2.
    \end{cases}
\]
Since $\varphi_3 = i_{[g_2g_3]}^{-1} \circ \varphi_1$ and $g_2g_3 \in \GL_{\ell+1}(R,J)$, applying the remark following Lemma~\ref{lemma:compatible-representatives} with $C = g_2 g_3$ to the standard simply connected representative yields a representative $x_{[\alpha]} \in \SL_{\ell+1}(R)$ of $X_{[\alpha]}$ such that
\[
    x_{[\alpha]} \equiv
    \begin{cases}
        x_{[\alpha]}(1) \pmod{J}, & [\alpha] \sim A_1 \text{ or } A_1^2, \\[2mm]
        x_{[\alpha]}(1, 1/2) \pmod{J}, & [\alpha] \sim A_2.
    \end{cases}
\]
These representatives will be further normalized in Subsection~\ref{subsec:choice_of_x}.


\subsection{Choice of representative for \texorpdfstring{$X_{[\alpha]}$}{X-i}}\label{subsec:choice_of_x}

First, consider the odd case ${}^2 A_{2m-1}$ with $m \geq 3$. 
We begin by selecting a preliminary representative $x_{[\alpha_1]} \in \SL_{2m}(R)$ for the element $X_{[\alpha_1]}$, chosen such that
\[
    x_{[\alpha_1]} \equiv x_{[\alpha_1]}(1) \pmod{J}.
\]
Consider the standard relation
\[
    h_{[\alpha_2]}(-1) \, x_{[\alpha_1]}(1) \, h_{[\alpha_2]}(-1)^{-1} = x_{[\alpha_1]}(1)^{-1}.
\]
Applying the map $\varphi_3$ to both sides yields the corresponding projective identity
\[
    h_{[\alpha_2]}(-1) \, X_{[\alpha_1]} \, h_{[\alpha_2]}(-1)^{-1} = X_{[\alpha_1]}^{-1}.
\]
Thus, there exists a scalar $\lambda \in R^{\times}$ such that
\[
    h_{[\alpha_2]}(-1) \, x_{[\alpha_1]} \, h_{[\alpha_2]}(-1)^{-1} = \lambda \, x_{[\alpha_1]}^{-1}
\]
in $\SL_{2m}(R)$. Taking the congruence modulo $J$, we find $\lambda \equiv 1 \pmod{J}$, and computing the determinant of both sides gives $\lambda^{2m} = 1$.

As established in Section~\ref{sec:image_of_h(-1)}, the $2$-primary component of the order of $\lambda$ is trivial: any $2$-power root of unity congruent to $1$ modulo $J$ must equal $1$, since $2 \in R^{\times}$ and $R$ is local. Consequently, the order of $\lambda$ is odd, and the squaring map induces an automorphism of the cyclic group $\langle\lambda\rangle$. Therefore, there exists
$\eta \in \langle\lambda\rangle$
such that
$\eta^2 = \lambda^{-1}$.

Replacing the preliminary representative $x_{[\alpha_1]}$ with $\eta \, x_{[\alpha_1]}$, we obtain the exact relation
\[
    h_{[\alpha_2]}(-1) \, x_{[\alpha_1]} \, h_{[\alpha_2]}(-1)^{-1} = x_{[\alpha_1]}^{-1}.
\]
Furthermore, because $\eta^{2m} = 1$, the modified representative remains in $\SL_{2m}(R)$, and since $\eta \equiv 1 \pmod{J}$, the congruence condition modulo $J$ is preserved.

We naturally define the representatives of $\varphi_3(x_{[\alpha]}(-1))$ and $\varphi_3(x_{[\alpha]}(\pm 2))$ as $x_{[\alpha]}^{-1}$ and $x_{[\alpha]}^{\pm 2}$, respectively, for $[\alpha]\not\sim A_2$.

Next, let $[\beta] \in \Phi_\rho$ be an arbitrary root of type $A_1^2$. There exists a product $w$ of normalized Weyl elements such that
\[
    X_{[\beta]} = w X_{[\alpha_1]}^{\pm 1} w^{-1}.
\]
Accordingly, once $x_{[\alpha_1]}$ has been scalar-normalized, we define
\[
    x_{[\beta]} := w \, x_{[\alpha_1]}^{\pm 1} \, w^{-1}.
\]

We now turn to the roots of type $A_1$. First, consider the simple root $[\alpha_m]$. At this point, we choose a preliminary representative $x_{[\alpha_m]} \in \SL_{2m}(R)$ for $X_{[\alpha_m]}$ satisfying $x_{[\alpha_m]} \equiv x_{[\alpha_m]}(1) \pmod{J}$. We will rigorously normalize this element as required in the subsequent subsection. The representatives $x_{[\alpha]}$ for all other roots $[\alpha] \sim A_1$ are then defined via Weyl conjugation from $x_{[\alpha_m]}$.

\medskip

Finally, consider the even case ${}^2 A_{2m}$ with $m \geq 3$. In this setting, the roots of type $A_1^2$ correspond to the long roots. We normalize their representatives $x_{[\alpha]}$ exactly as in the ${}^2 A_{2m-1}$ case: we scale a preliminary representative $x_{[\alpha_1]} \in \SL_{2m+1}(R)$ by an appropriate scalar of odd order to satisfy the exact identity $h_{[\alpha_2]}(-1) \, x_{[\alpha_1]} \, h_{[\alpha_2]}(-1)^{-1} = x_{[\alpha_1]}^{-1}$, and generate all other long root representatives via Weyl conjugation.

For roots $[\alpha] \sim A_2$ (the short roots), we employ a parallel scalar-normalization strategy. We first select a preliminary representative $x_{[\alpha_m]}$ for $X_{[\alpha_m]}$ satisfying $x_{[\alpha_m]} \equiv x_{[\alpha_m]}(1,1/2) \pmod{J}$. Utilizing the standard relation
\[
    h_{[\alpha_{m-1}]}(-1) \, x_{[\alpha_m]}(1,1/2) \, h_{[\alpha_{m-1}]}(-1)^{-1}
    =
    x_{[\alpha_m]}(-1,1/2),
\]
we rescale this preliminary matrix exactly as before to obtain a representative $x_{[\alpha_m]} \in \SL_{2m+1}(R)$ of $X_{[\alpha_m]}$ satisfying
\[
    h_{[\alpha_{m-1}]}(-1) \, x_{[\alpha_m]} \, h_{[\alpha_{m-1}]}(-1)^{-1}
    =
    x_{[\alpha_m]}^{-1}.
\]
The representatives $x_{[\alpha]}$ for all remaining roots $[\alpha] \sim A_2$ are subsequently generated from this precisely normalized element $x_{[\alpha_m]}$ via Weyl conjugation.

\medskip

We now proceed to normalize the elements $x_{[\alpha]}$ for $[\alpha] \in \Phi_\rho$. 
Since any two roots of the same length are conjugate under the action of the Weyl group, and the Weyl elements have already been normalized, it suffices to verify that the representatives $x_{[\alpha]}$ attain the desired normalized form for one representative of each root length. 
Accordingly, it is enough to check the normalization for $x_{[\alpha_1]}$ and $x_{[\alpha_{m}]}$.

We consider the cases ${}^2 A_{2m-1}\ (m \geq 3)$ and ${}^2 A_{2m}\ (m \geq 3)$ separately, as the computations vary in each case.


\subsection{The case \texorpdfstring{${}^2 A_{2m-1} \ (m \geq 3)$}{A-odd}}

Assume that $(\Phi,\rho)$ has twisted type ${}^2 A_{2m-1}$ with $m \geq 3$, and let
\[
    \Delta_\rho = \{ [\alpha_1], \dots, [\alpha_m] \}
\]
be the simple system of $\Phi_\rho$.
Observe that $[\alpha_1], \dots, [\alpha_{m-1}] \sim A_1^2$ are short roots, while $[\alpha_m] \sim A_1$ is a long root. 
We continue to use the notation $[\beta_i]$ $(i = 1,\dots,m)$ for the roots of type $A_1$ introduced in the previous section. Likewise, we retain the notation $[\widetilde{\alpha_i}]$ $(i=1,\dots,m-1)$ from the previous section.

Let $\{ e_1, \dots, e_{2m} \}$ denote the standard basis of $R^{2m}$. Recall that we defined
\[
    B_r := \langle e_r, e_{2m+1-r} \rangle \qquad (1 \leq r \leq m).
\]
Then
\[
    R^{2m} = B_1 \oplus \dots \oplus B_m.
\]

\subsubsection{\textbf{Normalization of $x_{[\alpha_1]}$}}

Since $x_{[\alpha_1]}(1)$ commutes with $h_{[\beta_i]}(-1)$ for each $i = 3,\dots,m$, this commutativity is preserved for their images under the map $\varphi_3$. By the standard scalar-normalization argument, the chosen representative of $X_{[\alpha_1]}$ also satisfies the same commuting relations in $\GL_{2m}(R)$. Consequently, with respect to the decomposition 
\[
    R^{2m} = B_1 \oplus \dots \oplus B_m,
\]
the matrix $x_{[\alpha_1]}$ assumes the block form
\[
    x_{[\alpha_1]} = \begin{pmatrix} 
        A & B & 0 & \dots & 0 \\ 
        C & D & 0 & \dots & 0 \\ 
        0 & 0 & u_1 & \dots & 0 \\ 
        \vdots & \vdots & \vdots & \ddots & \vdots \\ 
        0 & 0 & 0 & \dots & u_{m-2} 
    \end{pmatrix},
\]
where $A, B, C, D$, and $u_i$ (for $i = 1,\dots,m-2$) are $2 \times 2$ matrices satisfying
\[
    A, D, u_i \equiv I_2 \pmod{J}, \qquad 
    B \equiv \begin{pmatrix} 1 & 0 \\ 0 & 0 \end{pmatrix} \pmod{J}, \qquad 
    C \equiv \begin{pmatrix} 0 & 0 \\ 0 & 1 \end{pmatrix} \pmod{J}.
\]

From the standard relation
\[
    h_{[\alpha_2]}(-1)\, x_{[\alpha_1]}(1)\, h_{[\alpha_2]}(-1)^{-1} = x_{[\alpha_1]}(-1),
\]
we deduce the corresponding projective relation for $X_{[\alpha_1]}$. Given our choice of the representative $x_{[\alpha_1]}$ in $\GL_{2m}(R)$ (cf. Subsection~\ref{subsec:choice_of_x}), this relation lifts exactly to
\[
    h_{[\alpha_2]}(-1)\, x_{[\alpha_1]}\, h_{[\alpha_2]}(-1)^{-1} = x_{[\alpha_1]}^{-1}.
\]
By examining the action on the blocks $B_3, \dots, B_m$ and noting that $h_{[\alpha_2]}(-1)$ acts as $\pm I_2$ on these subspaces, it follows that
\[
    u_i = u_i^{-1}, \qquad \text{i.e.,} \qquad u_i^2 = I_2.
\]
Since $u_i \equiv I_2 \pmod{J}$, we conclude that
\[
    u_i = I_2 \qquad (i=1,\dots,m-2).
\]

Let
\[
    B' = B_4 \oplus \cdots \oplus B_m.
\]
Then, with respect to the decomposition
\[
    R^{2m} = B_1 \oplus B_2 \oplus B_3 \oplus B',
\]
the matrix $x_{[\alpha_1]}$ takes the form
\[
    x_{[\alpha_1]} = \begin{pmatrix} 
        A & B & 0 & 0 \\ 
        C & D & 0 & 0 \\ 
        0 & 0 & I_2 & 0 \\ 
        0 & 0 & 0 & I 
    \end{pmatrix}, \qquad \text{where } I = I_{2m-6}.
\]

Adhering to the conventions established in Subsection~\ref{subsec:choice_of_x}, we fix the representatives for $X_{[\alpha_2]}$, $X_{-[\alpha_2]}$, and $X_{[\alpha_1]+[\alpha_2]}$ by imposing the following exact Weyl conjugation identities:
\begin{align*} 
    x_{[\alpha_2]} &:= (w_{[\alpha_1]}(1) w_{[\alpha_2]}(1))\, x_{[\alpha_1]}\, (w_{[\alpha_1]}(1) w_{[\alpha_2]}(1))^{-1}, \\ 
    x_{-[\alpha_2]} &:= w_{[\alpha_2]}(1)\, x_{[\alpha_2]}\, w_{[\alpha_2]}(1)^{-1}, \\ 
    x_{[\alpha_1]+[\alpha_2]} &:= w_{[\alpha_2]}(1)\, x_{[\alpha_1]}^{-1}\, w_{[\alpha_2]}(1)^{-1}. 
\end{align*}

To simplify the notation, we set
\[
    X_1 := X_{[\alpha_1]}, \quad X_2 := X_{[\alpha_2]}, \quad X_{-2} := X_{-[\alpha_2]}, \quad X_3 := X_{[\alpha_1]+[\alpha_2]},
\]
and denote the corresponding representatives by $x_1, x_2, x_{-2}$, and $x_3$, respectively. With respect to the decomposition $R^{2m} = B_1 \oplus B_2 \oplus B_3 \oplus B'$, these matrices are given by
\begin{align*}
    x_2 &= \begin{pmatrix} 
        I_2 & 0 & 0 & 0 \\ 
        0 & PAP & PBP & 0 \\ 
        0 & PCP & PDP & 0 \\ 
        0 & 0 & 0 & I 
    \end{pmatrix}, &
    x_{-2} &= \begin{pmatrix} 
        I_2 & 0 & 0 & 0 \\ 
        0 & D & -C & 0 \\ 
        0 & -B & A & 0 \\ 
        0 & 0 & 0 & I 
    \end{pmatrix},\\
    x_3^{-1} &= \begin{pmatrix} 
        A & 0 & -BP & 0 \\ 
        0 & I_2 & 0 & 0 \\ 
        - PC & 0 & PDP & 0 \\ 
        0 & 0 & 0 & I 
    \end{pmatrix},&\text{ where }
    P &= \begin{pmatrix} 1 & 0 \\ 0 & -1 \end{pmatrix} \in M_2(R).
\end{align*}

From the commutativity relation
\[
    [x_{[\alpha_1]}(1), x_{[\alpha_1]+[\alpha_2]}(1)] = 1,
\]
we obtain
$x_1 x_3 = \lambda \, x_3 x_1$,
for some $\lambda \in R^{\times}$. Comparing the $(2,2)$-blocks on both sides yields
$D = \lambda D$.

Since $D \equiv I_2 \pmod{J}$, the matrix $D$ is invertible, forcing $\lambda = 1$. Comparing the remaining blocks, we obtain
\[
    (A - I_2)B = 0, \qquad C(A - I_2) = 0, \qquad CB = 0.
\]

Similarly, from the relation
\[
    [x_{[\alpha_1]}(1), x_{-[\alpha_2]}(1)] = 1,
\]
we get $x_1 x_{-2} = \lambda \, x_{-2} x_1$ for some $\lambda \in R^{\times}$. Comparison of the $(1,1)$-blocks yields
$A = \lambda A$.

Since $A \equiv I_2 \pmod{J}$, the matrix $A$ is invertible, and again $\lambda = 1$. By comparing the remaining blocks, we get
\[
    B(D - I_2) = 0, \qquad (D - I_2)C = 0, \qquad BC = 0.
\]

Using the exact relation
\[
    h_{[\alpha_2]}(-1)\, x_{[\alpha_1]}\, h_{[\alpha_2]}(-1)^{-1} = x_{[\alpha_1]}^{-1},
\]
we deduce
\[
    A^2 - BC = I_2, \qquad D^2 - CB = I_2.
\]
Thus,
$A^2 = I_2 = D^2$.

Since $A, D \equiv I_2 \pmod{J}$, it follows that
$A = I_2 = D$.

Next, we apply the standard Chevalley relation
\[
    [x_{[\alpha_1]}(1), x_{[\alpha_2]}(1)] = x_{[\alpha_1+\alpha_2]}(1).
\]
For our chosen representatives, this lifted relation is exact. Indeed, if a scalar factor were to appear in the lift, comparing the $(2,2)$-blocks of both sides would compel this scalar to be $1$, as both sides act as the identity $I_2$ on $B_2$. Hence, we obtain
$[x_1,x_2]=x_3$.

A direct calculation provides
$BPB = B$ and $CPC = -C$.

Now, consider the root
\[
    [\widetilde{\alpha}_1] = \nu_1 + \nu_2 = [\alpha_1] + 2[\alpha_2] + \cdots + 2[\alpha_{m-1}] + [\alpha_m]
\]
(refer to Section~\ref{subsec:image_of_w_for_odd_case} for the notation $\nu_i$). Then $[\widetilde{\alpha}_1]$ is a root of type $A_1^2$, i.e., a short root, and the subsystem generated by $[\alpha_1]$ and $[\widetilde{\alpha}_1]$ is of type $B_2$.

Observe that, in the standard group,
\[
    w_{[\widetilde{\alpha}_1]}(1)\, x_{[\alpha_1]}(1)\, w_{[\widetilde{\alpha}_1]}(1)^{-1} = x_{[\alpha_1]}(-1).
\]
For the designated representatives, we utilize the corresponding relation in the form
\[
    x_{[\alpha_1]} w_{[\widetilde{\alpha}_1]}(1) x_{[\alpha_1]} = w_{[\widetilde{\alpha}_1]}(1).
\]
Indeed, if a scalar factor were present, the two sides would differ by a unit scalar. Because both sides act identically on the common fixed summand $B_3 \oplus B'$, this scalar is forced to be $1$. Hence, the above exact relation imposes the following conditions on the entries of $B = (b_{ij})$ and $C = (c_{ij})$:
\[
    b_{11} = c_{22} \equiv 1 \pmod{J}, \qquad b_{12} = c_{21}, \qquad b_{21} = c_{12}, \qquad b_{22} = c_{11},
\]
together with
\[
    b_{12} = -b_{21}, \qquad b_{22} = -\frac{b_{12}^2}{b_{11}}.
\]

Thus,
\[
    B = \begin{pmatrix} p & q \\ -q & -q^2/p \end{pmatrix}, \qquad C = \begin{pmatrix} -q^2/p & -q \\ q & p \end{pmatrix},
\]
where $p \equiv 1 \pmod{J}$ and $q \equiv 0 \pmod{J}$. Since $BPB = B$, we have
$p^2 + q^2 = p$.

Consequently, the matrix $x_1$ assumes the form
\[
    x_1 = x_1(p,q) := \begin{pmatrix} 
        1 & 0 & p & q & 0 \\ 
        0 & 1 & -q & -q^2/p & 0 \\ 
        -q^2/p & -q & 1 & 0 & 0 \\ 
        q & p & 0 & 1 & 0 \\ 
        0 & 0 & 0 & 0 & I_{2m-4} 
    \end{pmatrix},
\]
where $p \equiv 1 \pmod{J}$ and $q \equiv 0 \pmod{J}$ such that
$p^2 + q^2 = p$.

\medskip

We now construct a change of basis that preserves all previously fixed elements, namely $h_{[\alpha]}(-1)$ and $w_{[\alpha]}(1)$ for every root $[\alpha] \in \Phi_\rho$, and with respect to which the matrix $x_1$ assumes the same form as $x_{[\alpha_1]}(1)$ in the original basis.

Let $Q \in \GL_{2m}(R)$ be the matrix defining the unitary structure (cf. Section~\ref{sec:G(R)_and_MT}). By definition, $Q^2 = -I$. Hence, for any $c \in J$, the matrix
$E(c) := I + cQ$
is invertible, with inverse
\[
    E(c)^{-1} = \frac{1}{1 + c^2}(I - cQ).
\]

A direct computation shows that
\[
    E(c)\, w_{[\alpha]}(1)\, E(c)^{-1} = w_{[\alpha]}(1),
\]
and consequently,
\[
    E(c)\, h_{[\alpha]}(-1)\, E(c)^{-1} = h_{[\alpha]}(-1).
\]
Moreover,
\[
    E(c) \, x_1(p,q) \, E(c)^{-1} = x_1(p', q'),
\]
where
\[
    p' = \frac{(p + cq)^2}{p(1 + c^2)},
    \qquad
    q' = \frac{pq(1-c^2) - (p^2-q^2)c}{p(1 + c^2)}.
\]

Choose
 $c = q/p$.
Then $c \in J$, since $q \in J$ and $p \in R^\times$. Using the relation $p^2 + q^2 = p$, we obtain
\[
    p' = 1,
    \qquad
    q' = 0.
\]
Consequently,
\[
    E(c) \, x_1 \, E(c)^{-1} = x_{[\alpha_1]}(1),
    \qquad
    c=q/p\in J.
\]

As noted earlier, every $x_{[\alpha]}(1)$ with $[\alpha]\sim A_1^2$ is obtained from $x_{[\alpha_1]}(1)$ by conjugation with a product of normalized Weyl elements. Since $E(c)$ commutes with each such element and normalizes $x_{[\alpha_1]}$, all these root elements are normalized under this change of basis.

\subsubsection{\textbf{Normalization of $x_{[\alpha_m]}$}}

For the remainder of this subsection, set $x_m:=E(c)x_{[\alpha_m]}E(c)^{-1}$; since $E(c)$ fixes the normalized torus and Weyl elements, the following argument applies to this transformed representative, while all short-root elements remain normalized.

It remains to normalize $X_{[\beta_i]}$ for all $[\beta_i] \sim A_1$, representing the long roots. As previously established, it is sufficient to consider the case of the simple root $[\alpha_m]$.

Since $x_{[\alpha_m]}(1)$ commutes with $h_{[\beta_i]}(-1)$ and $w_{[\beta_i]}(1)$ for each $i=1,\dots,m-1$, the transformed representative $x_m$ commutes projectively with these normalized elements. By the standard scalar-normalization argument, the relations lift to exact equalities in $\GL_{2m}(R)$. Consequently, with respect to the decomposition
\[
    R^{2m} = B_1 \oplus \cdots \oplus B_m,
\]
the matrix $x_m$ assumes the block form
\[
    x_m = \begin{pmatrix} 
        u_1 & \cdots & 0 & 0 \\ 
        \vdots & \ddots & \vdots & \vdots \\ 
        0 & \cdots & u_{m-1} & 0 \\ 
        0 & \cdots & 0 & A 
    \end{pmatrix},
\]
where $A$ and $u_i$ (for $i=1,\dots,m-1$) are $2\times 2$ matrices satisfying
\[
    A \equiv \begin{pmatrix}1 & 1 \\ 0 & 1\end{pmatrix} \pmod{J}, \qquad 
    u_i = \begin{pmatrix}
        v_i & w_i \\
        -w_i & v_i
    \end{pmatrix} \equiv I_2 \pmod{J}.
\]

Moreover, because the lifted representatives commute exactly with $w_{[\alpha_i]}(1)$ for $i=1,\dots,m-2$, it follows that
\[
    u_i = u_{i+1} \qquad (i=1,\dots,m-2).
\]
Hence, all blocks $u_i$ coincide, allowing us to write
\[
    u := u_1 = \cdots = u_{m-1} = \begin{pmatrix} v & w \\ - w & v \end{pmatrix}, \qquad v \equiv 1 \pmod{J}, \quad w \in J.
\]

Similarly, we express $A$ as
\[
    A = \begin{pmatrix} a & b \\ c & d \end{pmatrix}, \qquad a, b, d \equiv 1 \pmod{J}, \quad c \in J.
\]

Next, consider the relation
\[
    [x_{-[\alpha_{m-1}]}(1),\, x_{[\alpha_m + \alpha_{m-1}]}(1)] = x_{[\alpha_m]}(2) = x_{[\alpha_m]}(1)^2.
\]
Since both $[\alpha_m + \alpha_{m-1}]$ and $[\alpha_{m-1}]$ are roots of type $A_1^2$, their corresponding root elements remain normalized after conjugation by $E(c)$. Thus, we have
\[
    \bigl(i_{[E(c)]}\circ\varphi_3\bigr)\bigl(x_{[\alpha_m]}(2)\bigr) = x_{[\alpha_m]}(2),
\]
which implies that
$[x_m]^2 = x_{[\alpha_m]}(2)$
in $\PGL_{2m}(R)$. It follows that there exists some $\lambda \in R^\times$ with $\lambda \equiv 1 \pmod{J}$ such that
$x_m^2 = \lambda\, x_{[\alpha_m]}(2)$
in $\GL_{2m}(R)$.

By comparing the first $m-1$ blocks, we obtain
\[
    u^2 = \begin{pmatrix} v^2 - w^2 & 2vw \\ -2vw & v^2 - w^2 \end{pmatrix} = \lambda I_2,
\]
which yields the system
\[
    2vw = 0, \qquad v^2 - w^2 = \lambda.
\]
Since $2 \in R^\times$ and $v \equiv 1 \pmod{J}$, we deduce that
$w = 0$ and  $v^2 = \lambda$.

Comparing the final block, we have
\[
    A^2 = \begin{pmatrix} a^2 + bc & ab + bd \\ ac + cd & bc + d^2 \end{pmatrix} = \lambda \begin{pmatrix} 1 & 2 \\ 0 & 1 \end{pmatrix},
\]
which yields the relations
\[
    c (a+d) = 0, \quad a^2 + bc = \lambda, \quad d^2 + bc = \lambda, \quad b(a+d) = 2\lambda.
\]
Because $a, d \equiv 1 \pmod{J}$ and $2 \in R^\times$, the sum $a+d$ is invertible. Hence, these equations simplify to
\[
    c = 0, \quad a = b = d, \quad a^2 = \lambda.
\]

Given that $a^2 = \lambda = v^2$ and $a, v \equiv 1 \pmod{J}$, it must be that
$a = v$.
Therefore, in $\GL_{2m}(R)$, we find
$x_m = a\, x_{[\alpha_m]}(1)$.

Passing to the projective group $\PGL_{2m}(R)$, we conclude
\[
    [x_m] = [x_{[\alpha_m]}(1)],
\]
as desired.

\subsubsection{\textbf{Conclusion}}

Set
\[
    g_4 := E(c)^{-1}, \qquad \text{where } c = q/p.
\]
Then $g_4\in \GL_{2m}(R,J)$, and by construction,
\[
    g_4^{-1} \, x_1 \, g_4 = E(c)x_1E(c)^{-1} = x_{[\alpha_1]}(1),
\]
while, as shown above,
\[
    \bigl[g_4^{-1} \, x_{[\alpha_m]} \, g_4\bigr] = \bigl[x_{[\alpha_m]}(1)\bigr].
\]
Furthermore, this conjugation fixes all previously normalized torus and Weyl elements.
Consequently, it normalizes all the elements $x_{[\alpha]}(1)$ for $[\alpha] \in \Phi_\rho$.

We now define the map
\[
    \varphi_4 := i_{[g_4]}^{-1} \circ \varphi_3.
\]
By construction, $\varphi_4$ is an isomorphism from $E'_{\ad, \sigma}(A_{2m-1}, R)$ onto a subgroup of $\PGL_{2m}(R)$ satisfying
\[
    \varphi_4\bigl(x_{[\alpha]}(1)\bigr) = x_{[\alpha]}(1)
    \qquad \text{for all } [\alpha] \in \Phi_\rho.
\]
This concludes the proof of Step~(O-3) as outlined in Section~\ref{sec:outline_of_main_thm}.


\subsection{The case \texorpdfstring{${}^2 A_{2m} \ (m \geq 3)$}{A-even}}

Assume that $(\Phi,\rho)$ has twisted type ${}^2 A_{2m}$ with $m \geq 3$, and let $\Delta_\rho = \{[\alpha_1], \dots, [\alpha_m]\}$ be the simple system of $\Phi_\rho$.
Then $[\alpha_1], \dots, [\alpha_{m-1}] \sim A_1^2$ are long roots, whereas $[\alpha_m] \sim A_2$ is a short root. 
We continue to denote by $[\beta_i] \ (i = 1, \dots, m)$ the roots of type $A_2$, as introduced in the previous section. 
Likewise, we retain the notation $[\widetilde{\alpha_i}] \ (i=1, \dots, m-1)$ from the previous section. 

Let \(\{e_1,\ldots,e_{2m+1}\}\) be the standard basis of \(R^{2m+1}\). Recall that we defined
\[
    B_r := \langle e_r, e_{2m+2-r}\rangle \qquad (1\le r\le m),
    \qquad
    B_{m+1}:=\langle e_{m+1}\rangle,
\]
so that
\[
    R^{2m+1}=B_1\oplus\cdots\oplus B_m\oplus B_{m+1}.
\]

\subsubsection{\textbf{Normalization of $x_{[\alpha_1]}$}}

We proceed by considering separate cases according to the value of $m$.

\medskip

\noindent \textbf{Case $m \geq 5$.}
This case can be treated by an argument analogous to that used for the type ${}^2A_{2m-1}$. For the reader’s convenience, we briefly sketch the proof.

Since $x_{[\alpha_1]}(1)$ commutes with $h_{[\alpha_1]}(-1)$ and with $h_{[\alpha_i]}(-1)$ for every $i=3,\dots,m-1$, the corresponding projective relations also hold for the images under $\varphi_3$. As in the previous cases, the standard scalar-normalization argument shows that the chosen representatives in $\GL_{2m+1}(R)$ satisfy these relations exactly. Consequently, with respect to the decomposition
\[
    R^{2m+1} = B_1 \oplus \cdots \oplus B_{m+1},
\]
the matrix $x_{[\alpha_1]}$ has the block form
\[
    x_1:= x_{[\alpha_1]} =
    \begin{pmatrix}
        A & B & 0 & \cdots & 0 & 0 \\
        C & D & 0 & \cdots & 0 & 0 \\
        0 & 0 & u_1 & \cdots & 0 & 0 \\
        \vdots & \vdots & \vdots & \ddots & \vdots & \vdots \\
        0 & 0 & 0 & \cdots & u_{m-2} & 0 \\
        0 & 0 & 0 & \cdots & 0 & u_{m-1}
    \end{pmatrix},
\]
where $A,B,C,D,u_i \in M_2(R)$ for $i=1,\dots,m-2$, and $u_{m-1} \in R^\times$, satisfying
\[
    A, D, u_i \equiv I_2 \pmod{J}, \qquad u_{m-1} \equiv 1 \pmod{J},
\]
\[
    B \equiv \begin{pmatrix}1 & 0 \\ 0 & 0\end{pmatrix} \pmod{J}, 
    \qquad 
    C \equiv \begin{pmatrix}0 & 0 \\ 0 & 1\end{pmatrix} \pmod{J}.
\]

From the standard relation
\[
    h_{[\alpha_2]}(-1)\, x_{[\alpha_1]}(1)\, h_{[\alpha_2]}(-1)^{-1}
    =
    x_{[\alpha_1]}(-1),
\]
we obtain the corresponding relation for $X_{[\alpha_1]}$. 
By our choice of representatives (cf. Subsection~\ref{subsec:choice_of_x}), this relation lifts to the exact identity 
\[
    h_{[\alpha_2]}(-1)\, x_{[\alpha_1]}\, h_{[\alpha_2]}(-1)^{-1}
    =
    x_{[\alpha_1]}^{-1}.
\]
Examining the action on the blocks \(B_3,\dots,B_m,B_{m+1}\), and using that \(h_{[\alpha_2]}(-1)\) acts as \(\pm I_2\) on the two-dimensional blocks and as \(\pm 1\) on \(B_{m+1}\), we obtain
\[
    u_i = I_2 \quad (i=1,\dots,m-2), \qquad u_{m-1} = 1.
\]

Let $B' = B_4 \oplus \cdots \oplus B_{m+1}$. Then, with respect to the decomposition
\[
    R^{2m+1} = B_1 \oplus B_2 \oplus B_3 \oplus B',
\]
the matrix $x_{[\alpha_1]}$ reduces to
\[
    x_{[\alpha_1]} =
    \begin{pmatrix}
        A & B & 0 & 0 \\
        C & D & 0 & 0 \\
        0 & 0 & I_2 & 0 \\
        0 & 0 & 0 & I
    \end{pmatrix},
    \qquad \text{where } I = I_{2m-5}.
\]

Analogous to the odd-dimensional case, we fix representatives of $X_{[\alpha_2]}$, $X_{-[\alpha_2]}$, and $X_{[\alpha_1]+[\alpha_2]}$ using the corresponding Weyl conjugation identities.

The exact lifts of the relations
\begin{align*}
    [x_{[\alpha_1]}(1), x_{[\alpha_1]+[\alpha_2]}(1)] &= 1, \\
    [x_{[\alpha_1]}(1), x_{-[\alpha_2]}(1)] &= 1, \\
    h_{[\alpha_2]}(-1)\, x_{[\alpha_1]}(1)\, h_{[\alpha_2]}(-1)^{-1} &= x_{[\alpha_1]}(-1), \\
    [x_{[\alpha_1]}(1), x_{[\alpha_2]}(1)] &= x_{[\alpha_1+\alpha_2]}(1),
\end{align*}
are the same as those considered in the odd-rank case.
As before, any scalar factors in these relations are forced to be trivial by comparing their actions on the common identity blocks.
Repeating this computation yields 
\[
    A = I_2 = D, \qquad BC = 0 = CB, \qquad BPB = B, \qquad CPC = -C.
\]

Now consider the root
\[
    [\widetilde{\alpha}_1] = \nu_1 + \nu_2
    =
    [\alpha_1] + 2[\alpha_2] + \cdots + 2[\alpha_{m-1}] + 2[\alpha_m]
\]
(refer to Section~\ref{subsec:image_of_w_for_even_case} for the notation $\nu_i$). Then $[\widetilde{\alpha}_1]$ is a root of type $A_1^2$, i.e., a long root.

Observe that, in the standard group,
\[
    w_{[\widetilde{\alpha}_1]}(1) \, x_{[\alpha_1]}(1)\,
    w_{[\widetilde{\alpha}_1]}(1)^{-1}
    =
    x_{[\alpha_1]}(1).
\]
For the chosen representatives, the corresponding relation is used in the exact form
\[
    w_{[\widetilde{\alpha}_1]}(1) \, x_{[\alpha_1]}\,
    w_{[\widetilde{\alpha}_1]}(1)^{-1}
    =
    x_{[\alpha_1]}.
\]
Indeed, if a scalar factor occurred, then both sides would differ by a scalar unit. Since both sides act identically on the common fixed summand \(B_3\oplus B'\), this scalar must be equal to \(1\). Hence the above exact relation imposes the following conditions on the entries of $B = (b_{ij})$ and $C = (c_{ij})$:
\[
    b_{11} = c_{22} \equiv 1 \pmod{J}, \qquad
    b_{12} = -c_{21}, \qquad
    b_{21} = -c_{12}, \qquad
    b_{22} = c_{11},
\]
together with
\[
    b_{12} = b_{21}, \qquad
    b_{22} = \frac{b_{12}^2}{b_{11}}.
\]

Thus,
\[
    B = \begin{pmatrix}
        p & q \\
        q & q^2/p
    \end{pmatrix},
    \qquad
    C = \begin{pmatrix}
        q^2/p & -q \\
        -q & p
    \end{pmatrix},
\]
where $p \equiv 1 \pmod{J}$ and $q \equiv 0 \pmod{J}$.
Since $BPB=B$, we have 
$p^2 - q^2 = p$.

Consequently, the matrix $x_1$ takes the form 
\[
    x_1 = x_1(p,q) := \begin{pmatrix}
         1 & 0 & p & q & 0 \\
         0 & 1 & q & q^2/p & 0 \\
         q^2/p & -q & 1 & 0 & 0 \\
         -q & p & 0 & 1 & 0 \\
         0 & 0 & 0 & 0 & I_{2m-4}
    \end{pmatrix},
\]
where $p \equiv 1 \pmod{J}$ and $q \equiv 0 \pmod{J}$ such that $p^2 - q^2 = p$.

\medskip

We now construct a change of basis that preserves the previously normalized torus elements and Weyl elements of type $A_1^2$, and under which $x_1$ assumes its standard form.

Let $Q \in \GL_{2m+1}(R)$ denote the matrix defining the unitary structure (cf. Section~\ref{sec:G(R)_and_MT}), so that $Q^2=I$. For $c\in J$, define
\[
    E(c)=I+cQ, \qquad E(c)^{-1}=\frac{1}{1-c^2}(I-cQ).
\]
Then
\[
    E(c) \, w_{[\alpha]}(1) \, E(c)^{-1} = w_{[\alpha]}(1), 
    \qquad 
    E(c) \, h_{[\alpha]}(-1) \, E(c)^{-1} = h_{[\alpha]}(-1) \qquad ([\alpha]\sim A_1^2),
\]
and
\[
    E(c)\,x_1(p,q)\,E(c)^{-1}=x_1(p',q'),
\]
where
\[
    p'=\frac{(p-cq)^2}{p(1-c^2)}, 
    \qquad 
    q'=\frac{pq(1+c^2)-(p^2+q^2)c}{p(1-c^2)}.
\]

Choosing $c=q/p\in J$ and applying the relation $p^2-q^2=p$, we obtain $p'=1$ and $q'=0$. Consequently,
\[
    E(c)\,x_1\,E(c)^{-1}=x_{[\alpha_1]}(1),
\]
as desired.

\medskip

\noindent \textbf{Case $m = 4$.}
This case is essentially similar to the case $m \geq 5$. The only difference is that additional relations are needed to obtain the required initial block decomposition.

The element $x_{[\alpha_1]}(1)$ commutes with both $h_{[\alpha_1]}(-1)$ and $h_{[\alpha_3]}(-1)$. Consequently, this property is preserved projectively for its image under $\varphi_3$. By the standard scalar-normalization argument, these commuting relations also hold for the chosen representatives. It follows that the matrix $x_{[\alpha_1]}$ admits the block decomposition
\[
    x_1 := x_{[\alpha_1]} =
    \begin{pmatrix}
        A_1 & 0 & 0 \\
        0 & A_2 & 0 \\
        0 & 0 & A_3
    \end{pmatrix},
\]
where $A_1, A_2 \in M_4(R)$ and $A_3 \in R^\times$, corresponding respectively to the subspaces $B_1 \oplus B_2$, $B_3 \oplus B_4$, and $B_5$.

From the standard relation
\[
    h_{[\alpha_2]}(-1)\, x_{[\alpha_1]}(1)\, h_{[\alpha_2]}(-1)^{-1}
    =
    x_{[\alpha_1]}(1)^{-1},
\]
and the usual scalar-normalization argument, we deduce the exact lifted relation
\[
    h_{[\alpha_2]}(-1)\, x_{[\alpha_1]}\, h_{[\alpha_2]}(-1)^{-1}
    =
    x_{[\alpha_1]}^{-1}.
\]
This immediately implies that $A_3 = 1$.

Next, a direct computation using the exact lifts of the following relations
\begin{align*}
    w_{[\alpha_3]}(1)\, x_{[\alpha_1]}(1) &= x_{[\alpha_1]}(1)\, w_{[\alpha_3]}(1), \\
    w_{[\widetilde{\alpha}_3]}(1)\, x_{[\alpha_1]}(1) &= x_{[\alpha_1]}(1)\, w_{[\widetilde{\alpha}_3]}(1), \\
    x_{[\alpha_1]}(1)\,\bigl(w_{[\alpha_2]}(1)\, x_{[\alpha_1]}(1)\, w_{[\alpha_2]}(1)^{-1}\bigr)
   & =
    \bigl(w_{[\alpha_2]}(1)\, x_{[\alpha_1]}(1)\, w_{[\alpha_2]}(1)^{-1}\bigr)\, x_{[\alpha_1]}(1), \\
    h_{[\alpha_2]}(-1)\, x_{[\alpha_1]}(1)\, h_{[\alpha_2]}(-1)^{-1} &= x_{[\alpha_1]}(1)^{-1},
\end{align*}
yields $A_2 = I_4$.
As before, any potential scalar factors in these lifted relations are eliminated either through the scalar-normalization convention or by comparison across a common fixed direct summand.

Applying the same argument as in the case $m \geq 5$, we obtain
\[
    A_1 =
    \begin{pmatrix}
        1 & 0 & p & q \\
        0 & 1 & q & \frac{q^2}{p} \\
        \frac{q^2}{p} & -q & 1 & 0 \\
        -q & p & 0 & 1
    \end{pmatrix},
\]
where $p \equiv 1 \pmod{J}$, $q \in J$, and $p^2-q^2=p$.
Therefore, using the same change-of-basis argument as for $m \geq 5$, we conclude that $x_1$ can be normalized to the desired standard form.

\medskip

\noindent \textbf{Case $m = 3$.}
This case is more difficult because fewer relations are available.

The element $x_{[\alpha_1]}(1)$ commutes with $h_{[\alpha_1]}(-1)$, and therefore this property holds projectively for its image under $\varphi_3$. 
By the usual scalar-normalization argument, this commuting relation lifts exactly to the chosen representative. Consequently, the matrix $x_{[\alpha_1]}$ admits the block decomposition
\[
    x_1 := x_{[\alpha_1]} =
    \begin{pmatrix}
        A_1 & 0 \\
        0 & A_2
    \end{pmatrix},
\]
where $A_1 \in M_4(R)$ and $A_2 \in M_3(R)$, corresponding to the respective subspaces $B_1 \oplus B_2$ and $B_3 \oplus B_4$.

We now claim that
\[
    A_1 =
    \begin{pmatrix}
        1 & 0 & p & q \\
        0 & 1 & q & \frac{q^2}{p} \\
        \frac{q^2}{p} & -q & 1 & 0 \\
        -q & p & 0 & 1
    \end{pmatrix},
\]
where $p \equiv 1 \pmod{J}$, $q \in J$, $p^2-q^2=p$, and $A_2 = I_3$.
This follows by a direct computation using the exact lifts of the relations:
\begin{align*}
    h_{[\alpha_2]}(-1) \, x_{[\alpha_1]}(1) \, h_{[\alpha_2]}(-1)^{-1} &= x_{[\alpha_1]}(-1), \\
    w_{[\widetilde{\alpha}_1]}(1) \, x_{[\alpha_1]}(1) &= x_{[\alpha_1]}(1) \, w_{[\widetilde{\alpha}_1]}(1), \\
    [x_{[\alpha_1]}(1), x_{[\alpha_1]+[\alpha_2]}(1)] &= 1, \\
    [x_{[\alpha_1]}(1), x_{-[\alpha_2]}(1)] &= 1, \\
    [x_{[\alpha_1]}(1), x_{[\alpha_2]}(1)] &= x_{[\alpha_1+\alpha_2]}(1).
\end{align*}

Thus, in this case we obtain
\[
    x_1 = x_1(p,q) :=
    \begin{pmatrix}
        1 & 0 & p & q & 0 \\
        0 & 1 & q & \frac{q^2}{p} & 0 \\
        \frac{q^2}{p} & -q & 1 & 0 & 0 \\
        -q & p & 0 & 1 & 0 \\
        0 & 0 & 0 & 0 & I_3
    \end{pmatrix}.
\]

Finally, using the same change-of-basis matrix $E(c)$ as for $m \geq 5$, we conclude that $x_1$ is normalized to the required form.

\medskip

Thus, after the corresponding change of basis $E(c)$, we have normalized $x_{[\alpha_1]}$. Since every $x_{[\alpha]}(1)$ with $[\alpha]\sim A_1^2$ is obtained from it by conjugation with a product of normalized Weyl elements, all such elements are normalized as well.

\subsubsection{\textbf{Image of $w_{[\beta_i]} \ (i = 1, \dots, m)$}}

After the preceding change of basis, each transformed element $w_{[\beta_i]}$ still has the form $v_i(a,b)$, where $a \equiv 1/2 \pmod{J}$ and $b \in J$. We now show that the normalized root elements of type $A_1^2$ uniquely determine these elements and force
\[
    w_{[\beta_i]} = w_{[\beta_i]}(1,\tfrac{1}{2}), \qquad i = 1, \dots, m.
\]

It suffices to establish this for the element $w_{[\beta_m]} = w_{[\alpha_m]}$, since all other elements $w_{[\beta_i]}$ are obtained from $w_{[\alpha_m]}$ by conjugation with already normalized Weyl group elements.

Observe that, in the standard group,
\[
    w_{[\alpha_m]}(1,\tfrac{1}{2})\,
    x_{[\alpha_{m-1}]}(1)^2 \,
    w_{[\alpha_m]}(1,\tfrac{1}{2})^{-1}
    =
    x_{[\widetilde{\alpha}_{m-1}]}(1).
\]
For our chosen representatives, this lifted relation is exact. If a scalar factor were to appear, the two sides would differ by a unit scalar. However, because both sides act identically on the common fixed direct summand complementary to
\[
    B_{m-1}\oplus B_m\oplus B_{m+1},
\]
this scalar is forced to be \(1\). Thus, we obtain
\[
    v_m(a,b)\,
    x_{[\alpha_{m-1}]}(1)^2\,
    v_m(a,b)^{-1}
    =
    x_{[\widetilde{\alpha}_{m-1}]}(1),
\]
or, equivalently,
\[
    v_m(a,b)\,
    x_{[\alpha_{m-1}]}(1)^2
    =
    x_{[\widetilde{\alpha}_{m-1}]}(1)\,
    v_m(a,b).
\]

On the subspace $B_{m-1}\oplus B_m\oplus B_{m+1}$, the matrix
\[
    v_m(a,b)x_{[\alpha_{m-1}]}(1)^2-x_{[\widetilde{\alpha}_{m-1}]}(1)v_m(a,b)
\]
has only the entries $(2a+b^2-1)/a$, $b$, $2a-1$, and $-2b$ possibly nonzero; hence
\[
    a=\tfrac{1}{2}
    \qquad\text{and}\qquad
    b=0.
\]
Consequently,
\[
    v_m(a,b) = v_m(\tfrac{1}{2},0) = w_{[\alpha_m]}(1, \tfrac{1}{2}),
\]
which completes the proof of the claim.

\subsubsection{\textbf{Normalization of $x_{[\alpha_m]}$}}

This case differs substantially from the corresponding argument for ${}^2 A_{2m-1}$, and we therefore provide a detailed proof. 

After the preceding change of basis, set $x_m:=E(c)x_{[\alpha_m]}E(c)^{-1}$. The scalar-normalization argument and a direct centralizer computation show that its projective commutation relations with the normalized elements $h_{[\alpha_i]}(-1)$, $w_{[\alpha_i]}(1)$, $w_{[\widetilde{\alpha}_i]}(1)$, $x_{[\alpha_i]}(1)$, and $x_{[\widetilde{\alpha}_i]}(1)$ for $i=1,\dots,m-2$ yield the decomposition
\[
    R^{2m+1} = B' \oplus B'', 
    \qquad \text{where } B' = B_1 \oplus \cdots \oplus B_{m-1} \text{ and } B'' = B_m \oplus B_{m+1},
\]
with respect to which the matrix $x_m$ assumes the block form
\[
    x_m =
    \begin{pmatrix}
        \lambda I_{2m-2} & 0 \\
        0 & A
    \end{pmatrix},
\]
where $A \in M_3(R)$ and $\lambda \in R^{\times}$ satisfy
\[
    A \equiv 
    \begin{pmatrix} 
        1 & 1/2 & 1 \\ 
        0 & 1 & 0 \\
        0 & 1 & 1
    \end{pmatrix} \pmod{J},
    \qquad
    \lambda \equiv 1 \pmod{J}.
\]

From the standard relation
\[
    h_{[\alpha_{m-1}]}(-1) \, x_{[\alpha_m]}(1,1/2) \, h_{[\alpha_{m-1}]}(-1)^{-1}
    =
    x_{[\alpha_m]}(-1,1/2),
\]
we deduce the corresponding projective identity
\[
    h_{[\alpha_{m-1}]}(-1) \, X_{[\alpha_m]} \, h_{[\alpha_{m-1}]}(-1)^{-1}
    =
    X_{[\alpha_m]}^{-1}.
\]
Given our choice of the representative $x_m = x_{[\alpha_m]}$ for $X_{[\alpha_m]}$ (cf. Subsection~\ref{subsec:choice_of_x}), a similar exact relation holds for $x_{[\alpha_m]}$, namely,
\[
    x_m \, h_{[\alpha_{m-1}]}(-1)\, x_m
    =
    h_{[\alpha_{m-1}]}(-1).
\]
This implies that
\[
    \lambda^2 = 1.
\]
Since $\lambda \equiv 1 \pmod{J}$ and $2 \in R^\times$, we conclude that
\[
    \lambda = 1.
\]

Next, utilizing the relations
\[
    w_{[\alpha_{m-1}]}(1)\, x_{[\alpha_m]}(1,1/2)\, w_{[\alpha_{m-1}]}(1)^{-1}
    =
    x_{[\alpha_m]+[\alpha_{m-1}]}(1,1/2),
\]
and
\[
    [x_{[\alpha_{m-1}]}(1),\, x_{[\alpha_{m-1}]+[\alpha_m]}(1,1/2)] = 1,
\]
we obtain
\[
    x_{[\alpha_{m-1}]}(1)\, w_{[\alpha_{m-1}]}(1)\, x_m\, w_{[\alpha_{m-1}]}(1)^{-1}
    =
    w_{[\alpha_{m-1}]}(1)\, x_m\, w_{[\alpha_{m-1}]}(1)^{-1}\, x_{[\alpha_{m-1}]}(1).
\]
Here, any scalar factors arising from the Weyl conjugation cancel, and a potential scalar factor in the lifted commutativity relation is forced to be \(1\) by comparing the action on the common fixed direct summand \(B_1\oplus\cdots\oplus B_{m-2}\). 
Furthermore, a direct entrywise computation shows that the entries of $A=(a_{ij})$ satisfy
\[
    a_{11} = a_{22} = 1, 
    \qquad 
    a_{21} = a_{23} = a_{31} = 0.
\]

Additionally, from the standard relation
\[
    [x_{[\alpha_m]}(1,1/2),\, x_{[\alpha_m]+[\alpha_{m-1}]}(1,1/2)]
    =
    x_{2[\alpha_m]+[\alpha_{m-1}]}(-1),
\]
we derive the exact lifted relation
\[
    x_{2[\alpha_m]+[\alpha_{m-1}]}(1)\, x_m\, 
    \bigl(w_{[\alpha_{m-1}]}(1)\, x_m\, w_{[\alpha_{m-1}]}(1)^{-1}\bigr)
    =
    \bigl(w_{[\alpha_{m-1}]}(1)\, x_m\, w_{[\alpha_{m-1}]}(1)^{-1}\bigr)\, x_m.
\]
Indeed, any potential scalar factor is again eliminated by restricting the action to the common fixed summand \(B_1\oplus\cdots\oplus B_{m-2}\). The same computation yields the conditions
\[
    a_{13}(1 - a_{33}) = 0, 
    \qquad 
    a_{13} a_{32} = 1,
\]
from which it follows that $a_{33} = 1$.

Finally, reapplying the exact relation
\[
    x_m\, h_{[\alpha_{m-1}]}(-1)\, x_m = h_{[\alpha_{m-1}]}(-1)
\]
yields
\[
    2 a_{12} = a_{13} a_{32}.
\]
Thus, $a_{12} = 1/2$.

Writing $a:=a_{13}\in R^\times$ (so $a^{-1}=a_{32}$ and $a\equiv1\pmod J$), we arrive at the simplified block form
\[
    A =
    \begin{pmatrix}
        1 & 1/2 & a \\
        0 & 1 & 0 \\
        0 & a^{-1} & 1
    \end{pmatrix}.
\]

To eliminate the parameter $a$, consider the diagonal matrix
\[
    D(a) = \diag(1,\dots,1,a),
\]
where the entry $a$ is in the coordinate corresponding to $B_{m+1}$. Since $a\in R^\times$ and $a\equiv1\pmod J$, the matrix $D(a)$ lies in $\GL_{2m+1}(R,J)$ and commutes with all previously normalized torus and Weyl elements, as well as with $x_{[\alpha]}(1)$ for $[\alpha]\sim A_1^2$. Moreover,
\[
    D(a)\, x_m \, D(a)^{-1} = x_{[\alpha_m]}(1,1/2).
\]
Thus, this change of basis normalizes $x_m$.

Finally, since all other roots $[\beta]$ of type $A_2$ are conjugate to $[\alpha_m]$ under the Weyl group, the corresponding elements $x_{[\beta]}(1,1/2)$ are also normalized.

\subsubsection{\textbf{Conclusion}}

Let \(E(c)\) denote the matrix introduced above to normalize the root elements corresponding to roots of type \(A_1^2\), and let \(D(a)\) be the diagonal matrix employed to normalize \(x_{[\alpha_m]}\). We define
\[
    g_4 := E(c)^{-1}D(a)^{-1}.
\]
Then \(g_4 \in \GL_{2m+1}(R,J)\). By construction, conjugation by \(g_4^{-1} = D(a)E(c)\) fixes all previously normalized torus and Weyl elements. Furthermore, it simultaneously normalizes every root element \(x_{[\alpha]}(1)\) for which \([\alpha] \sim A_1^2\), alongside \(x_{[\alpha_m]}(1,1/2)\).

We now define the map
\[
    \varphi_4 := i_{[g_4]}^{-1} \circ \varphi_3.
\]
By construction, \(\varphi_4\) is an isomorphism from \(E'_{\ad, \sigma}(A_{2m}, R)\) onto a subgroup of \(\PGL_{2m+1}(R)\) satisfying
\begin{gather*}
    \varphi_4\bigl(x_{[\alpha]}(1)\bigr) =  x_{[\alpha]}(1) \qquad \text{for all } [\alpha] \in \Phi_\rho \text{ such that } [\alpha] \sim A_1^2,\\
    \varphi_4\bigl(x_{[\alpha]}(1,1/2)\bigr) = x_{[\alpha]}(1,1/2) \qquad \text{for all } [\alpha] \in \Phi_\rho \text{ such that } [\alpha] \sim A_2.
\end{gather*}
This concludes the proof of Step~(E-3) as outlined in Section~\ref{sec:outline_of_main_thm}.


\section{The Images of \texorpdfstring{$x_{[\alpha]}(t)$}{x(t)}} \label{sec:image_of_x(t)}


Finally, to complete the proof of Theorem~\ref{MT_local1}, it remains to carry out Step~(C) outlined in Section~\ref{sec:outline_of_main_thm}. More precisely, we must show that there exists a ring automorphism $\mu$ of $R$ satisfying $\mu \circ \theta = \theta \circ \mu$ such that
\[
    \varphi_4\bigl(x_{[\alpha]}(t)\bigr) = x_{[\alpha]}(\mu(t))
    \qquad
    \text{for every } [\alpha] \in \Phi_\rho \text{ and } t \in R_{[\alpha]}.
\]

Denote by
\[
    X_{[\alpha]}(t) := \varphi_4\bigl(x_{[\alpha]}(t)\bigr)
    \in \PGL_{\ell+1}(R).
\]
Our first task is to choose a distinguished representative
\[
    x'_{[\alpha]}(t) \in \SL_{\ell+1}(R)
\]
of $X_{[\alpha]}(t)$. This construction is carried out in the following subsection.


\subsection{Choice of representatives for \texorpdfstring{$X_{[\alpha]}(t)$}{X\_[alpha](t)}}

Recall that, by the results established in the preceding section, the automorphism $\varphi_4$ fixes the elements $x_{[\alpha]}(1)$ for every root $[\alpha]$ of type $A_1$ or $A_1^2$, and it fixes the elements $x_{[\alpha]}(1,1/2)$ for every root $[\alpha]$ of type $A_2$. Furthermore, modulo $J$, the automorphism $\varphi_4$ induces a field automorphism $f \colon k \to k$ on the residue field $k = R/J$.

Our next objective is to select representatives in $\SL_{\ell+1}(R)$ for $X_{[\alpha]}(t)$ with $[\alpha]\sim A_1^2$ such that the identities used later lift to exact relations.

Since $\varphi_4 = i_{[C_1]}^{-1} \circ \varphi_1$, where $C_1 \coloneqq g_2 g_3 g_4 \in \GL_{\ell+1}(R,J)$, applying the remark following Lemma~\ref{lemma:compatible-representatives} with $C = C_1$ to $x_{[\alpha]}(t) \in E'_{\mathrm{sc},\sigma}(\Phi,R)$ for $[\alpha]\sim A_1^2$, we obtain a preliminary representative $x'_{[\alpha]}(t) \in \SL_{\ell+1}(R)$ of $X_{[\alpha]}(t)$ satisfying
\[
    x'_{[\alpha]}(t) \equiv x_{[\alpha]}\bigl(f'(t)\bigr) \pmod{J},
\]
where $f'(t)\in R$ is an arbitrary lift of $f(t+J)\in k$.

First, consider the odd case ${}^2A_{2m-1}$ for $m \geq 3$. We focus on the preliminary representative $x'_{[\alpha_1]}(t) \in \SL_{2m}(R)$ chosen above. The standard relation
\[
    h_{[\alpha_2]}(-1) x_{[\alpha_1]}(t) h_{[\alpha_2]}(-1)^{-1} = x_{[\alpha_1]}(t)^{-1}
\]
induces a corresponding projective relation for $X_{[\alpha_1]}(t)$. Consequently, there exists a scalar $\lambda \in R^\times$ such that our chosen representative satisfies
\[
    h_{[\alpha_2]}(-1) x'_{[\alpha_1]}(t) h_{[\alpha_2]}(-1)^{-1} = \lambda x'_{[\alpha_1]}(t)^{-1}
\]
in $\SL_{2m}(R)$. Taking this identity modulo $J$, we deduce that $\lambda \equiv 1 \pmod{J}$. Taking the determinant of both sides yields $\lambda^{2m} = 1$.

As established in Section~\ref{sec:image_of_h(-1)}, because $2 \in R^\times$ and $R$ is local, any $2$-power root of unity congruent to $1$ modulo $J$ must be trivial. Thus, the $2$-primary component of the order of $\lambda$ is trivial, meaning $\lambda$ has odd order. Consequently, the squaring map induces an automorphism on the cyclic group $\langle\lambda\rangle$, guaranteeing the existence of an element $\eta \in \langle\lambda\rangle$ such that $\eta^2 = \lambda^{-1}$.

By replacing the preliminary representative $x'_{[\alpha_1]}(t)$ with the rescaled element $\eta x'_{[\alpha_1]}(t)$, we obtain the exact relation
\[
    h_{[\alpha_2]}(-1) x'_{[\alpha_1]}(t) h_{[\alpha_2]}(-1)^{-1} = x'_{[\alpha_1]}(t)^{-1}.
\]
Because $\eta^{2m} = 1$, this modified representative remains in $\SL_{2m}(R)$, and since $\eta \equiv 1 \pmod{J}$, the original congruence condition modulo $J$ is preserved.

Next, let $[\beta] \in \Phi_\rho$ be an arbitrary root of type $A_1^2$. There exists a product $w$ of normalized Weyl elements such that
\[
    x_{[\beta]}(t) = w x_{[\alpha_1]}(t)^{\pm 1} w^{-1}.
\]
Accordingly, now that $x'_{[\alpha_1]}(t)$ has been scalar-normalized, we define the representatives for all other roots of type $A_1^2$ via Weyl conjugation:
\[
    x'_{[\beta]}(t) \coloneqq w x'_{[\alpha_1]}(t)^{\pm 1} w^{-1}.
\]

For roots of type $A_1$, it is unnecessary to fix representatives for $X_{[\alpha]}(t)$ at this stage. Instead, we will determine them directly using the Chevalley commutator formula in a subsequent step.

\medskip

Finally, consider the even case ${}^2 A_{2m}$ for $m \geq 3$. In this setting, the roots of type $A_1^2$ correspond to the long roots. We normalize their representatives $x'_{[\alpha]}(t)$ exactly as we did in the ${}^2 A_{2m-1}$ case: we scale the preliminary representative $x'_{[\alpha_1]}(t) \in \SL_{2m+1}(R)$ by an appropriate scalar of odd order to satisfy the exact identity 
\[
    h_{[\alpha_2]}(-1) x'_{[\alpha_1]}(t) h_{[\alpha_2]}(-1)^{-1} = x'_{[\alpha_1]}(t)^{-1},
\]
and we subsequently generate all other long root representatives via Weyl conjugation.

Similarly, for roots $[\alpha] \sim A_2$ (the short roots), we defer fixing representatives for $X_{[\alpha]}(t)$ at this point; these will also be determined later using the Chevalley commutator formula.


\subsection{The case \texorpdfstring{${}^2 A_{2m-1} \ (m \geq 3)$}{2A2m-1}}

Assume that $(\Phi,\rho)$ has twisted type ${}^2 A_{2m-1}$ with $m \geq 3$, and let $\Delta_\rho = \{[\alpha_1], \dots, [\alpha_m]\}$ be the simple system of $\Phi_\rho$. 
Observe that $[\alpha_1], \dots, [\alpha_{m-1}] \sim A_1^2$ are short roots, while $[\alpha_{m}] \sim A_1$ is a long root. 
We continue to use the notation $[\beta_i] \ (i = 1, \dots, m)$ for the roots of type $A_1$ introduced in the previous section. 
Likewise, we retain the notation $[\widetilde{\alpha_i}] \ (i=1,\dots,m-1)$ from the previous section.

Let \(\{e_1,\ldots,e_{2m}\}\) be the standard basis of \(R^{2m}\). Recall that we defined
\[
    B_r := \langle e_r, e_{2m+1-r}\rangle \qquad (1\le r\le m),
\]
so that
\[
    R^{2m}=B_1\oplus\cdots\oplus B_m.
\]

\subsubsection{\textbf{Normalization of $x'_{[\alpha]}(t)$ for $[\alpha] \in \Phi_\rho$ of type $A_1^2$}}

We first establish the normalization of the element 
\[
    x_1(t) := x'_{[\alpha_1]}(t) \qquad (t \in R).
\]
Since $x_{[\alpha_1]}(t)$ commutes with $h_{[\beta_i]}(-1)$ for each $i = 3, \dots, m$, the same property holds projectively for their images under the map $\varphi_4$. By the standard scalar-normalization argument, these commuting relations also hold for the chosen representative $x_1(t)$ of $X_{[\alpha_1]}(t)$ in $\GL_{2m}(R)$. Consequently, with respect to the decomposition 
\[
    R^{2m} = B_1 \oplus \dots \oplus B_m,
\]
the matrix $x_1(t)$ can be written in block form as
\[
    x_1(t) = \begin{pmatrix}
        A & B & 0 & \dots & 0 \\
        C & D & 0 & \dots & 0 \\
        0 & 0 & u_1 & \dots & 0 \\
        \vdots & \vdots & \vdots & \ddots & \vdots \\
        0 & 0 & 0 & \dots & u_{m-2}
    \end{pmatrix},
\]
where $A, B, C, D, u_i \ (i = 1, \dots, m-2)$ are $2 \times 2$ matrices satisfying
\[
    A, D, u_i \equiv I_2 \pmod{J}, 
    \qquad 
    B \equiv \begin{pmatrix}
        f'(t) & 0 \\
        0 & 0
    \end{pmatrix} \pmod{J}, 
    \qquad 
    C \equiv \begin{pmatrix}
        0 & 0 \\
        0 & \overline{f'(t)}
    \end{pmatrix} \pmod{J}.
\]
Here $f'(t)$ denotes a lift of the element $f(t+J)\in k$, where $f:k\to k$ is the residue-field automorphism induced by $\varphi_4$.

From the standard relation
\[
    h_{[\alpha_2]}(-1)\, x_{[\alpha_1]}(t)\, h_{[\alpha_2]}(-1)^{-1}
    =
    x_{[\alpha_1]}(-t),
\]
we obtain the corresponding projective inverse relation for $X_{[\alpha_1]}(t)$. By our choice of the representative $x_1(t)$, this relation lifts to an exact identity in $\GL_{2m}(R)$, namely
\[
    h_{[\alpha_2]}(-1)\, x_1(t) \, h_{[\alpha_2]}(-1)^{-1}
    =
    x_1(t)^{-1}.
\]
Examining the action on the blocks $B_3, \dots, B_m$, and recalling that $h_{[\alpha_2]}(-1)$ acts as $\pm I_2$ on these subspaces, it follows that
\[
    u_i = u_i^{-1}, \qquad \text{i.e.,} \qquad u_i^2 = I_2.
\]
Since $u_i \equiv I_2 \pmod{J}$, we conclude that
\[
    u_i = I_2 \qquad (i=1,\dots,m-2).
\]

Let $B' = B_3 \oplus \cdots \oplus B_m$. Then, with respect to the decomposition
\[
    R^{2m} = B_1 \oplus B_2 \oplus B',
\]
the matrix $x_1(t)$ takes the form
\[
    x_1(t) =
    \begin{pmatrix}
        A & B & 0 \\
        C & D & 0 \\
        0 & 0 & I 
    \end{pmatrix},
    \qquad \text{where } I = I_{2m-4}.
\]

Since $x_{[\alpha_1]}(t)$ commutes with $x_{[\alpha_1]}(1)$, $x_{[\alpha_1]+[\alpha_2]}(1)$, and $x_{-[\alpha_2]}(1)$, its image under $\varphi_4$ commutes projectively with these elements. Thus, the chosen representative $x_1(t)$ commutes with them up to a scalar factor; comparison of the invertible diagonal blocks (as in Section~\ref{sec:image_of_x(1)}) forces this scalar to be $1$. Hence, $x_1(t)$ satisfies the exact relations:
\begin{gather*}
    x_1(t)\, x_{[\alpha_1]}(1) = x_{[\alpha_1]}(1)\, x_1(t), \\ 
    x_1(t)\, x_{[\alpha_1]+[\alpha_2]}(1) = x_{[\alpha_1]+[\alpha_2]}(1)\, x_1(t), \\
    x_1(t)\, x_{-[\alpha_2]}(1) = x_{-[\alpha_2]}(1)\, x_1(t).
\end{gather*}
A direct computation reveals that these relations force $x_1(t)$ to assume the form
\[
    x_1(t) =
    \begin{pmatrix}
        1 & a & b & 0 & 0 \\
        0 & 1 & 0 & 0 & 0  \\
        0 & 0 & 1 & 0 & 0 \\
        0 & c & d & 1 & 0 \\
        0 & 0 & 0 & 0 & I
    \end{pmatrix},
    \qquad I=I_{2m-4},
\]
where $a, b, c, d \in R$. 

Next, using the relation
\[
    h_{[\alpha_2]}(-1)\, x_1(t)\, h_{[\alpha_2]}(-1)^{-1}
    =
    x_1(t)^{-1},
\]
we deduce that
\[
    x_1(t) \, h_{[\alpha_2]}(-1) \, x_1(t) = h_{[\alpha_2]}(-1).
\]
This forces the parameters to satisfy
\[
    a = 0 = d.
\]

Since the projective class of $x_1(t)$ lies in the twisted group, comparison on the identity block forces the projective unitary scalar to be $1$; hence the unitary condition gives
\[
    c = \overline{b}.
\]
Writing $\mu(t) := b$, the matrix $x_1(t)$ simplifies to
\[
    \varphi_4(x_{[\alpha_1]}(t)) = x_1(t) = x_{[\alpha_1]}(\mu(t)) =
    \begin{pmatrix}
        1 & 0 & \mu(t) & 0 & 0 \\
        0 & 1 & 0 & 0 & 0 \\
        0 & 0 & 1 & 0 & 0 \\
        0 & \overline{\mu(t)} & 0 & 1 & 0 \\
        0 & 0 & 0 & 0 & I_{2m-4}
    \end{pmatrix}.
\]

Finally, every root element $x_{[\alpha]}(t)$ with $[\alpha]\in\Phi_\rho$ of type $A_1^2$ is obtained from $x_{[\alpha_1]}(\pm t)$ by conjugation with a product of normalized Weyl elements. Since these Weyl elements are fixed by $\varphi_4$ and the other representatives were defined by exact Weyl-conjugation identities, it follows that
\[
    \varphi_4\bigl(x_{[\alpha]}(t)\bigr)
    =
    x_{[\alpha]}(\mu(t))
\]
for every $t\in R$ and every root $[\alpha]\in\Phi_\rho$ of type $A_1^2$.

\subsubsection{\textbf{$\mu$ as a ring automorphism}}

We first prove that $\mu$ is additive. Let $[\alpha] \in \Phi_\rho$ be of type $A_1^2$, and let $t_1,t_2 \in R$. Then we have
\[
    x_{[\alpha]}(t_1+t_2) = x_{[\alpha]}(t_1)\,x_{[\alpha]}(t_2).
\]
Applying $\varphi_4$ to this equality yields
\[
    x_{[\alpha]}(\mu(t_1+t_2)) = x_{[\alpha]}(\mu(t_1))\, x_{[\alpha]}(\mu(t_2)) = x_{[\alpha]}(\mu(t_1) + \mu(t_2)).
\]
Comparing the corresponding parameters, we obtain
\[
    \mu(t_1+t_2) = \mu(t_1)+\mu(t_2)
\]
for all $t_1, t_2 \in R$.
Hence, $\mu$ is additive.

We now prove that $\mu$ is multiplicative. By the Chevalley commutator relation,
\[
    [x_{[\alpha_1]}(t_1),x_{[\alpha_2]}(t_2)]
    =
    x_{[\alpha_1]+[\alpha_2]}(t_1t_2),
\]
for all $t_1,t_2\in R$. Applying $\varphi_4$ to this relation, we obtain
\[
    x_{[\alpha_1]+[\alpha_2]}(\mu(t_1t_2))
    =
    [x_{[\alpha_1]}(\mu(t_1)),
    x_{[\alpha_2]}(\mu(t_2))]
    =
    x_{[\alpha_1]+[\alpha_2]}(\mu(t_1)\mu(t_2)).
\]
Hence,
\[
    \mu(t_1t_2)
    =
    \mu(t_1)\mu(t_2)
\]
for all $t_1,t_2\in R$, and therefore $\mu$ is multiplicative.

Since $\varphi_4$ fixes $x_{[\alpha_1]}(1)$, we have $\mu(1)=1$; consequently, $\mu: R \longrightarrow R$ is a unital ring homomorphism. We claim that $\mu$ is a ring automorphism. Indeed, applying the preceding argument to the inverse automorphism $\varphi_4^{-1}$, which also fixes all the normalized
elements, we obtain a unital ring homomorphism $\nu: R \longrightarrow R$
such that
\[
    \varphi_4^{-1}\bigl(x_{[\alpha]}(t)\bigr)
    =
    x_{[\alpha]}(\nu(t))
\]
for every root $[\alpha] \sim A_1^2$ and every $t \in R$.
Therefore,
\[
    x_{[\alpha]}(t) = \varphi_4^{-1} \bigl(\varphi_4(x_{[\alpha]}(t))\bigr) = \varphi_4^{-1} \bigl(x_{[\alpha]}(\mu(t))\bigr) = x_{[\alpha]}(\nu(\mu(t))).
\]
Comparing the parameters, we obtain $\nu(\mu(t))=t$.
Similarly, $\mu(\nu(t))=t$.
Thus $\nu=\mu^{-1}$, and consequently $\mu$ is a ring automorphism.

Finally, using the relation
\[
    w_{[\widetilde{\alpha_1}]}(1)\,
    x_{[\alpha_1]}(t)\,
    w_{[\widetilde{\alpha_1}]}(1)^{-1}
    =
    x_{[\alpha_1]}(-\overline{t}),
    \qquad t\in R,
\]
and applying $\varphi_4$, we obtain
\[
    x_{[\alpha_1]}(-\overline{\mu(t)})
    =
    x_{[\alpha_1]}(-\mu(\overline{t})).
\]
Therefore, $\mu(\overline{t}) = \overline{\mu(t)}$ for all $t \in R$, that is, $\mu \circ \theta = \theta \circ \mu$.
In particular,
\[
    \mu(R_\theta) = R_\theta,
    \qquad
    \mu(R_\theta^-) = R_\theta^{-}.
\]

\subsubsection{\textbf{Normalization of $x'_{[\alpha]}(t)$ for $[\alpha] \in \Phi_\rho$ of type $A_1$}}

The normalization of the remaining root elements $x_{[\alpha]}(t)$, with $[\alpha] \in \Phi_\rho$ of type $A_1$, is now straightforward.

For every $t \in R_\theta$, the Chevalley commutator relation yields
\[
    [x_{-[\alpha_{m-1}]}(t),
    x_{[\alpha_{m-1}]+[\alpha_m]}(1/2)]
    =
    x_{[\alpha_m]}(t).
\]
Applying $\varphi_4$ to both sides and using the fact that the root elements corresponding to $-[\alpha_{m-1}]$ and $[\alpha_{m-1}]+[\alpha_m]$ are of type $A_1^2$ and hence they have already been normalized, we obtain
\[
    \varphi_4\bigl(x_{[\alpha_m]}(t)\bigr)
    =
    x_{[\alpha_m]}(\mu(t))
\]
for all $t\in R_\theta$.

Finally, every root element $x_{[\alpha]}(t)$ with $[\alpha]\in\Phi_\rho$ of type $A_1$ is obtained from $x_{[\alpha_m]}(\pm t)$ by conjugation with a product of normalized Weyl elements. Since these Weyl elements are fixed by $\varphi_4$, it follows that
\[
    \varphi_4\bigl(x_{[\alpha]}(t)\bigr)
    =
    x_{[\alpha]}(\mu(t))
\]
for every $t\in R_\theta$ and every root $[\alpha]\in\Phi_\rho$ of type $A_1$, as desired.


\subsection{The case \texorpdfstring{${}^2 A_{2m} \ (m \geq 3)$}{2A2m}}

Assume that $(\Phi,\rho)$ has twisted type ${}^2 A_{2m}$ with $m \geq 3$, and let $\Delta_\rho = \{[\alpha_1], \dots, [\alpha_m]\}$ be the simple system of $\Phi_\rho$. 
Observe that $[\alpha_1], \dots, [\alpha_{m-1}] \sim A_1^2$ are long roots, while $[\alpha_{m}] \sim A_2$ is a short root. 
We continue to use the notation $[\beta_i] \ (i = 1, \dots, m)$ for the roots of type $A_2$ introduced in the previous section. 
Likewise, we retain the notation $[\widetilde{\alpha_i}] \ (i=1,\dots,m-1)$ from the previous section.

Let $\{e_1, \dots, e_{2m+1}\}$ be the standard basis of $R^{2m+1}$. Recall that we defined
\[
    B_r := \langle e_r, e_{2m+2-r} \rangle \quad (1 \leq r \leq m), 
    \qquad 
    B_{m+1} := \langle e_{m+1} \rangle,
\]
so that
\[
    R^{2m+1} = B_1 \oplus \cdots \oplus B_m \oplus B_{m+1}.
\]

\subsubsection{\textbf{Normalization of $x'_{[\alpha]}(t)$ for $[\alpha] \in \Phi_\rho$ of type $A_1^2$}}

This case is closely analogous to the corresponding case of ${}^2A_{2m-1}$.
We first establish the normalization of the element
\[
    x_1(t) := x'_{[\alpha_1]}(t), \qquad t\in R.
\]
We claim that, with respect to the decomposition
\[
    R^{2m+1}=B_1\oplus\cdots\oplus B_m\oplus B_{m+1},
\]
the matrix $x_1(t)$ admits the block form
\[
    x_1(t)=
    \begin{pmatrix}
    A & B & 0 & \cdots & 0\\
    C & D & 0 & \cdots & 0\\
    0 & 0 & u_1 & \cdots & 0\\
    \vdots & \vdots & \vdots & \ddots & \vdots\\
    0 & 0 & 0 & \cdots & u_{m-1}
    \end{pmatrix},
\]
where $A,B,C,D,u_i$ $(i=1,\dots,m-2)$ are $2\times2$ matrices and $u_{m-1}\in R^\times$, satisfying
\begin{gather*}
    A, D, u_i \equiv I_2 \pmod{J},
    \qquad
    u_{m-1} \equiv 1 \pmod{J}, \\
    B \equiv \begin{pmatrix}
        f'(t) & 0\\
        0 & 0
    \end{pmatrix} \pmod{J},
    \qquad
    C \equiv \begin{pmatrix}
        0 & 0\\
        0 & \overline{f'(t)}
    \end{pmatrix} \pmod{J}.
\end{gather*}
Here \(f'(t)\) denotes a chosen lift of \(f(t+J)\), where \(f:k\to k\) is the residue-field automorphism induced by \(\varphi_4\).

It remains to justify the preceding block decomposition. 
If $m \geq 5$, this follows from the fact that $x_{[\alpha_1]}(t)$ commutes with $h_{[\alpha_1]}(-1)$ and with $h_{[\alpha_i]}(-1)$ for each $i=3,\dots,m-1$. 
These commuting relations are preserved projectively for their images under $\varphi_4$, and by the standard scalar-normalization argument, they lift to exact equalities in $\GL_{2m+1}(R)$.

If $m=4$, the same conclusion follows from a direct five-dimensional centralizer computation applied to $x_{\pm [\alpha_3]}(1)$ and $x_{\pm [\alpha_4]}(1, 1/2)$. 
Together with commutation with $h_{[\alpha_1]}(-1)$, this gives the stated block decomposition and forces the complementary block to be scalar.

Finally, if $m=3$, a direct three-dimensional centralizer computation gives the analogous decomposition from commutation with $h_{[\alpha_1]}(-1)$ and $x_{\pm [\alpha_3]}(1, 1/2)$.

Next, from the standard relation
\[
    h_{[\alpha_2]}(-1)\,x_{[\alpha_1]}(t)\,h_{[\alpha_2]}(-1)^{-1}
    =
    x_{[\alpha_1]}(-t),
\]
we deduce the corresponding projective identity for $X_{[\alpha_1]}(t)$. Given our choice of the representative $x_1(t)$, this relation lifts to an exact equality in $\GL_{2m+1}(R)$, namely
\[
    h_{[\alpha_2]}(-1)\,x_1(t)\,h_{[\alpha_2]}(-1)^{-1}
    =
    x_1(t)^{-1}.
\]
By analyzing the action on the blocks $B_3, \dots, B_m, B_{m+1}$ and recalling that $h_{[\alpha_2]}(-1)$ acts as $\pm I_2$ on the two-dimensional blocks $B_i$ ($i = 3, \dots, m$) and as $\pm 1$ on $B_{m+1}$, we conclude that
\[
    u_i = I_2 \quad (i = 1, \dots, m-2),
    \qquad
    u_{m-1} = 1.
\]

Let $B' := B_3 \oplus \cdots \oplus B_m \oplus B_{m+1}$.
Then, with respect to the decomposition
\[
    R^{2m+1} = B_1 \oplus B_2 \oplus B',
\]
the matrix $x_1(t)$ takes the form
\[
    x_1(t) = \begin{pmatrix}
        A & B & 0 \\
        C & D & 0 \\
        0 & 0 & I
    \end{pmatrix},
    \qquad
    \text{where } I = I_{2m-3}.
\]

Now, using the exact lifted commutativity relations with
\[
    x_{[\alpha_1]}(1),\qquad
    x_{[\alpha_1]+[\alpha_2]}(1),\qquad
    x_{-[\alpha_2]}(1),
\]
and repeating the same block computation as in the case ${}^2A_{2m-1}$, we deduce that
\[
    \varphi_4(x_{[\alpha_1]}(t)) = x_1(t) = x_{[\alpha_1]}(\mu(t)) =
    \begin{pmatrix}
        1&0&\mu(t)&0&0\\
        0&1&0&0&0\\
        0&0&1&0&0\\
        0&\overline{\mu(t)}&0&1&0\\
        0&0&0&0&I_{2m-3}
    \end{pmatrix}.
\]
Here the possible scalar factors in the lifted commutativity relations are forced to be $1$ by comparison on the common fixed summands, as in the odd-dimensional case.

Finally, every root element $x_{[\alpha]}(t)$ with $[\alpha]\in\Phi_\rho$ of type $A_1^2$ is obtained from $x_{[\alpha_1]}(\pm t)$ by conjugation with a product of normalized Weyl elements. Since these Weyl elements are fixed by $\varphi_4$ and the remaining representatives were defined by exact Weyl-conjugation identities, it follows that
\[
    \varphi_4\bigl(x_{[\alpha]}(t)\bigr) = x_{[\alpha]}(\mu(t))
\]
for every $t\in R$ and every root $[\alpha]\in\Phi_\rho$ of type $A_1^2$.

Moreover, by exactly the same argument as in the case ${}^2A_{2m-1}$, we conclude that $\mu:R\to R$ is a ring automorphism satisfying
\[
    \mu \circ \theta = \theta \circ \mu,
\]
as required.

\subsubsection{\textbf{Normalization of $x'_{[\alpha]}(t)$ for $[\alpha] \in \Phi_\rho$ of type $A_2$}}

Write $t = (t_1, t_2) \in \mathcal{A}(R)$. 
Then
\[
    t_1 \overline{t_1} = t_2 + \overline{t_2},
\]
which implies that
\[
    t_2 - \frac{t_1 \overline{t_1}}{2} \in R_\theta^-.
\]
Hence, in the group \(\mathcal{A}(R)\), we have the decomposition
\[
    t = (t_1, t_2)
      =
      \left(t_1, \frac{t_1 \overline{t_1}}{2}\right)
        \oplus
        \left(0, t_2 - \frac{t_1 \overline{t_1}}{2}\right).
\]
Consequently,
\[
    \varphi_4 \left( x_{[\alpha_m]}(t_1,t_2) \right)
    =
    \varphi_4 \left( x_{[\alpha_m]} \left(t_1, \frac{t_1 \overline{t_1}}{2}\right) \right) \, 
    \varphi_4 \left( x_{[\alpha_m]} \left(0, t_2 - \frac{t_1 \overline{t_1}}{2}\right) \right).
\]

We write
\[
    \mu(t):=(\mu(t_1),\mu(t_2)).
\]
Since \(\mu\) commutes with \(\theta\), the pair \(\mu(t)\) again belongs to \(\mathcal{A}(R)\).

To show that
\[
    \varphi_4 \left( x_{[\alpha_m]}(t) \right)
    =
    x_{[\alpha_m]}(\mu(t))
    =
    x_{[\alpha_m]}(\mu(t_1), \mu(t_2)),
\]
it suffices to prove that
\[
    \varphi_4 \left( x_{[\alpha_m]} \left(t_1, \frac{t_1 \overline{t_1}}{2}\right) \right)
    =
    x_{[\alpha_m]}\left(\mu(t_1), \frac{\mu(t_1) \overline{\mu(t_1)}}{2}\right)
\]
and
\[
    \varphi_4 \left( x_{[\alpha_m]} \left(0, t_2 - \frac{t_1 \overline{t_1}}{2}\right) \right)
    =
    x_{[\alpha_m]}\left(0, \mu(t_2) - \frac{\mu(t_1) \overline{\mu(t_1)}}{2}\right).
\]

The first equality follows from the Chevalley commutator relation
\[
    [x_{-[\alpha_{m-1}]}(t_1), x_{[\alpha_{m-1}]+[\alpha_m]}(1,1/2)]
    =
    x_{[\alpha_m]}\left(t_1, \frac{t_1 \overline{t_1}}{2}\right)
    x_{[\alpha_{m-1}]+2[\alpha_m]}(-t_1/2).
\]
After applying $\varphi_4$, both the transformed commutator formula and the image of its right-hand side contain the factor $x_{[\alpha_{m-1}]+2[\alpha_m]}(-\mu(t_1)/2)$, since the $A_1^2$-root images are known, $x_{[\alpha_{m-1}]+[\alpha_m]}(1,1/2)$ is fixed, and $\mu(-t_1/2)=-\mu(t_1)/2$. Cancelling this common factor gives
\[
    \varphi_4 \left( x_{[\alpha_m]} \left(t_1, \frac{t_1 \overline{t_1}}{2}\right) \right)
    =
    x_{[\alpha_m]}\left(\mu(t_1), \frac{\mu(t_1)\overline{\mu(t_1)}}{2}\right).
\]

The second equality follows similarly from the commutator relation
\[
    [x_{-[\alpha_{m-1}]}(u), x_{[\alpha_{m-1}]+2[\alpha_m]}(1/2)]
    =
    x_{[\alpha_m]}(0,u),
    \qquad
    \text{where }
    u := t_2 - \frac{t_1 \overline{t_1}}{2} \in R_\theta^-.
\]
Applying $\varphi_4$ and using the normalization already established for the roots of type $A_1^2$, we obtain
\[
    \varphi_4 \left( x_{[\alpha_m]} (0,u) \right)
    =
    x_{[\alpha_m]}(0,\mu(u)).
\]
Since $\mu$ is additive, multiplicative, and commutes with $\theta$, we have
\[
    \mu(u)
    =
    \mu(t_2) - \frac{\mu(t_1)\overline{\mu(t_1)}}{2}.
\]
Hence
\[
    \varphi_4 \left( x_{[\alpha_m]} \left(0, t_2 - \frac{t_1 \overline{t_1}}{2}\right) \right)
    =
    x_{[\alpha_m]}\left(0, \mu(t_2) - \frac{\mu(t_1) \overline{\mu(t_1)}}{2}\right).
\]

Combining the two equalities above, we conclude that
\[
    \varphi_4 \left( x_{[\alpha_m]}(t_1,t_2) \right)
    =
    x_{[\alpha_m]}(\mu(t_1),\mu(t_2)).
\]

Finally, every root element $x_{[\alpha]}(t)$ with $[\alpha]\in\Phi_\rho$ of type $A_2$ is obtained from $x_{[\alpha_m]}(t)^{\pm1}$ by conjugation with a product of normalized Weyl elements. Since $\mu$ commutes with $\theta$ and these Weyl elements are fixed by $\varphi_4$, it follows that
\[
    \varphi_4 \bigl(x_{[\alpha]}(t)\bigr)
    =
    x_{[\alpha]}(\mu(t))
\]
for every $t \in \mathcal{A}(R)$ and every root $[\alpha] \in \Phi_\rho$ of type $A_2$, as desired. This completes Step~(C).


\section{Proof of Theorem~\ref{MT_local2} and Theorem~\ref{MT_local3}} \label{sec:proof_of_MT2_and_MT3}


In this section, we first prove Theorem~\ref{MT_local2}. We then use this result to establish Theorem~\ref{MT_local3}.

\begin{proof}[Proof of Theorem~\ref{MT_local2}]
    Let $\varphi \colon E'_{\pi,\sigma}(\Phi,R) \to E'_{\pi,\sigma}(\Phi,R)$ be a given automorphism.  
    Lemma~\ref{lemma:E_ad=E_pi/Z} implies that $\varphi$ induces an automorphism
    \[
        \overline{\varphi} \colon E'_{\mathrm{ad},\sigma}(\Phi,R) \longrightarrow E'_{\mathrm{ad},\sigma}(\Phi,R)
    \]
    of the adjoint elementary twisted Chevalley group. 
    
    By Theorem~\ref{MT_local1}, the automorphism $\overline{\varphi}$ can be expressed as 
    \[
        \overline{\varphi} = i_g \circ \mu,
    \]
    where $i_g$ is the inner automorphism induced by an element $g \in G_{\mathrm{ad},\sigma}(\Phi,R)$, and $\mu$ is a ring automorphism.
    
    The ring automorphism $\mu$ of $R$ naturally induces a ring automorphism of $E'_{\pi,\sigma}(\Phi,R)$, which we also denote by $\mu$.  
    The remaining task is to lift the inner automorphism $i_g$ to an automorphism of $E'_{\pi,\sigma}(\Phi,R)$.

    \smallskip

    Since $R$ is a local ring, the group $G_{\mathrm{ad},\sigma}(\Phi,R)$ admits the decomposition
    \[
        G_{\mathrm{ad},\sigma}(\Phi,R) = T_{\mathrm{ad},\sigma}(\Phi,R) E'_{\mathrm{ad},\sigma}(\Phi,R).
    \]
    Therefore, every $g \in G_{\mathrm{ad},\sigma}(\Phi,R)$ can be factored as
    \[
        g = h(\chi) X,
    \]
    where $\chi \in \Hom_{1}(\Lambda_{r}, R^{\times})$ and $X \in E'_{\mathrm{ad}, \sigma}(\Phi, R)$.  
    By Lemma~\ref{lemma:E_ad=E_pi/Z}, there exists $X' \in E'_{\pi, \sigma}(\Phi, R)$ such that $\delta'_\pi(X')=X$. 
    
    To lift the torus element $h(\chi)$, we must extend the character $\chi$ of $\Lambda_r$ to a self-conjugate character $\chi'$ of $\Lambda_{\pi}$ over a suitable ring extension of $R$.

    Let $R[\Lambda_{r}]$ and $R[\Lambda_{\pi}]$ be the corresponding group algebras, with basis elements denoted by $e^{\lambda}$. The character $\chi$ defines an $R$-algebra homomorphism
    \[
        \varepsilon_{\chi} \colon R[\Lambda_{r}] \longrightarrow R
        \quad \text{by} \quad
        \varepsilon_{\chi}(e^{\alpha}) = \chi(\alpha), \quad \alpha \in \Lambda_{r}.
    \]
    Define 
    \[
        S \coloneqq R[\Lambda_{\pi}] \otimes_{R[\Lambda_{r}]} R,
    \]
    where $R$ is regarded as an $R[\Lambda_{r}]$-algebra via $\varepsilon_{\chi}$.

    Let $\mathcal{D}$ be a set of representatives for the finite quotient $\Lambda_{\pi}/\Lambda_{r}$ containing $0$. Then
    \[
        R[\Lambda_{\pi}] = \bigoplus_{\lambda \in \mathcal{D}} R[\Lambda_{r}]e^{\lambda}
    \]
    as an $R[\Lambda_{r}]$-module. Consequently,
    \[
        S \cong \bigoplus_{\lambda \in \mathcal{D}} R\bigl(e^{\lambda} \otimes 1\bigr)
    \]
    as an $R$-module. In particular, the canonical homomorphism $R \to S$ given by $r \mapsto 1 \otimes r$ is injective, allowing us to identify $R$ with its image in $S$.

    Define automorphisms of the group algebras by
    \[
        \tau_{r} \Biggl( \sum_{\alpha \in \Lambda_{r}} r_{\alpha}e^{\alpha} \Biggr) = \sum_{\alpha \in \Lambda_{r}} \theta(r_{\alpha})e^{\rho(\alpha)}
        \quad \text{and} \quad
        \tau_{\pi} \Biggl( \sum_{\lambda \in \Lambda_{\pi}} r_{\lambda}e^{\lambda} \Biggr) = \sum_{\lambda \in \Lambda_{\pi}} \theta(r_{\lambda})e^{\rho(\lambda)}.
    \]
    Since $\chi$ is self-conjugate, we have $\varepsilon_{\chi} \circ \tau_{r} = \theta \circ \varepsilon_{\chi}$.
        
    Thus, the formula
    \[
        \theta_{S}(b \otimes r) \coloneqq \tau_{\pi}(b) \otimes \theta(r)
    \]
    defines an involution on $S$ that extends $\theta$.

    For every $\lambda \in \Lambda_{\pi}$, put
    \[
        \chi'(\lambda) \coloneqq e^{\lambda} \otimes 1.
    \]
    Then $\chi'(\lambda)$ is invertible with inverse $\chi'(-\lambda)$, meaning
    \[
        \chi' \in \Hom(\Lambda_{\pi}, S^{\times}).
    \]
    For every $\alpha \in \Lambda_{r}$, the defining relations of the tensor product yield
    \[
        \chi'(\alpha) = e^{\alpha} \otimes 1 = 1 \otimes \chi(\alpha) = \chi(\alpha).
    \]
    Thus, $\chi'\big|_{\Lambda_{r}} = \chi$.
    Furthermore,
    \[
        \theta_{S}\bigl(\chi'(\lambda)\bigr) = e^{\rho(\lambda)} \otimes 1 = \chi'\bigl(\rho(\lambda)\bigr).
    \]
    Hence, $\chi' \in \Hom_{1}(\Lambda_{\pi}, S^{\times})$.

    Define the lifted torus element
    \[
        h' \coloneqq h(\chi') \in T_{\pi,\sigma}(\Phi,S).
    \]
    For every $[\alpha] \in \Phi_{\rho}$ and $t \in R_{[\alpha]}$, the action of the torus on the standard root generators gives
    \[
        h' x_{[\alpha]}(t) (h')^{-1} = x_{[\alpha]}\bigl(\chi(\alpha) \cdot t\bigr).
    \]
    Since $\chi(\alpha) \in R^{\times}$ and $\chi$ is self-conjugate, we have $\chi(\alpha) \cdot t \in R_{[\alpha]}$. 
    
    Applying the same formula to $(h')^{-1}$, we obtain
    \[
        h' E'_{\pi,\sigma}(\Phi,R) (h')^{-1} = E'_{\pi,\sigma}(\Phi,R).
    \]
    Moreover, $h'$ centralizes $T_{\pi,\sigma}(\Phi,R)$. Since $G_{\pi,\sigma}(\Phi,R) = T_{\pi,\sigma}(\Phi,R) E'_{\pi,\sigma}(\Phi,R)$, it follows that
    \[
        h' G_{\pi,\sigma}(\Phi,R) (h')^{-1} = G_{\pi,\sigma}(\Phi,R).
    \]

    Now, define
    \[
        g' \coloneqq h' X' \in G_{\pi,\sigma}(\Phi,S).
    \]
    Since $X' \in E'_{\pi,\sigma}(\Phi,R) \subseteq G_{\pi,\sigma}(\Phi,R)$, we have
    \[
        g' G_{\pi,\sigma}(\Phi,R) (g')^{-1} = G_{\pi,\sigma}(\Phi,R).
    \]
    Thus, conjugation by $g'$ restricts to a well-defined inner automorphism $i_{g'}$ of $E'_{\pi,\sigma}(\Phi,R)$. 
    By construction, $\delta'_\pi\circ i_{g'}=i_g\circ\delta'_\pi$; hence $i_{g'}$ induces $i_g$ on the adjoint quotient $E'_{\mathrm{ad},\sigma}(\Phi,R)$.

    \smallskip

    Next, consider the automorphism
    \[
        \varphi_{1} \coloneqq \mu^{-1} \circ i_{g'}^{-1} \circ \varphi \;\in\; \Aut\bigl(E'_{\pi,\sigma}(\Phi,R)\bigr).
    \]
    By construction, the automorphism induced by $\varphi_{1}$ on $E'_{\mathrm{ad},\sigma}(\Phi,R)$ is the identity. Thus, for every $x\in E'_{\pi,\sigma}(\Phi,R)$, the element $z(x):=x^{-1}\varphi_1(x)$ lies in $Z(E'_{\pi,\sigma}(\Phi,R))$. 
    Since $z(xy)=z(x)z(y)$, the map $z:E'_{\pi,\sigma}(\Phi,R)\to Z(E'_{\pi,\sigma}(\Phi,R))$ is a homomorphism; because $E'_{\pi,\sigma}(\Phi,R)$ is perfect, $z$ is trivial.  
    Hence, $\varphi_{1} = \mathrm{id}$.

    \smallskip

    Therefore, the original automorphism $\varphi$ decomposes exactly as
    \[
        \varphi = i_{g'} \circ \mu,
    \]
    where $g' \in G_{\pi,\sigma}(\Phi,S)$ normalizes $G_{\pi,\sigma}(\Phi,R)$, and $\mu$ is a ring automorphism commuting with $\theta$.  
    This completes the proof.
\end{proof}

The preceding argument in fact yields a stronger conclusion than Theorem~\ref{MT_local2}. Specifically, the element $g' \in G_{\pi,\sigma}(\Phi,S)$ defining the inner automorphism $i_{g'}$ can be chosen from the subgroup
\[
    T_{\pi,\sigma}(\Phi,S)\,E'_{\pi,\sigma}(\Phi,R).
\]
Furthermore, if we decompose $g' = h(\chi)X$ where $\chi \in \operatorname{Hom}_{1}(\Lambda_\pi, S^{\times})$ and $X \in E'_{\pi,\sigma}(\Phi,R)$, then the restriction $\chi|_{\Lambda_r}$ necessarily belongs to $\operatorname{Hom}_{1}(\Lambda_r, R^{\times})$.

\begin{defn}[Diagonal automorphisms]
    Let $\chi \in \operatorname{Hom}_1(\Lambda_r, R^{\times})$ be a character of the root lattice $\Lambda_r$. Consider a ring extension $S$ of $R$ satisfying the following conditions:
    \begin{enumerate}
        \item[(a)] the involution $\theta$ on $R$ extends to an involution on $S$;
        \item[(b)] there exists a character $\chi' \in \operatorname{Hom}_1(\Lambda_\pi, S^{\times})$ such that $\chi'|_{\Lambda_r} = \chi$.
    \end{enumerate}
    (Note that such a ring extension $S$ always exists.) 
    Let $h:= h(\chi') \in T_{\pi, \sigma}(\Phi, S)$ be the corresponding torus element.
    If $h \, G_{\pi, \sigma}(\Phi, R) \, h^{-1} = G_{\pi, \sigma}(\Phi, R)$, then the conjugation map $i_h \colon x \mapsto h x h^{-1}$ is an automorphism of $G_{\pi, \sigma}(\Phi, R)$, called a \emph{diagonal automorphism}.
\end{defn}

It follows immediately from this construction that for groups of adjoint type, every diagonal automorphism is strictly inner. 

The preceding discussion establishes the following corollary to the proof of Theorem~\ref{MT_local2}.

\begin{cor}\label{cor:inner as strictly inner}
    Under the hypotheses of Theorem~\ref{MT_local2}, every automorphism of $E'_{\pi, \sigma}(\Phi, R)$ factors as a product of strictly inner, diagonal, and ring automorphisms.
\end{cor}

We now proceed to the proof of Theorem~\ref{MT_local3}.

\begin{proof}[Proof of Theorem~\ref{MT_local3}]
    Let $\varphi : G_{\pi,\sigma}(\Phi,R) \to G_{\pi,\sigma}(\Phi,R)$ be an automorphism.  
    Since $E'_{\pi,\sigma}(\Phi,R)$ is a characteristic subgroup of $G_{\pi,\sigma}(\Phi,R)$, the restriction $\varphi|_{E'_{\pi,\sigma}(\Phi,R)}$ is an automorphism of $E'_{\pi,\sigma}(\Phi,R)$.  
    By Theorem~\ref{MT_local2}, we may write
    \[
        \varphi|_{E'_{\pi,\sigma}(\Phi,R)} \;=\; i_g \circ \mu,
    \]
    where $i_g$ is an inner automorphism induced by some element $g \in G_{\pi,\sigma}(\Phi,S)$, and $\mu$ is induced by a ring automorphism of $R$.  
    Clearly, $\mu$ extends naturally to an automorphism of the whole group $G_{\pi,\sigma}(\Phi,R)$.

    It remains to verify that $i_g$ extends to an automorphism of $G_{\pi,\sigma}(\Phi, R)$.  
    Since $R$ is a local ring, we have the standard decomposition
    \[
        G_{\pi,\sigma}(\Phi, R) = T_{\pi,\sigma}(\Phi, R)\, E'_{\pi,\sigma}(\Phi, R).
    \]
    We already know that $i_g$ stabilizes $E'_{\pi,\sigma}(\Phi, R)$. 
    We claim that it also satisfies
    \[
        g T_{\pi,\sigma}(\Phi, R) g^{-1} \subseteq G_{\pi,\sigma}(\Phi, R).
    \]
    
    By the construction in the proof of Theorem~\ref{MT_local2}, we may choose $g$ in the form
    \[
        g = h(\chi) X,
    \]
    where $\chi \in \operatorname{Hom}_1(\Lambda_\pi, S^{\times})$ such that $\chi|_{\Lambda_r} \in \operatorname{Hom}_{1}(\Lambda_r, R^{\times})$, and $X \in E'_{\pi,\sigma}(\Phi,R)$. 
    
    For any $h(\chi') \in T_{\pi,\sigma}(\Phi,R)$, we have
    \[
        g h(\chi') g^{-1}
        = h(\chi) X h(\chi') X^{-1} h(\chi)^{-1}
        = Y h(\chi') Y^{-1},
    \]
    where $Y \coloneqq h(\chi) X h(\chi)^{-1}$.  
    Because $\chi|_{\Lambda_r}$ takes values in $R^{\times}$, conjugation by $h(\chi)$ preserves $E'_{\pi,\sigma}(\Phi,R)$; hence, $Y \in E'_{\pi,\sigma}(\Phi,R)$.  
    Therefore,
    \[
        g h(\chi') g^{-1} \in E'_{\pi,\sigma}(\Phi,R) T_{\pi,\sigma}(\Phi,R) = G_{\pi,\sigma}(\Phi,R).
    \]
    This proves our claim. 

    Combining this with the fact that $E'_{\pi,\sigma}(\Phi,R)$ is invariant under $i_g$, we deduce that
    \[
        g \, G_{\pi,\sigma}(\Phi,R) \, g^{-1} \subseteq G_{\pi,\sigma}(\Phi,R).
    \]
    Thus, $i_g$ is a well-defined homomorphism on $G_{\pi, \sigma}(\Phi, R)$.
    Since it is given by conjugation, $i_g$ is trivially injective. 
    
    We now show that $i_g$ is surjective. Let $y \in G_{\pi, \sigma}(\Phi, R)$. We seek an element $x \in G_{\pi, \sigma}(\Phi, R)$ such that $i_g(x) = y$.

    Define $b \coloneqq i_g(y) y^{-1}$. 
    We claim that $b \in E'_{\pi, \sigma}(\Phi, R)$. 
    Indeed, $h(\chi)$ centralizes the torus and $X\in E'_{\pi,\sigma}(\Phi,R)$, so $i_g$ induces the identity on $G_{\pi,\sigma}(\Phi,R)/E'_{\pi,\sigma}(\Phi,R)$. Consequently, $i_g(y)y^{-1}\in E'_{\pi,\sigma}(\Phi,R)$.

    Because the restriction of $i_g$ to $E'_{\pi,\sigma}(\Phi,R)$ is an automorphism, it is surjective on this subgroup. Hence, there exists an element $a \in E'_{\pi,\sigma}(\Phi,R)$ such that $i_g(a) = b^{-1}$.
    Consequently,
    \[
        i_g(ay) = i_g(a) i_g(y) = b^{-1} i_g(y) = y.
    \]
    By setting $x = ay \in G_{\pi, \sigma}(\Phi, R)$, we conclude that $i_g$ is surjective, as desired.

    \smallskip

    Finally, set
    \[
        \varphi_{1} \;=\; \mu^{-1} \circ i_g^{-1} \circ \varphi.
    \]
    Then $\varphi_1$ is an automorphism of $G_{\pi, \sigma}(\Phi, R)$ such that $\varphi_1 \mid_{E'_{\pi, \sigma}(\Phi, R)}$ is the identity map. Since $E'_{\pi,\sigma}(\Phi, R)$ is a normal subgroup of $G_{\pi, \sigma}(\Phi, R)$, for any $x \in E'_{\pi, \sigma}(\Phi, R)$ and $a \in G_{\pi, \sigma}(\Phi, R)$, there exists $y \in E'_{\pi, \sigma}(\Phi, R)$ such that 
    \[
        a x a^{-1} = y.
    \]
    Applying $\varphi_1$ to both sides, we obtain
    \[
        \varphi_1(a x a^{-1}) = \varphi_1(y) = y = a x a^{-1},
    \]
    which implies
    \[
        \varphi_1(a) x \varphi_1(a^{-1}) = a x a^{-1}.
    \]
    Therefore, we have
    \[
        a^{-1} \varphi_1(a) x = x a^{-1} \varphi_1(a).
    \]
    This holds for every $x \in E'_{\pi, \sigma}(\Phi, R)$, so $a^{-1} \varphi_1(a) \in C_{G_{\pi, \sigma}(\Phi, R)}(E'_{\pi, \sigma}(\Phi, R))$, the centralizer of $E'_{\pi, \sigma}(\Phi, R)$ in $G_{\pi, \sigma}(\Phi, R)$. 
    By \cite{KS3}, this centralizer coincides with the center $Z(G_{\pi, \sigma}(\Phi, R))$. 
    Hence,
    \[
        a^{-1}\varphi_1(a) \in Z\bigl(G_{\pi,\sigma}(\Phi,R)\bigr) \qquad \text{for all } a \in G_{\pi,\sigma}(\Phi,R).
    \]

    Define the map $\tau \colon G_{\pi,\sigma}(\Phi,R) \to Z\bigl(G_{\pi,\sigma}(\Phi,R)\bigr)$ by $\tau(a) \coloneqq a^{-1}\varphi_1(a)$. 
    It follows that $\varphi_1(a) = a \tau(a) = \tau(a) a$.

    We claim that $\tau$ is a group homomorphism. Indeed, for all $a,b \in G_{\pi,\sigma}(\Phi,R)$, we have
    \[
        \tau(ab) = (ab)^{-1}\varphi_1(ab) = b^{-1}a^{-1}\varphi_1(a)\varphi_1(b) = b^{-1}\tau(a)\varphi_1(b) = \tau(a)b^{-1}\varphi_1(b) = \tau(a)\tau(b),
    \]
    where the penultimate equality uses the fact that $\tau(a)$ is central.

    Thus, $\tau$ is a group homomorphism, meaning that $\varphi_1$ is a central automorphism of $G_{\pi,\sigma}(\Phi,R)$. Since $\varphi_1 = \mu^{-1} \circ i_g^{-1} \circ \varphi$, we finally obtain
    \[
        \varphi = i_g \circ \mu \circ \varphi_1.
    \]
    Therefore, $\varphi$ is the composition of an inner automorphism, a ring automorphism, and a central automorphism. This completes the proof.
\end{proof}

The proof of Theorem~\ref{MT_local3}, together with Corollary~\ref{cor:inner as strictly inner}, yields the following result.

\begin{cor}
    Under the hypotheses of Theorem~\ref{MT_local3}, every automorphism of $G_{\pi, \sigma}(\Phi, R)$ factors as a product of strictly inner, diagonal, ring, and central automorphisms.
\end{cor}


\section*{Acknowledgements}

The authors are deeply grateful to the late Professor Eugene Plotkin, who took part in discussions of this project and generously shared his insights, comments, and encouragement. This paper is dedicated to his memory.

The authors also thank Shripad M.~Garge, Dipendra Prasad, and Pavel Gvozdevsky for valuable discussions and helpful comments on various aspects of this work.



\end{document}